\documentclass[12pt, a4paper]{amsart}

\usepackage{amsmath, amsfonts, amsthm, amssymb, xcolor, booktabs,
  array, placeins, xfrac, multirow, chngpage}

\newif\ifpdfAuthoring \pdfAuthoringtrue

\ifpdfAuthoring \usepackage[ bookmarksnumbered=true,
bookmarksopen=false, hypertexnames=false, breaklinks=true,
colorlinks=true, pagebackref=true, hyperindex=true, plainpages=false,
pdfauthor={Valery Gritsenko, Nils-Peter Skoruppa and Don Zagier},
pdftitle={Theta Blocks}, pdfsubject={Preprint, Siegen 2017},
pdfkeywords={Jacobi forms},
]{hyperref} \usepackage{ae} \fi

\newcommand {\BF}{{\mathfrak W}} \newcommand {\BP} {{\mathbb B}}
\newcommand {\C} {{\mathbb C}} \newcommand {\Div} {\sym{Div}}
\newcommand {\GTB} {\TB^*} \newcommand {\HP} {{\mathbb H}} \newcommand
{\Lvec} {shadow vector} \newcommand {\Mp}[1][\Z] {\sym{Mp}(2,{#1})}
\newcommand {\N} {{\mathbb N}} \newcommand {\Orth}
{\mathop{\null\mathrm {O}}\nolimits} \newcommand {\Q} {{\mathbb Q}}
\newcommand {\R} {{\mathbb R}} \newcommand {\SL}[1][\Z]
{\sym{SL}(2,{#1})}  \newcommand
{\TB} {{\mathfrak B}} \newcommand {\TQ} {G(\TB)}
\newcommand {\Z} {{\mathbb Z}} \newcommand {\dual}[1]
{{#1}^\sharp} \newcommand {\ev}[1] {{#1}_{\rm ev}} \newcommand {\ex}
{{\mathbf e}} \newcommand {\e}[1] {e\left(#1\right)} \newcommand
{\isom} {\cong} \newcommand {\lam} {\lambda} \newcommand {\lat}[1]
{\underline{#1}}
\newcommand {\leg}[2] {\left(\frac{#1}{#2}\right)}
\newcommand {\tleg}[2] {\left(\tfrac{#1}{#2}\right)}
\newcommand {\mat}[4]
{\left(\begin{smallmatrix}#1&#2\\#3&#4\end{smallmatrix}\right)}
\newcommand {\ord} {\sym{ord}} 
\newcommand {\red}[1] {\widetilde #1} \newcommand {\remout}[1]
{\relax}  \newcommand
{\sectheta} {\vartheta^*} \newcommand {\spv}[1] {{#1}^{\bullet}}
\newcommand {\sym}[1] {\operatorname{#1}} \newcommand {\ve}[1]
{\mathbf {#1}} \newcommand {\vt}[1] {\vartheta_{#1}}
\newcommand {\Th} {Th}

\usepackage[breakable]{tcolorbox}

\theoremstyle{plain}
\newtheorem{Theorem}{Theorem}[section]
\newtheorem{Conjecture}[Theorem]{Conjecture}
\newtheorem*{Fact}{Fact}
\newtheorem{Proposition}[Theorem]{Proposition}
\newtheorem{Corollary}[Theorem]{Corollary}
\newtheorem{Lemma}[Theorem]{Lemma}
\newtheorem{Supplement}[Theorem]{Supplement}
\newtheorem*{TheoremX}{Theorem}

\theoremstyle{definition} \newtheorem*{Definition}{Definition}
\newtheorem{Example}[Theorem]{Example}
\newtheorem*{Remark}{Remark}

\makeatletter

\def\part{\@startsection{part}{0}%
  \z@{1.7\linespacing\@plus\linespacing}{1.5\linespacing}%
  {\normalfont\scshape\centering}}

\renewcommand{\@secnumfont}{\bfseries}

\renewcommand{\@bibtitlestyle}{%
  \@xp\part\@xp*\@xp{\refname}%
}

\makeatother

\title{%
  Theta Blocks }

\author{%
  Valery Gritsenko}
\address{%
  Universit\'e de Lille, Villeneuve d'Ascq, France and
  National Research University ``Higher School of Economics'', Moscow, Russian Federation}
\email{%
  Valery.Gritsenko@math.univ-lille1.fr}
\thanks{The first author was
  supported by the Laboratory of Mirror Symmetry of the National
  Research University ``Higher School of Economics'' (Russian Federation
  Government grant, agreement no.~14.641.31.0001)}

\author{%
  Nils-Peter Skoruppa}
\address{%
  Department Mathematik, Universit\"at Siegen, 57068 Siegen, Germany}
\email{%
  nils.skoruppa@gmail.com}

\author{%
  Don Zagier} \address{%
  Max-Planck-Institut f\"ur Mathematik, Vivatsgasse 7, 53111 Bonn, Germany}
\email{%
  dbz@mpim-bonn.mpg.de}
  


\begin{document}

\begin{abstract}
  We define theta blocks as products of Jacobi theta functions divided
  by powers of the Dedekind eta-function and show that they give a
  powerful new method to construct Jacobi forms and Siegel modular
  forms, with applications also in lattice theory and algebraic
  geometry. One of the central questions is when a theta block defines
  a Jacobi form. It turns out that this seemingly simple question is
  connected to various deep problems in different fields ranging from
  Fourier analysis over infinite-dimensional Lie algebras to the
  theory of moduli spaces in algebraic geometry. We give several
  answers to this question.
\end{abstract}

\vspace*{-2.5cm}
\maketitle

\tableofcontents
\addtocontents{toc}{\protect\setcounter{tocdepth}{1}}

\part*{Introduction}
\label{sec:introduction}

The Jacobi theta function $\vartheta(\tau,z)$, defined for
$\tau\in\HP$, $z\in\C$ either as the theta series
\begin{equation}
  \label{eq:theta}
  \vartheta(\tau,z)=\sum_{r=-\infty}^\infty \leg{-4}r
  q^{r^2/8}\,\zeta^{r/2} \qquad \bigl(q=e^{2\pi i\tau},\;\zeta=e^{2\pi
    iz}\bigr)
\end{equation}
or else by the triple product
\begin{equation}
  \label{eq:triple-product}
  \vartheta(\tau,z) = q^{1/8}\,\zeta^{1/2}\,
  \prod_{n=1}^\infty\bigl(1-q^n\bigr)
  \bigl(1-q^n\zeta\bigr)\bigl(1-q^{n-1}\zeta^{-1}\bigr)
  ,
\end{equation}
is a holomorphic Jacobi form (with non-trivial character) of weight
1/2 and index 1/2.  (The definitions of holomorphic Jacobi forms with
character and of their weight and index are reviewed in \S2.)  For
$a\in\N$ we denote by $\vartheta_a$ the Jacobi form
\begin{equation*}
  \vartheta_a(\tau,z)\;:=\;\vartheta(\tau,az)
\end{equation*}
of weight 1/2 and index $a^2/2$, while
\begin{equation*}
  \eta(\tau)
  =
  \sum_{r=1}^\infty \leg{12}rq^{r^2/24}
  =
  q^{1/24}\prod_{n=1}^\infty \bigl(1-q^n\bigr)
\end{equation*}
denotes the Dedekind eta-function. The starting point of this paper is
the following observation:
\begin{Fact}
  Let $a$ and $b$ be positive integers. Then the quotient
  \begin{equation*}
    Q_{a,b}(\tau,z)
    =
    \frac{\vartheta_a(\tau,z)\,\vartheta_b(\tau,z)\,
      \vartheta_{a+b}(\tau,z)}{\eta(\tau)}
  \end{equation*}
  is a holomorphic Jacobi form of weight $1$ and index $a^2+ab+b^2$,
  and is a cusp form if $3g^3\mid ab(a+b)$, where $g=\gcd(a,b)$.
\end{Fact}

We will give several proofs and generalizations of this result.  To do
this, we first give (in \S3) a general criterion for the divisibility
of a holomorphic Jacobi form, and in particular of a product of
$\vartheta_a$'s, by a given power of $\eta$.  This will involve
defining the notion of the {\em order} of a Jacobi form at infinity, a
notion which has apparently not previously been introduced but which
seems quite fundamental to the theory.  This criterion will then be
used to prove the holomorphy of $Q_{a,b}$ and to give many other
examples---both infinite families proved theoretically and sporadic
examples found by computer---of theta products divisible by high
powers of $\eta$.  A typical example is the family of holomorphic
Jacobi forms of weight~2
\begin{equation}
  \label{eq:R-definition}
  R_{a,b,c,d} =
  \frac{\vartheta_a\,\vartheta_b\,\vartheta_c\,\vartheta_d\,
    \vartheta_{a+b}\,\vartheta_{b+c}\,\vartheta_{c+d}\,\vartheta_{a+b+c}\,
    \vartheta_{b+c+d}\,\vartheta_{a+b+c+d}}{\eta^6}
\end{equation}
where $a$, $b$, $c$ and $d$ are natural numbers.  In many cases,
including both the families $Q_{a,b}$ and $R_{a,b,c,d}$, we will also
give explicit formulas for quotients of the form
$\eta^{-s}\vartheta_{a_1}\cdots\vartheta_{a_N}$ as theta series of
rank $N-~s$. Some of these are obtained by using a general criterion
(described in~\S3) for the divisibility of one theta series by
another, while others arise by specializing the Macdonald identities
(also known as Kac-Weyl denominator formulas) for suitable root
systems.

A weakly holomorphic Jacobi form of the form
$\vartheta_{a_1}\cdots\vartheta_{a_N}/\eta^d$ is called a theta block
of length~$N$, and it is called a holomorphic theta block if it is a
Jacobi form. Its weight is equal to $(N-d)/2$, and one of the
principal aims of this article is to construct explicit examples of
holomorphic theta blocks whose weight is relatively small with respect
to the length. For instance, the Jacobi form $R_{a,b,c,d}$ has
length~$10$ and weight~$2$, and more generally in
Section~\ref{sec:A4-family} we will construct families of length
$n(n+1)/2$ and weight~$n/2$. In Section~\ref{sec:families} we will
develop a general theory for constructing such families and will see
many more concrete examples in the sections following it. We will be
interested both in theoretical bounds for the minimal weight $k$ for
given length~$N$ (in Section~\ref{sec:long-and-low} it is shown that
the minimal weight is bounded below and above by $c_1\log N$ and
$c_2(\log N)^3$ for positive constants~$c_i$) and in constructing
  explicit holomorphic theta blocks of small weight.

The special families that we
  construct turn out to give a very useful way of constructing
Jacobi forms, especially Jacobi forms of low weight.  For instance,
both the first Jacobi form and the first Jacobi cusp form of weight 2
and trivial character, which have index 25 and 37, respectively and
were constructed with some effort in~\cite{Eichler-Zagier}, are now
obtained immediately as the two first cases $(a,b,c,d)=(1,1,1,1)$ and
$(1,1,1,2)$ of the family $R_{a,b,c,d}$, and many other interesting
examples of Jacobi forms of low weight and given character can be
obtained as special cases of products of the functions $Q_{a,b}$ or of
the other families.  Such forms have several applications, e.g.~to
questions concerning the classification of moduli spaces of polarized
abelian surfaces or of K3-surfaces.  We will describe these
applications and give some general discussion of the situation for
small weight.  In particular, we shall see that all holomorphic Jacobi
forms of weight 1/2 and weight 1 and arbitrary character can be
obtained as theta quotients
$\eta^{-s}\vartheta_{a_1}^{\pm}\cdots\vartheta_{a_N}^{\pm1}$, and will
give conjectures and partial results for higher weights.  We expect
that the spaces of Jacobi forms of small weight and arbitrary index
and character on the full modular group are in fact spanned by theta
quotients. As we shall see in Section~\ref{sec:generalities} this
statement is, however, false for large weights.

We finally mention a side result of our studies of theta blocks,
namely a rather short proof
(in~\S\ref{sec:Examples-constructed-from-root-systems}) of the Macdonald
identities based on Jacobi forms of lattice index.

\part*{Part I: Basic Theory}

\section{Review of Jacobi forms}
\label{sec:review}

We first recall the definition of Jacobi forms as given
in~\cite{Eichler-Zagier}.  Let $k$ and $m$ be be non-negative
integers.  Then a {\em holomorphic Jacobi form of weight $k$ and index
  $m$} (on the full modular group $\Gamma=\text{SL}(2,\Z)$, or more
precisely on the full Jacobi group
$\Gamma^{\text J}=\Gamma\times\Z^2$) is a holomorphic function
$\phi:\HP\times\C\rightarrow\C$ which satisfies the two transformation
equations
\begin{equation}
  \label{eq:G-inv}
  \phi\bigl(\frac{a\tau+b}{c\tau+d},\,\frac z{c\tau+d}\bigr)
  =(c\tau+d)^k\,\ex\bigl(\frac{mcz^2}{c\tau+d}\bigr)\,\phi(\tau,z)
\end{equation}
and
\begin{equation}
  \label{eq:H-inv}
  \phi\bigl(\tau,\,z+\lam\tau+\mu\bigr)
  =
  \ex\bigl(-m(\lam^2\tau+2\lam z)\bigr)\,\phi(\tau,z)
\end{equation}
for all $\tau\in\HP$, $z\in\C$, $\mat abcd\in\Gamma$ and
$\bigl(\smallmatrix\lam\\ \mu\endsmallmatrix\bigr)\in\Z^2$ (here
$\ex(x)=e^{2\pi ix}$ as usual) and which has a Fourier expansion of
the form
\begin{equation}
  \label{eq-fouexp}
  \phi(\tau,z)
  =
  \sum_{
    \begin{subarray}{c}n\in\Z\\ n\ge0
    \end{subarray}
  }
  \sum_{
    \begin{subarray}{c}r\in\Z\\ r^2\le4mn
    \end{subarray}
  }
  c(n,r)\,q^n\,\zeta^r\,,
\end{equation}
where $q$ and $\zeta$ denote $\ex(\tau)$ and $\ex(z)$, respectively.
The Fourier coefficients $c(n,r)$ then automatically satisfy the
periodicity property
\begin{equation}
  \label{eq:periodicity}
  c(n,r)=c(n+\lam r+\lam^2m,\,r+2\lam m) \qquad\text{for all $\lam\in\Z$}
\end{equation}
(this is equivalent to~\eqref{eq:H-inv}), so that $c(n,r)$ is actually
only a function of the numbers $d=4nm-r^2$ and $\,r\bmod{2m}\,$ in
$\Z_{\ge0}$ and $\Z/2m\Z$.
A {\em Jacobi cusp form} of weight $k$ and index $m$ is a holomorphic
Jacobi form in which the condition $4nm-r^2\ge0$ in~\eqref{eq-fouexp}
is strengthened to $4nm-~r^2>~0$, while a {\em weak Jacobi form} is
defined like a holomorphic Jacobi form but with the condition
$4nm-r^2\ge0$ in~\eqref{eq-fouexp} dropped entirely; the periodicity
property~\eqref{eq:periodicity} then implies that $c(n,r)=0$ unless
$m\lam^2+r\lam+n\ge0$ for all $\lam\in\Z$ and hence that $|r|$ is
still bounded (by $\sqrt{4nm+m^2}$) for each $n$, so that $\phi$ still
belongs to $\C[\zeta,\zeta^{-1}][[q]]\,$.  Finally, a {\em weakly
  holomorphic Jacobi form} of weight~$k$ and index~$m$ is a
holomorphic function $\phi:\HP\times\C\to\C$
satisfying~\eqref{eq:G-inv}, \eqref{eq:H-inv} and~\eqref{eq-fouexp}
but without the condition $4nm\ge r^2$ in~\eqref{eq-fouexp} and with
the condition $n\ge0$ weakened to $n\ge n_0$ for some $n_0\in\Z$.  An
equivalent definition is that $\Delta(\tau)^h\phi(\tau,z)$ is a
holomorphic Jacobi form (of weight~$k+12h$ and index~$m$) for some
$h\in\Z$, where $\Delta=\eta^{24}\in S_{12}(\Gamma)$.  Such a form has
a Fourier expansion in $\C[\zeta,\zeta^{-1}][\![q^{-1},q]\!]$, the ring of
Laurent series in~$q$ with coefficients which are Laurent polynomials
in~$\zeta$.

The spaces of holomorphic Jacobi forms, Jacobi cusp forms, weak Jacobi
forms and weakly holomorphic Jacobi forms are denoted by $J_{k,m}(1)$,
$J_{k,m}^{\text{cusp}}(1)$, $J_{k,m}^{\text{weak}}(1)$ and
$J_{k,m}^!(1)$, respectively, the latter in analogy with the more
standard notation $M_k^!=M_k^!(\Gamma)$ for the space of weakly
holomorphic modular forms of weight~$k$ on~$\Gamma$ (= holomorphic
functions in~$\HP$ which transform like modular forms of weight~$k$
but are allowed to grow like a negative power of $q$ as
$\Im(\tau)\to\infty$).  The ``1'' in parentheses, which was not used
in~\cite{Eichler-Zagier}, means that the Jacobi forms under
consideration have trivial character, and will be dropped when forms
with arbitrary character are permitted.  For $m=0$ the Jacobi forms
are independent of~$z$, so that we have
$J_{k,0}(1)=J_{k,0}^{\text{weak}}(1)=M_k(\Gamma)$,
$J_{k,0}^{\text{cusp}}(1)=S_k(\Gamma)$, and
$J_{k,0}^{\,!}(1)=M_k^{\,!}(\Gamma)$.  We also have
$J_{k,m}(1)J_{k',m'}(1)\subseteq J_{k+k',m+m'}(1)$, so that the vector
space $J_{*,*}(1)=\bigoplus_{k,m\ge0}J_{k,m}(1)$ is a bigraded ring.
Note that the weights of weak or weakly holomorphic Jacobi forms can
be negative, although in the case of weak Jacobi forms they are
bounded below by $-2m$.

In this paper we will still consider Jacobi forms on the full modular
group, but will allow rational weights and indices.  For such forms
the transformation equations~\eqref{eq:G-inv} and~\eqref{eq:H-inv} are
true only up to certain roots of unity (of bounded order) depending on
$\mat abcd$ and $\bigl(\smallmatrix\lam\\ \mu\endsmallmatrix\bigr)$,
and the exponents~$n$ and~$r$ in~\eqref{eq-fouexp} can be rational
(though again with bounded denominator).  The quickest way to give a
definition is simply to say that $\phi(\tau,z)^N$ is a holomorphic (or
cuspidal, or weak, or weakly holomorphic) Jacobi form of weight $Nk$
and index $Nm$ for some positive integer $N$.  The explicit formulas
for the roots of unity occurring in the transformation equations with
respect to the action of $\Gamma$ and $\Z^2$ (multiplier system) are
quite complicated, but we do not have to give them explicitly because
there is an easy implicit description which suffices for the cases we
are interested in (products of the functions $\vartheta_a(\tau,z)$ and
of rational powers of $\eta(\tau)$).  We use the symbol $\varepsilon$
to denote the multiplier system of the function $\eta(\tau)$, and more
generally $\varepsilon^h$ for any $h\in\Q$ to denote the multiplier
system of (any branch of) the function $\eta(\tau)^h$.  We also note
that the index~$m$ of any Jacobi form $\phi$, even a weakly
holomorphic one or one with arbitrary character, is always a
non-negative half-integer, because $2m$ is the number of zeros of the
function $z\mapsto\phi(\tau,z)$ in a fundamental domain for the action
of the group $\Z\tau+\Z$ of translations of $\C$.  For $m$ integral
and $k,\,h\in\Q$ we will say that a Jacobi form $\phi$ of weight~$k$
and index~$m$ has {\em character $\varepsilon^h$} if
$\eta(\tau)^{-h}\phi(\tau,z)$ is a (weakly holomorphic) Jacobi form in
the usual sense, i.e., if $k-h/2\in\Z$ and
$\eta^{-h}\phi\in J_{k-h/2,m}^{\,!}(1)$.  For half-integral index we
observe that the square of the Jacobi theta function
$\vartheta(\tau,z)^2$ is a holomorphic Jacobi form of weight $1$,
index~$1$ and character~$\varepsilon^6$ in the above sense, so we
simply {\em define} its character to be $\varepsilon^3$; then for
$m\in\Z_{\ge0}+\frac12$ and $k,\,h\in\Q$ we define a Jacobi form of
weight~$k$, index~$m$ and character $\varepsilon^h$ by the requirement
that $\eta(\tau)^{-h-3}\vartheta(\tau,z)\phi(\tau,z)$ belong to
$J^{\,!}_{k-1-h/2,m+1/2}(1)$.  The definitions in both cases depend
only on $h$ modulo~24, so we get spaces $J_{k,m}(\varepsilon^h)$,
$J_{k,m}^{\text{cusp}}(\varepsilon^h)$,
$J_{k,m}^{\text{weak}}(\varepsilon^h)$ and
$J_{k,m}^{\,!}(\varepsilon^h)$ for all $m\in\frac12\Z_{\ge0}$,
$k\in\Q$ and $h\in\Q/24\Z$ with $2k\equiv h\bmod2$.  The
formulas~\eqref{eq:G-inv}, \eqref{eq:H-inv} and~\eqref{eq-fouexp}
imply that
$\phi\in q^{h/24}\,\zeta^m\,\C[\zeta,\zeta^{-1}][\![q^{-1},q]\!]$ for
$\phi$ belonging to any of these spaces.  We clearly have
$J_{k,m}(\varepsilon^h)J_{k',m'}(\varepsilon^{h'})\subseteq
J_{k+k',m+m'}(\varepsilon^{h+h'})$ and also
$\phi_a\in J_{k,a^2m}(\varepsilon^h)$ if
$\phi\in J_{k,m}(\varepsilon^h)$, where $\phi_a(\tau,z)$ denotes the
Jacobi form $\phi(\tau,az)$.  In particular we have
$\vartheta_a\in J_{1/2,a^2/2}(\varepsilon^3)$ and more generally
\begin{equation*}
  \vartheta_{\ve a}\,
  :=
  \,\prod_{i=1}^N\vartheta_{a_i}\;\in\;J_{N/2,A/2}(\varepsilon^{3N})
\end{equation*}
for $\ve a=(a_1,\dots,a_N)\in\Z^N,\;\,A=\sum_{i=1}^Na_i^2$.

It is not hard to show that every function whose $N$-th power, for
some positive integer $N$, is a (weak or weakly holomorphic) Jacobi
form of intgral weight and index with trivial character is indeed in
$J_{k,m}(\varepsilon^h)$ (or $J_{k,m}^{\text{weak}}(\varepsilon^h)$ or
$J_{k,m}^{\,!}(\varepsilon^h)$) for suitable rational $k$ and $h$.
Moreover, it is easily verified that, for any index $m$ in
$\frac12\Z$, the transformation formula~\eqref{eq:H-inv} remains true
if one multiplies the right-hand side by the
factor~$\ex\bigl(m(\lam+\mu)\bigr)$. Note also, that for any rational
$k$ and $h$ and half-integral $m$ every element of $\phi$ in
$J_{k,m}^{\,!}(\varepsilon^h)$ has still a Fourier expansion of the
form~\eqref{eq-fouexp}, where, however, $r$ runs through $\Z$ or
$\frac12\Z$ accordingly as $m$ is integral or not, and $n$ runs
through all rational numbers $n\ge n_o$ which are in $h/24+\Z$. The
(modified) transformation formula~\eqref{eq:H-inv} implies that, for
any integer $\lambda$,
\begin{equation}
  \label{eq:FC-inv}
  C_\phi(\Delta,r)=\ex(m\lambda)C_\phi(\Delta,r+2m\lambda)
  ,
\end{equation}
where $C(\Delta,r)=c\big(\frac{r^2-\Delta}{4m},r\big)$.

Finally, we mention another special Jacobi form:
\begin{equation}
  \label{eq:omega}
  \sectheta(\tau,z)
  =
  \sum_{r\in\Z}
  \leg{12}rq^{r^2/24}\zeta^{r/2}
  ,
\end{equation}
which appears also in the famous Watson quintuple product identity
\begin{multline*}
  \sectheta(\tau,z)
  =\eta(\tau)\frac{\vartheta(\tau,2z)}{\vartheta(\tau,z)} =
  q^{1/24}\zeta^{1/2} \cdot\\
  \cdot \prod_{n=1}^\infty\big(1-q^n\big)
  \bigl(1+q^n\zeta\bigr)\bigl(1+q^{n-1}\zeta^{-1}\bigr)
  \bigl(1-q^{2n-1}\zeta^2\bigr)\bigl(1-q^{2n-1}\zeta^{-2}\bigr) .
\end{multline*}
The Jacobi form $\sectheta$ has weight $1/2$, index $3/2$ and
multiplier system~$\varepsilon$.  For an integer $a$, we will use the
notation $\sectheta_a$ for the Jacobi form~$\sectheta(\tau,az)$.

\section{The order of a weakly holomorphic Jacobi form at infinity}
\label{sec:order}

Let $\phi$ be a weakly holomorphic non-zero Jacobi form $\phi$ of index
$m$ with Fourier coefficients $c_\phi(n,r)$. We associate to $\phi$ a
function $\ord(\phi,x)$ of a real variable $x$ by setting
\begin{equation}
  \label{eq:valuation}
  \ord(\phi,x)
  =
  \min\Big\{n+rx+mx^2: n,r\text{ such that }c_\phi(n,r)\not=0\Big\}
  .
\end{equation}
This function has several remarkable properties and it will play a key
role in the construction of theta blocks of small weights. In
particular, as the following theorem shows, the map
$\phi\mapsto \ord(\phi,\cdot)$ defines a valuation with values in the
(additive) group of continuous functions on $\R/\Z$.  A similar
valuation could be associated to other cusps if one had to consider
Jacobi forms on subgroups of $\SL$ which have more than one cusp,
which justifies calling $\ord(\phi,\cdot)$ the {\em order of $\phi$ at
  infinity}.
\begin{Theorem}
  \label{thm:order}
  The function $\ord_\phi=\ord(\phi,\,\cdot\,)$ defined by~\eqref{eq:valuation} has the
  following properties:
  \begin{enumerate}
  \item It is continuous, piecewise quadratic, and periodic with
    period~$1$.
  \item If $\phi$ is of index $m=0$ (i.e.~if $\phi$ is a weakly
    holomorphic elliptic modular form, independent of~$z$), then
    $\ord_\phi$ is constant and equals the usual order of $\phi$ at
    the cusp infinity.
  \item For fixed any real $u$, $x$ and $y$, there is a constant
    $C=C(u,x,y)$ such that one has
    \begin{equation*}
      \phi(\tau,x\tau+y)\,\ex(mx^2\tau)
      = \big(C+o(1)\big)e^{-2\pi\ord(\phi,x)v}
    \end{equation*}
    as $\tau=u+iv$ and $v$ tends to infinity. The constant $C$ depends
    only on $u,x,y$ modulo $N\Z$ for a suitable integer~$N\ge 1$ and
    is different from zero for almost all~$u,x,y$.
  \item For any two weakly holomorphic Jacobi forms $\phi$ and $\psi$
    one has
    \begin{equation*}
      \ord_{\phi\,\psi}= \ord_{\phi}+\ord_{\psi}
      .
    \end{equation*}
  \item Let $\phi$ in $J_{k,m}^{\,!}(\varepsilon^h)$. Then $\phi$ is
    in $J_{k,m}(\varepsilon^h)$ if and only if $\ord_\phi\ge 0$, and
    $\phi$ is in $J_{k,m}^{\text{cusp}}(\varepsilon^h)$ if and only if
    $\ord_\phi> 0$.
  \item For any integer $l$ and any weakly holomorphic Jacobi form
    $\phi$, one has $\ord_{U_l\phi}(x)=\ord_\phi(lx)$, where $U_l$
    denotes the operator $(U_l\phi)(\tau,z)=\phi(\tau,lz)$.
  \end{enumerate}
\end{Theorem}
\begin{proof}
  For proving (1) we note that $\ord_\phi$ is locally equal to the
  minimum of {\em finitely} many quadratic polynomials, hence is
  continuous and piecewise quadratic. If we write
  \begin{equation*}
    \ord(\phi,x)
    = \min\Big\{\tfrac{(r+2mx)^2-\Delta}{4m}: \Delta, r, \text{ such that }
    c_\phi\big(\tfrac{r^2-\Delta}{4m},r\big)\not=0\Big\}
    ,
  \end{equation*}
  we see that the periodicity is an immediate consequence of the
  identity~\eqref{eq:FC-inv}.

  The statements~(2) and~(6) are obvious, and (4) follows immediately
  from~(3).

  For~(3) we observe that the left-hand side of the claimed identity
  equals
  \begin{multline*}
    \sum_{n,r}
    c_\phi(n,r)\, \ex\big((n+rx+mx^2)\tau+ry\big)\\
    = \sum_{(n,r)\in S} c(n,r)\,\ex(\ord(\phi,x) \tau + ry) +o(e^{-2
      \pi\,\ord(\phi,x) v }) ,
  \end{multline*}
  where $S$ is the (finite) set of pairs $(n,r)$ of rational numbers
  such that $n+rx+mx^2=\ord(\phi,x)$ and $c(n,r)\not=0$.

  Finally, for~(5) we note that $\phi$ is a holomorphic (cusp) form if
  and only if $c_\phi(n,r)=0$ unless the discriminant $r^2-4mn$ of the
  quadratic polynomial $f(x)=n+rx+mx^2$ is (strictly) negative,
  i.e.~unless $f(x)$ is (strictly) positive for all $x$. This proves
  the theorem.
\end{proof}

The order of Jacobi's theta function $\vartheta(\tau,z)$ (introduced
in~\eqref{eq:theta}) will play an important role in the following. We
shall use the letter $B$ for it, i.e.~we set $B(x)=\ord(\vartheta,x)$.
From the Fourier development~\eqref{eq:theta} of $\vartheta$ we see
that $\ord(\vartheta,x)$ equals the minimum of
$\frac18 r^2+\frac12 rx + \frac12 x^2 =\frac12(x+\frac r2)^2$, where
$r$ ranges through the odd integers. In other words
\begin{equation}
  \label{eq:B}
  B(x)
  :=
  \ord(\vartheta,x)
  =
  \min_{k\in\Z} \tfrac12 (x-\tfrac12 + k)^2
  .
\end{equation}
Note also that
\begin{equation*}
  \ord(\vartheta/\eta,x)=\tfrac12 \BP(x)
  ,
\end{equation*}
where $\BP(x)$ is the periodic function with period~$1$ which, for
$0\le x\le1$, equals the second Bernoulli polynomial $x^2-x+\frac16$.

\section{Theta blocks}
\label{sec:generalities}

Recall from Section~\ref{sec:review} that
$\vt a(\tau,z)=\vartheta(\tau,az)$ defines an element
of~$J_{1/2,a^2/2}(\varepsilon^3)$. From Theorem~\ref{thm:order} we deduce that
$\ord(\vt a,x)=B(ax)$ with the function $B(x)$ defined
in~\eqref{eq:B}. From the product expansion~\eqref{eq:triple-product}
of $\vartheta$ we deduce that, for fixed $\tau$, the set of zeros of
$\vartheta(\tau,\cdot)$ coincides with the
lattice~$\Z\tau+\Z$. Accordingly, we find that the zeros of
$\vt a(\tau,z)$ are all simple and are given by the $a$-division
points of the lattice $\Z\tau+\Z$., i.e.~by the points of the lattice
$\frac1a\left(\Z\tau+\Z\right)$.
\begin{Definition}
  A \emph{theta block} of {\em length $r$} is a function of the form
  \begin{equation}
    \label{eq:theta_block}
    \vt{a_1}\vt{a_2}\cdots\vt{a_r}\eta^n
    ,
  \end{equation}
  where $n$ is an integer and the $a_j$ are integers different
  from zero.  A \emph{generalized theta block} is a holomorphic
  function in $\HP\times\C$ of the form
  \begin{equation}
    \label{eq:gen_theta_block}
    \frac{\vt{a_1}\vt{a_2}\cdots\vt{a_r}}{\vt{b_1}\vt{b_2}\cdots\vt{b_s}}\eta^n
    ,
  \end{equation}
  where $n$ is an integer and the $a_j$, $b_j$ are non-zero integers.
  We call a theta block or generalized theta block {\em holomorphic} if
  it is holomorphic also at infinity, i.e., if it is a Jacobi
  form. Conversely, an arbitrary function of the
  form~\eqref{eq:gen_theta_block}, without the requirement of
  holomorphy in $\HP\times\C$, is called a {\em theta quotient}.

\end{Definition}
Occasionally we will also allow rational values for $n$, and will then
call the corresponding function a theta block, generalized theta block
or theta quotient {\em with fractional $eta$-power}. Clearly, any such
function is a meromorphic Jacobi form in
$J_{k,M/2}^{\text{mer.}}(\varepsilon^h)$, where
\begin{equation*}
  k = \tfrac{r-s+n}2,
  \quad
  M = \sum_{i=1}^r a_i^2-\sum_{i=1}^s b_i^2,
  \quad
  h = 3r-3s+n
 ,
\end{equation*}
If $f$ is a generalized theta block (with integral or fractional
eta-power), then $f$ is a weakly holomorphic Jacobi form in
$J_{k,M/2}^{\text{!}}(\varepsilon^h)$.
\begin{Example}
  \label{ex:S_a}
  The function $\sectheta(\tau,z)$ defined in~\eqref{eq:omega} is a
  generalized theta block. More generally, for every positive
  integer~$a$ we have the generalized theta block
  \begin{equation*}
    S_a
    =
    \prod_{d|a}\vt d^{\mu(a/d)}
    ,
  \end{equation*}
  where $\mu$ denotes the M\"obius function.  Note that $S_a$ is
  holomorphic in $\HP \times \C$, its zeros, as function of $z$ for
  fixed $\tau$, are simple and are given by the primitive $a$-division
  points of the lattice $L_\tau=\Z\tau+\Z$, i.e.~by those points of
  $\frac1a L_\tau$ whose images in~$\frac1a L_\tau/L_\tau$ have exact
  order~$a$. Hence $S_a$ defines an element
  of~$J_{0,\psi(a)\phi(a)/2}^!(\varepsilon^{3\varphi(a)})$ for
  $a\ge 2$ (whereas, for $a=1$, we have $S_1=\vartheta$), where
  $\varphi(a)$ is the Euler $\varphi$-function and $\psi(a)$ denotes
  the sum of all positive divisors $d$ of $a$ such that $d/a$ is
  squarefree. Its order at infinity is given by
  \begin{equation*}
    \ord(S_a,x) = \sum_{d|a} \mu(a/d)\,B(dx)
    .
  \end{equation*}
\end{Example}

Note that the theta blocks form a semigroup with respect to the usual
multiplication of functions. We shall denote this semigroup
by~$\TB$. Similarly, the generalized theta blocks form a semigroup
which we shall denote by $\GTB$. The theta quotients, finally, form a
group denoted by $\TQ$. We shall determine the structure of this
group.

For a fixed $\tau$, the divisor of a theta block $f(\tau,z)$, viewed
as a theta function on $\C/\Z\tau+\Z$, is of the form
$\sum_a n_a\Pi_a$, where $a$ runs through $\Z_{>0}$, where the
integers $n_a$ vanish for almost all $a$, and where $\Pi_a$ is the
(formal) sum of the primitive $a$-division points of
$\C/\Z\tau+\Z$. The formal sum
\begin{equation*}
  \Div(f):=\sum_a n_a(a)\in\Z[\Z_{>0}]
\end{equation*}
does not depend on $\tau$. Moreover, the map $f\mapsto \Div(f)$
defines a group homomorphism. Using this map the structure of $\TQ$
can be described as follows.
\begin{Theorem}
  \label{thm:structure-of-group-of-theta-blocks}
  The map $f \mapsto \Div(f)$ defines an exact sequence
  \begin{equation*}
    1 \rightarrow \eta^\Z
    \rightarrow \TQ
    \xrightarrow{\Div} \Z[\Z_{>0}]
    \rightarrow 1.
  \end{equation*}
  The sequence splits via the map
  $D=\sum n_a(a) \mapsto S_D := \prod_a S_a^{n_a}$
\end{Theorem}
\begin{proof}
  From the discussion in Example~\ref{ex:S_a} it is clear that
  $D\mapsto S_D$ defines a section of the map $\Div$, which is then,
  in particular, surjective. If a theta block, for each fixed $\tau$,
  has no zeroes or poles in $\C/\Z\tau+\Z$, then it is of index $0$,
  hence an elliptic modular form without zeroes in the upper half plan
  (but with possibly a pole at the cusp infinity), hence a power of
  $\eta$.
\end{proof}

There are two immediate consequences of the theorem.
\begin{Corollary}
  \label{cor:gen-theta-block-criterion}
  A theta quotient is a generalized theta block (i.e.~weakly
  holomorphic) if and only if it equals a product of the functions
  $S_a$ and a power of $\eta$.
\end{Corollary}

\begin{Theorem}
  \label{thm:number-of-general-theta-blocks}
  For any positive integer or half-integer $m$, the number of
  generalized theta blocks of index $m$, counted up to multiples of
  powers of~$\eta$, is finite. It equals the coefficient of $q^{2m}$
  in the power series expansion of
  $1/\prod_{a=1}^\infty \big(1-q^{\varphi(a)\psi(a)}\big)$.
\end{Theorem}
\begin{proof}
  Indeed, according to the theorem the number in question equals the
  number of elements $D=\sum_an_a(a)$ in $\Z[\Z_{>0}]$ such that all
  $n_a$ are non-negative and $m=\frac12 \sum_a n_a\varphi(a)\psi(a)$.
  But this number is finite since
  \begin{equation*}
    \varphi(a)\psi(a)
    =a^2\prod_{p|a}\big(1-\frac1{p^2}\big)
    >
    a^2\prod_{p}\big(1-\frac1{p^2}\big)
    =
    6a^2/\pi^2
    ,
  \end{equation*}
  which is bigger than $2m$ for $a$ large.
\end{proof}
\begin{Remark}
  It is known~\cite{Skoruppa:Critical-Weight} that, for fixed $m$ and
  $h$, as $k$ tends to infinity one has
  $\dim {J_{k,m}}\big(\varepsilon^h\big) = c\cdot k + O(1)$, where $c$
  is a constant depending on $m$ and $h$. In particular, we see that
  generalized theta blocks of a given index~$m$ can never span the whole space
  of Jacobi forms of weight $k$, index $m$ and given character if $k$
  is sufficiently large.  Table~\ref{tab:nearly_hol_theta_blocks}
  lists, for small indices~$m$, all generalized theta blocks of index
  $m$, up to powers of $\eta$, normalized by a fractional $\eta$-power
  so that the minimum of their order at infinity becomes zero, i.e.~so
  that they are holomorphic but not cuspidal.
\end{Remark}

\begin{table}[tbp]
  \centering
  \caption{\Small%
    For small index $m$ the sets $S_m$ of all non-cuspidal generalized
    theta blocks with fractional $\eta$-power
    in~$\bigoplus_{k,h\in\Q}J_{k,m}\big(\varepsilon^h\big)$}
  \begin{tabular}{ll}
    \toprule
    $m$ & $S_m$\\
    \midrule
    $\frac{1}{2}$& $\vartheta_{1}$  \\
    $1$& $\vartheta_{1}^{2}$  \\
    $\frac{3}{2}$& $\vartheta_{1}^{3}$,  $\frac
                   {\vartheta_{2}}{\vartheta_{1}}\eta^{\frac{1}{24}}$  \\
    $2$& $\vartheta_{1}^{4}$,  $\vartheta_{2}$  \\
    $\frac{5}{2}$& $\vartheta_{1}^{5}$, 
                   $\vartheta_{1}\vartheta_{2}\eta^{-\frac{1}{40}}$  \\
    $3$& $\vartheta_{1}^{6}$, 
         $\vartheta_{1}^{2}\vartheta_{2}\eta^{-\frac{1}{24}}$,  $\frac
         {\vartheta_{2}^{2}}{\vartheta_{1}^{2}}\eta^{\frac{1}{12}}$  \\
    $\frac{7}{2}$& $\vartheta_{1}^{7}$, 
                   $\vartheta_{1}^{3}\vartheta_{2}\eta^{-\frac{3}{56}}$,  $\frac
                   {\vartheta_{2}^{2}}{\vartheta_{1}}\eta^{\frac{1}{28}}$  \\
    $4$& $\vartheta_{1}^{8}$, 
         $\vartheta_{1}^{4}\vartheta_{2}\eta^{-\frac{1}{16}}$, 
         $\vartheta_{2}^{2}$,  $\frac
         {\vartheta_{3}}{\vartheta_{1}}\eta^{\frac{1}{16}}$  \\
    $\frac{9}{2}$& $\vartheta_{1}^{9}$, 
                   $\vartheta_{1}^{5}\vartheta_{2}\eta^{-\frac{5}{72}}$, 
                   $\vartheta_{1}\vartheta_{2}^{2}\eta^{-\frac{1}{36}}$,  $\frac
                   {\vartheta_{2}^{3}}{\vartheta_{1}^{3}}\eta^{\frac{1}{8}}$, 
                   $\vartheta_{3}$  \\
    $5$& $\vartheta_{1}^{10}$, 
         $\vartheta_{1}^{6}\vartheta_{2}\eta^{-\frac{3}{40}}$, 
         $\vartheta_{1}^{2}\vartheta_{2}^{2}\eta^{-\frac{1}{20}}$,  $\frac
         {\vartheta_{2}^{3}}{\vartheta_{1}^{2}}\eta^{\frac{3}{40}}$, 
         $\vartheta_{1}\vartheta_{3}$  \\
    $\frac{11}{2}$& $\vartheta_{1}^{11}$, 
                    $\vartheta_{1}^{7}\vartheta_{2}\eta^{-\frac{7}{88}}$, 
                    $\vartheta_{1}^{3}\vartheta_{2}^{2}\eta^{-\frac{3}{44}}$,  $\frac
                    {\vartheta_{2}^{3}}{\vartheta_{1}}\eta^{\frac{3}{88}}$, 
                    $\vartheta_{1}^{2}\vartheta_{3}$,  $\frac
                    {\vartheta_{2}\vartheta_{3}}{\vartheta_{1}^{2}}\eta^{\frac{9}{88}}$  \\
    $6$& $\vartheta_{1}^{12}$, 
         $\vartheta_{1}^{8}\vartheta_{2}\eta^{-\frac{1}{12}}$, 
         $\vartheta_{1}^{4}\vartheta_{2}^{2}\eta^{-\frac{1}{12}}$, 
         $\vartheta_{2}^{3}$,  $\frac
         {\vartheta_{2}^{4}}{\vartheta_{1}^{4}}\eta^{\frac{1}{6}}$, 
         $\vartheta_{1}^{3}\vartheta_{3}$,  $\frac
         {\vartheta_{2}\vartheta_{3}}{\vartheta_{1}}\eta^{\frac{1}{24}}$,  $\frac
         {\vartheta_{4}}{\vartheta_{2}}\eta^{\frac{1}{24}}$  \\
    \bottomrule
  \end{tabular}
  \label{tab:nearly_hol_theta_blocks}
\end{table}

As we have seen, it is easy to decide whether a theta quotient is
weakly holomorphic.  It remains to analyze the behaviors of a general
theta block at infinity. We shall discuss this question from various
points of view in the next sections. Here we confine ourselves to the
study of the map which associates to a theta quotient its order at
infinity. For this we note that $\ord(f,\cdot)$, for a theta quotient $f$, is
an element of the additive group of real valued functions on the real
line which is spanned by the functions $B(ax)$ ($a\in\Z_{>0}$) and
$\frac 1{24}$.  It is a somewhat surprising fact that the order at
infinity already determines the theta quotient.  Namely, we shall
prove
\begin{Theorem}
  The map $f\mapsto \ord(f,\cdot)$ defines an isomorphism between the group
  of theta quotients $\TQ$ and the additive group $\BF$ of functions
  spanned by the $B(ax)$ ($a\in\Z_{>0}$) and the constant
  function~$\frac 1{24}$.
\end{Theorem}
\begin{proof}
  We shall prove in a moment that the functions $B(ax)$ and $1/24$ are
  linearly independent over $\Z$ (and even over $\C$). Hence from the
  order at infinity $\ord(f,\cdot)$ of a theta quotient $f$ as
  in~\eqref{eq:gen_theta_block} we can read off the numbers $a_j$,
  $b_j$ and $n$, which proves the theorem.

  The claimed linear independence of the $B(ax)$ and $1/24$ becomes
  obvious if one expands $B(x)$ into its Fourier series:
  \begin{equation*}
    B(x) = \frac 1{4\pi^2} \sum_{\begin{subarray}{c} n\in\Z\\ n\not=0
      \end{subarray}
    }\frac{e^{2\pi inx}}{n^2} + \frac 1{24} .
  \end{equation*}
  Hence, if $b(x)=\sum_{l\ge 1} c_l B(lx)+c_0/24$ with integers $c_l$,
  almost all equal to zero, then
  \begin{equation}
    \label{eq:fourier-development-of-order}
    b(x) = \frac 1{4\pi^2} \sum_{\begin{subarray}{c} n\in\Z\\ n\not=0
      \end{subarray}
    } \frac {p(e^{2\pi i nx})-p(0)+p(1)/2} {n^2},
  \end{equation}
  where $p(t)$ denotes the polynomial\footnote{These formulas could be
    written more smoothly if we had defined $\vt a$ as the quotient
    $\vartheta(\tau,az)/\eta(\tau)$, whose order is
    $\frac12\BP_2(ax)$, where $\BP_2(x)=y^2-y+\frac16$ ($y=$
    fractional part of $x$) is the periodically continued second
    Bernoulli polynomial.  } $p(t) = \sum_{l\ge 0} c_lt^l$. (For the
  identity we used also $\sum_{n=1}^\infty 1/n^2=\pi^2/6$.)  By the
  uniqueness of the Fourier expansion of $b(x)$ the polynomial $p$ is
  uniquely determined by $p$, i.e.~we have a map $b\mapsto p$, which
  defines an isomorphism of $\BF$ with the group of polynomials over
  $\Z$. This implies the claimed linear independence.
\end{proof}

It is worthwhile to summarize the discussion of this section in terms
of the composition of the isomorphism $f \mapsto \ord(f,\cdot)$ with the
isomorphism of $\BF$ and the group of polynomials over $\Z$ used in the
preceding proof.
\begin{Theorem}
  \label{thm:the-polynomial-formulation}
  The map
  \begin{equation*}
    p(t)=\sum_{l\ge 0}c_lt^l \mapsto \vartheta_p =
    \eta^{2c_0}\prod_{l\ge 1}\big(\vartheta_l/\eta)^{c_l}
  \end{equation*}
  defines an isomorphism of the (additive) group $\frac12\Z+t\Z[t]$
  and the group~$\TQ$ of theta quotients.  The theta
  quotient~$\vartheta_p$ defines a meromorphic Jacobi form of weight
  $k=p(0)$, index~$m=\frac 12 \big(p'(1)+p''(1)\big)$ and
  character~$\varepsilon^h$ with $h = 2p(1)$.  It is weakly
  holomorphic if and only if, for all positive integers~$N$,
  \begin{equation}
    \label{eq:holo-in-H}
    \frac 1N\sum_{\zeta^N=1} p(\zeta)\ge c_0
  \end{equation}
  (the sum is over all $N$-th roots of unity).  Its order at infinity
  $\ord(\vt p,\cdot)$ is given by
  \begin{equation}
    \label{eq:holo-at-infty}
    \ord(\vt p,x) = \frac1{4\pi^2} \sum_{\begin{subarray}{c} n\in\Z \\
        n\not=0
      \end{subarray}
    } \frac{p(e^{2\pi i xn})} {n^2} .
  \end{equation}
\end{Theorem}
\begin{proof}
  The statements concerning the weight, index and character are
  obvious. (See the discussion at the beginning of this section.)  The
  formula for the order at infinity is a restatement
  of~\eqref{eq:fourier-development-of-order}.  Finally, if we write
  $\vartheta_p=\eta^{c}\prod_a S_a^{n_a}$, then
  $p(t)=c_0+\sum_l\sum_a n_a\mu(a/l)t^l$ (where $\mu(a/l)=0$ if $a$ is not
  a multiple of $l$). Accordingly we find
  $n_a+c_0=\frac1a \sum_{\zeta^a=1}p(\zeta)$, and we recognize the stated
  criterion for being weakly holomorphic as a restatement of the first
  corollary of Theorem~\ref{thm:structure-of-group-of-theta-blocks}.
\end{proof}

The construction of generalized theta blocks, i.e.,~of theta quotients
which define Jacobi forms, amounts therefore to the construction of
polynomials $p(t)$ whose coefficients apart from the constant term are
integers, that satisfy~\eqref{eq:holo-in-H} and such that the
right-hand side of~\eqref{eq:holo-at-infty} is non-negative for all
$x\in\R$. We come back to this question in the following section.

We end this section by a criterion for a Jacobi form to be a
generalized theta block, whose easy proof we leave to the reader.
\begin{Theorem}
  \label{eq:theta-block-criterion}
  A weakly holomorphic Jacobi form $\phi$ on the full modular group is
  a generalized theta block if and only if, for every $\tau$, the
  function $z\mapsto \phi(\tau,z)$ has at most division points of
  $\C/\Z\tau+\Z$ as zeroes.
\end{Theorem}

\section{Long theta blocks of low weight}
\label{sec:long-and-low}

In this section and the next section we shall be interested to
construct Jacobi forms as theta blocks (with fractional eta power)
consisting of many factors $\vt a$ but still with low weight. As
already remarked after Theorem~\ref{thm:the-polynomial-formulation}
this amounts to the construction of polynomials $p(t)$ in $\R+t\Z[t]$
whose coefficients apart from the first one are non-negative, such that
$p(1)$ is large but $p(0)$ is at the same time small, and such that
the right-hand side of~\eqref{eq:holo-at-infty} is non-negative.

Accordingly the growth of the following function $\sym{wt}(N)$ is of
interest:
\begin{equation}
  \label{eq:the-challenge-quantities}
  \begin{split}
    \max(N) &= \sup \Big\{ \min_x \Bigl({\textstyle\sum}_{j=1}^N
    B(a_jx)\Bigr):a_1,\dots,a_N\in\Z_{\ge 1} \Big\}
    ,\\
    \sym{wt}(N) &= \tfrac N2 - 12\max(N) .
  \end{split}
\end{equation}
The quantity $\sym{wt}(N)$ measures the lowest weight for which there
exists a Jacobi form which is a theta block built from exactly $N$
factors $\vt a$. Clearly, $\max(N) < N/24$ (since
$\min_x B(ax) < \int_0^1 B(ax)\,dx=1/24$) and therefore
$\sym{wt}(N) > 0$.  Alternatively, this inequality also follows from
the fact that a non-constant holomorphic Jacobi form has positive
weight.

As a consequence of Theorem~\ref{thm:the-polynomial-formulation} we
can relate our problem to one which is well studied in the literature.
\begin{Lemma}
  Let $T_N$ denote the set of polynomials $p(t)$ in $\R+t\Z_{\ge0}[t]$
  such that $p(e^{2\pi ix})+p(e^{-2\pi ix})\ge 0$ for all real $x$,
  and whose sum of non-constant coefficients equals~$N$.  One has
  \begin{equation*}
    \sym{wt}(N) \le \inf\big\{p(0):
    p\in T_N \big\}
    \,
    .
  \end{equation*}
\end{Lemma}
\begin{proof}
  The inequality results from the fact that the image of $T_N$ under
  the map of Theorem~\ref{thm:the-polynomial-formulation} is contained
  in the set $\widetilde T_N$ of theta blocks (with fractional eta
  power) whose order function is non-negative.
\end{proof}

We do not know whether the image of $T_N$ equals $\widetilde T_N$.  If
this held true then the inequality of the lemma would in fact be an
equality.

The asymptotic behavior of
\begin{equation*}
  \sym{ct}(N):=\inf\big\{p(0): p\in T_N \big\}
\end{equation*}
was studied in \cite{Odlyzko}, \cite{Kolountzakis},
\cite{Belov-Konyagin} et.~al. In the first two of these articles it
was shown that $\sym{ct}(N)$ does not grow faster than
$\log(n)\,n^{1/3}$ and $n^{1/3}$, respectively. The so far strongest
result (to the best of our knowledge) is $\sym{ct}(N)\ll \log^3 N$ for
$N\ge 2$ (see~\cite[Thm.~0.5]{Belov-Konyagin}).  More precisely, one
has
\begin{TheoremX}[\protect{\cite[Cor.~5.4]{Belov-Konyagin}}]
  For all $N$ one has
  \begin{equation*}
    \sym{ct}^{\downarrow}
    :=
    \inf\left\{p(0):
    p\in T_N^{\downarrow} \right\}
    \le
    45\,000
    (1+(\log N)^3)
    ,
  \end{equation*}
  where $T_N^{\downarrow}$ is the subset of polynomials
  $p(t)=a_0+a_1t+a_2+\cdots$ in $T_N$ whose non-constant coefficients
  $a_1$, $a_2$, \dots form a decreasing sequence.
\end{TheoremX}
Note that the right-hand side is also an upper bound for $\sym{ct}(N)$
since $T_N^{\downarrow}$ is a subset of $T_N$. The same paper
(\protect{\cite[Thm.~0.5]{Belov-Konyagin}}) also gives an estimate of
$\sym{ct}^{\downarrow}(N)$ from below, namely
\begin{equation*}
  \tfrac{\log^2N}{\log\log N}
  \ll
  \sym{ct}^{\downarrow}(N)
  .
\end{equation*}
However, since we know neither the exact relation between
$\sym{wt}(N)$ and $\sym{ct}(N)$ nor between the latter and
$\sym{ct}^{\downarrow}(N)$, the last estimate is not useful for us. It
might give an indication for a lower bound of $\sym{wt}(N)$ though.

We can indeed prove a similar estimate for $\sym{wt}(N)$ from below by
relating $\sym{wt}(N)$ to another well-studied problem, namely the
determination of the quantity
\begin{equation*}
  A(N)
  :=
  \inf\left\{
    \int_0^1\big|\sym{Re} p(e^{2\pi i x})\big|\,dx:
    p(t)\in t\Z_{\ge 0}[t],\ p(1)=N
  \right\}
  .
\end{equation*}
We thank Danylo Radchenko who pointed out this connection to us and
also found and proved the following Lemma.
\begin{Lemma}
  \label{eq:L1-estimate}
  For all $N\ge 1$, one has
  \begin{equation*}
    \sym{wt}(N)
    \ge
    \Big(\frac 6{\pi^2}-\frac12\Big)
    \cdot A(N)
    .
  \end{equation*}
\end{Lemma}

\begin{proof}
  Let $p(t)$ be in $t\Z[t]$, and let $c_0$ be a real number such that
  $c_0/12+\ord(\vt {p},x)=\ord(\vt {c_0+p},x)\ge 0$ for all $x$.  Then
  $c_0/12$ is an upper bound for
  -$\int_0^1 \min(\ord({\vt {p}},x),0)\,dx$. But the latter integral
  equals $\frac12\int_0^1|\ord({\vt {p}},x)|\,dx$ (since
  $\int_0^1 \ord({\vt {p}},x)\,dx$ equals
  $0$). From~\eqref{eq:holo-at-infty} we therefore obtain, setting
  $I_n=\int_0^1|\sym{Re} p(e^{2\pi inx})|\,dx$,
\begin{equation*}
    c_0\ge 6\int_0^1|\ord({\vt {p}},x)|\,dx
    \ge
    \frac 3{\pi^2}\bigg(I_1
    -
    \sum_{n\ge 2}\frac {I_n}{n^2}
    \Bigg)
    =\Big(\frac 6{\pi^2}-\frac 12\Big)I_1
    ,
  \end{equation*}
  where for the last equality we used that the $p(e^{2\pi inx})$ all
  have the same $L^1$-norm. The lemma is now obvious.
\end{proof}

Lower bounds for the left-hand side of the inequality of the last
lemma have been studied (in a more general context)
in~\cite{McGehee-Carruth-Pigno}. In particular, their results imply
\begin{TheoremX}[\protect{\cite[Thm.~2]{McGehee-Carruth-Pigno}}]
  \label{thm:MCP-consequence}
  For all $N\ge 1$ one has
  \begin{equation*}
    A(N)
    \ge
    \frac {H_{2N}}{60}
    ,
  \end{equation*}
  where $H_{2N}=\sum_{n=1}^{2N}\tfrac 1n$ denote the $2N$th harmonic
  number.
\end{TheoremX}
This theorem as stated here is not exactly identical to
\cite[Thm.~2]{McGehee-Carruth-Pigno}. In fact, they prove, for any
sequence of integers $a_1<a_2<\dots< a_n$ and any sequence of complex
numbers $\lambda_1$, \dots, $\lambda_N$, the inequality
\begin{equation*}
  \int_0^1\big|\sum_{j=1}^n\lambda_j e^{2\pi ia_jx}\big|
  \ge
  \frac 1{60}\sum_{j=1}^n\frac{|\lambda_j|}{j}
  .
\end{equation*}
Theorem~\ref{thm:MCP-consequence} is an obvious consequence.

Summarizing the preceding discussion we obtain
\begin{Theorem}
  \label{thm:theoretical-answer-to-the-challenge}
  The quantity $\sym{wt(N)}$ in~\eqref{eq:the-challenge-quantities}
  satisfies
  \begin{equation*}
    \frac {H_{2N}}{555.930\dots}\le \sym{wt}(N) \le 45000 (1+\log^3N)
  \end{equation*}
  for all $N\ge 1$.
\end{Theorem}
In particular, $\sym{wt}(N)$ grows at least like a constant
times~$\log N$ and at most like a constant times~$\log^3 N$ as $N$
goes to infinity. Note, however, that the bounds given in the theorem
are very poor for $N$ of intermediate size. For instance, for $N=50$
these bounds are
\begin{equation*}
  0.00933 \;\le\;\sym{wt}(50) \;\le\;2.74\,\times\,10^6\,,
\end{equation*}
whereas Table~\ref{tab:best-values} in the next section shows that in
fact $\,\sym{wt}(50)<2.224\,$.

As we see from Theorem~\ref{thm:theoretical-answer-to-the-challenge}
there exist theta blocks with an arbitrary high number $N$ of
$\vartheta$-factors which are Jacobi forms but have relatively small
weight $\ll \log^3 N$. It is challenging to construct such theta
blocks explicitly. The rest of this article will somehow pivot around
this subject. In particular, we shall construct infinite families of
theta blocks with a high number of $\vt a$-factors, fairly small
weight and yet holomorphic at infinity. We shall even develop a theory
that will permit to construct such families systematically.  In the next section, however, we confine ourselves to describing the
results of our direct search for interesting theta blocks.

\part*{Part II: Examples}

\section{Experimental search for long theta blocks}
\label{sec:Experimental-search-for-long-theta-blocks}

As we explained in the last section we are interested in long theta
blocks of low weight which are holomorphic at infinity.  For this we
need, first of all, to describe an efficient method to calculate the
minimum of the order of a theta block.  For $\ve a=(a_1,\dots,a_N)$ in
$\Z^N$, set
\begin{gather*}
  \vt {\ve a}:=\prod_{j=1}^N \vt {a_j},
  \\
  B_{\ve a}(x):=\sum_{j=1}^N B(a_jx), \quad s_{\ve a} = 24\min_{x}
  B_{\ve a}(x) .
\end{gather*}
Recall that the theta block $\vt {\ve a}$ has $B_{\ve a}$ as order at
infinity. Hence $s_{\ve a}$ is the maximal fractional power of $\eta$
by which we can divide $\vt {\ve a}$ and still have a Jacobi form. The
weight of the resulting form is
\begin{equation*}
  k_{\ve a}
  =
  \tfrac 12 (N - s_{\ve a})
  .
\end{equation*}

Note that $B(x)$ is one half of the square of the distance of $x$ to
the closest point in $\frac12+\Z$. Accordingly, $B_{\ve a}(x)$, for a
given $x$, is one half of the square of the Euclidean distance of
$x\ve a$ to the closest point in $\ve {\frac 12}+\Z^N$, where
$\ve {\frac 12}=(\frac12,\dots,\frac12)$.  In other words,
$s_{\ve a}/24$ is one half of the square distance of the line
$\R \cdot \ve a$ to the set $\ve {\frac12}+\Z^N$. If $\frac {\ve n}2$
is a point in $\ve {\frac12}+\Z^N$ then its square distance to
$\R\cdot {\ve a}$ equals the square length of its orthogonal
projection onto the orthogonal complement of $R\cdot {\ve a}$, i.e.~it
equals
$\frac 14\left({\ve n}^2-{(\ve n \cdot \ve a)^2}/{\ve a^2}\right)$.
If we set
\begin{equation}
  \label{eq:S-a}
  S_{\ve a}({\ve n})
  :=
  {\ve n}^2\cdot{\ve a^2} - {(\ve n \cdot \ve a)^2}
  ,
\end{equation}
then we can summarize
\begin{equation}
  \label{eq:B-QF}
  s_{\ve a}
  =
  24\min_{x} B_{\ve a}(x)
  =
  3\min_{\ve n \in \ve 1+2\Z^N} S_{\ve a}({\ve n})/{\ve a}^2
  .
\end{equation}
This formula has several consequences.

First of all, if ${\ve n}_0$ in $\ve 1+2\Z^N$ ($\ve 1=(1,\dots,1)$)
minimizes $S_{\ve a}({\ve n})$, then the minimum of $B_{\ve a}(x)$ is
assumed at $x={{\ve n}_0\cdot \ve a}/2{\ve a^2}$. Note that $x$ is a
rational number with denominator $2{\ve a^2}$. More precisely, we can
state

\begin{Proposition}
  $B_{\ve a}(x)$ assumes its minimum at one of the points
  \begin{equation}
    x=\frac s{2M}+\frac{k}M
    \quad (0\le k <M)
    ,
  \end{equation}
  where $s=\sum_{j=1}^Na_j$ and $M=\sum_{j=1}^Na_j^2$.
\end{Proposition}
\begin{Remark}
  The proposition tells us in particular that we can determine the
  minimum of $B_{\ve a}(x)$ for a given ${\ve a}$ by trying all the
  $M$ values $x$ as in the theorem, which needs
  $M=\sum_{j=1}^Na_j^2$ steps.
\end{Remark}

Secondly,~\eqref{eq:B-QF} implies the following criterion for
$\vt {\ve a}/\eta^d$ defining a Jacobi form.

\begin{Proposition}
  \label{prop:hyperbolic-criterion}
  The quotient $\vt {\ve a}/\eta^d$ is holomorphic at infinity if and
  only if
  \begin{equation*}
    S_{\ve a}({\ve n}) \ge \frac d3 \,{\ve a}^2
  \end{equation*}
  for all vectors ${\ve n}$ in $\ve 1+2\Z^N$ (with $S_a({\ve n})$ as
  in~\eqref{eq:S-a}). It is a cusp form if and only if the inequality
  is strict for all ${\ve n}$ in $\ve 1+2\Z^N$.
\end{Proposition}
\begin{Remark}
  As we shall see below it is sometimes useful to write
  $S_{\ve a}({\ve n})$ in a slightly different form. Namely, as a
  simple computation shows, one has
  \begin{equation*}
    S_{\ve a}({\ve n})=\sum_{1\le i<j\le N} \left(a_in_j-a_jn_i\right)^2
    ,
  \end{equation*}
  where we used ${\ve n}=(n_1,\dots,n_N)$.
\end{Remark}

For minimizing $S_{\ve a}({\ve n})$ for a given $\ve a$ the following
formula is sometimes useful.
\begin{Proposition}
  \label{prop:dual-criterion}
  Let $\ve u_i$ ($1\le j\le r$) be linearly independent vectors in
  $\Z^N$ spanning the orthogonal complement of $\ve a$, and let
  $G=\left(\ve u_i\cdot \ve u_j\right)_{1\le i,j\le r}$ be the Gram
  matrix of the $\ve u_j$. Then
  \begin{equation*}
    S_{\ve a}({\ve n}) = \left(\ve xG^{-1}\ve x^t\right){\ve a}^2
    ,
  \end{equation*}
  where $\ve x=(\ve n\cdot \ve u_1,\dots,\ve n\cdot \ve u_r)$.
\end{Proposition}

\begin{Remark}
  If the $\ve u_j$ of the proposition do not span the orthogonal
  complement of $\ve a$ but are still orthogonal to $\ve a$, then we
  have still
  \begin{equation*}
    S_{\ve a}({\ve n}) \ge \left(\ve xG^{-1}\ve x^t\right){\ve a}^2
  \end{equation*}
  for all vectors ${\ve n}$ in $\ve 1+2\Z^N$, where
  $\ve x=(\ve n\cdot \ve u_1,\dots,\ve n\cdot \ve u_r)$ (as one easily
  sees by complementing the $\ve u_j$ to a full basis of the
  orthogonal comeplement of $\ve a$ by integral vectors which are
  orthogonal to the $\ve u_j$.)

  Assume that $\ve u_j^2$ is odd for all~$j$. Then
  $\ve n \cdot \ve u_j\equiv \ve u_j^2\equiv 1\bmod 2$ for ${\ve n}$
  in $\ve 1+2\Z^N$, i.e.~$\ve x\in \ve 1+2\Z^r$, and hence
  \begin{equation*}
    S_{\ve a}({\ve n}) \ge \left(\min_{\ve x \in \ve 1+2\Z^r} \ve xG^{-1}\ve x^t\right){\ve a}^2
  \end{equation*}

  If the $\ve u_i$ are in addition pairwise orthogonal, so that
  $G^{-1}$ is the diagonal matrix with $1/\ve u_j^2$ ($1\le j\le r$)
  as diagonal elements, we conclude (using $(\ve n\cdot u_j)\ge 1$)
  \begin{equation*}
    S_{\ve a}({\ve n}) \ge \left(\sum_{j=1}^r \frac 1{\ve u_j^2}\right){\ve a}^2
  \end{equation*}
  for all vectors ${\ve n}$ in $\ve 1+2\Z^N$.
\end{Remark}
\begin{proof}[Proof of Proposition~\ref{prop:dual-criterion}]
  Indeed, for $\ve x$ in $\R^N$ let $\ve x_\perp$ the orthogonal
  projection of $\ve x$ onto the space spanned by the $\ve u_j$. Then
  $\ve n^2=\ve n_\perp^2 + (\ve n-\ve n_\perp)^2$ and
  $\ve n_\perp\cdot\ve a=0$, and hence
  \begin{equation*}
    S_{\ve a}(\ve n)
    =
    \ve n_\perp^2\cdot \ve a^2 + (\ve n-\ve n_\perp)^2\cdot \ve a^2
    - \left((\ve n-\ve n_\perp)\cdot\ve a\right)^2
    \ge \ve n_\perp^2\cdot \ve a^2
    .
  \end{equation*}
  But $\ve n_\perp=\sum_{j=1}^r (\ve n \cdot \ve u_j)\ve u_j^*$, where
  $\ve u_j^*$ are the vectors of the dual basis of $\ve u_j$
  ($j=1,\dots,r$) in the space spanned by the $\ve u_j$. Therefore
  $\ve n_\perp^2 = \ve x H \ve x^t$ with
  $\ve x=(\ve n\cdot \ve u_1,\dots,\ve n\cdot \ve u_r)$ and
  $H=\left(\ve u_i^*\cdot \ve u_j^*\right)_{i,j}$. Since $H=G^{-1}$,
  the proposition is now obvious.
\end{proof}

\begin{table}[th]
  \centering
  \caption{\Small Best experimental values of $k_{\ve a}$ for
    $N\le 50$.  The first three rows give the true best
    values. ($\underline{a}$ stands for the vector $(1,2,\dots,a)$,
    $a$ for the vector $(a)$ and '$\cdot$' for concatenation; hence
    $\underline 5\cdot2\cdot 7=(1,2,3,4,5,2,7)$.)}
  \tiny
  \begin{tabular}{@{} >{$}r<{$} >{$}r<{$} >{$}r<{$} >{$}l<{$} @{}}
    \toprule
    {N}&{k_{\ve a}\phantom{0.500}}&{\ve a^2}&{\ve a}\\
    \midrule

1 & 1/2 = 0.500 & 1 & 1\\2 & 7/10 = 0.700 & 5 & \underline{2}\\3 & 6/7 = 0.857 & 14 & \underline{3}\\4 & 9/10 = 0.900 & 15 & \underline{3}\cdot1\\5 & 25/22 = 1.136 & 55 & \underline{5}\\6 & 27/28 = 0.964 & 56 & \underline{5}\cdot1\\7 & 11/9 = 1.222 & 108 & \underline{5}\cdot2\cdot7\\8 & 49/40 = 1.225 & 240 & \underline{7}\cdot10\\9 & 37/28 = 1.321 & 168 & \underline{6}\cdot2\cdot3\cdot8\\10 & 13/11 = 1.182 & 286 & \underline{9}\cdot1\\11 & 289/193 = 1.497 & 386 & \underline{10}\cdot1\\12 & 465/338 = 1.376 & 507 & \underline{11}\cdot1\\13 & 589/398 = 1.480 & 796 & \underline{11}\cdot1\cdot17\\14 & 304/205 = 1.483 & 820 & \underline{13}\cdot1\\15 & 1917/1210 = 1.584 & 605 & \underline{10}\cdot1\cdot3\cdot4\cdot5\cdot13\\16 & 281/172 = 1.634 & 1032 & \underline{14}\cdot1\cdot4\\17 & 175/107 = 1.636 & 1284 & \underline{14}\cdot2\cdot3\cdot16\\18 & 5007/3002 = 1.668 & 1501 & \underline{16}\cdot\underline{2}\\19 & 1463/895 = 1.635 & 1790 & \underline{17}\cdot\underline{2}\\20 & 256/151 = 1.695 & 2114 & \underline{18}\cdot\underline{2}\\21 & 2839/1650 = 1.721 & 2475 & \underline{19}\cdot\underline{2}\\22 & 9607/5750 = 1.671 & 2875 & \underline{20}\cdot\underline{2}\\23 & 2933/1658 = 1.769 & 3316 & \underline{21}\cdot\underline{2}\\24 & 2391/1339 = 1.786 & 2678 & \underline{19}\cdot\underline{3}\cdot5\cdot13\\25 & 13961/7618 = 1.833 & 3809 & \underline{22}\cdot\underline{3}\\26 & 54/29 = 1.862 & 4350 & \underline{23}\cdot1\cdot3\cdot4\\27 & 18441/9926 = 1.858 & 4963 & \underline{23}\cdot\underline{3}\cdot25\\28 & 20515/11078 = 1.852 & 5539 & \underline{25}\cdot\underline{3}\\29 & 4577/2486 = 1.841 & 6215 & \underline{26}\cdot\underline{3}\\30 & 6459/3472 = 1.860 & 6944 & \underline{27}\cdot\underline{3}\\31 & 9679/5190 = 1.865 & 7785 & \underline{27}\cdot\underline{3}\cdot29\\32 & 427/220 = 1.941 & 7040 & \underline{26}\cdot\underline{5}\cdot28\\33 & 8187/4285 = 1.911 & 8570 & \underline{29}\cdot\underline{3}\cdot1\\34 & 34583/17338 = 1.995 & 8669 & \underline{28}\cdot\underline{5}\cdot30\\35 & 13259/6970 = 1.902 & 10455 & \underline{31}\cdot\underline{3}\cdot5\\36 & 42723/21038 = 2.031 & 10519 & \underline{31}\cdot\underline{3}\cdot5\cdot8\\37 & 12403/6272 = 1.978 & 12544 & \underline{33}\cdot\underline{3}\cdot1\\38 & 1002/479 = 2.092 & 11496 & \underline{32}\cdot\underline{5}\cdot1\\39 & 3371/1678 = 2.009 & 12585 & \underline{33}\cdot\underline{5}\cdot1\\40 & 63307/30026 = 2.108 & 15013 & \underline{35}\cdot\underline{3}\cdot5\cdot8\\41 & 18392/8795 = 2.091 & 17590 & \underline{37}\cdot\underline{3}\cdot1\\42 & 17649/8131 = 2.171 & 16262 & \underline{36}\cdot\underline{5}\cdot1\\43 & 2763/1306 = 2.116 & 17631 & \underline{37}\cdot\underline{5}\cdot1\\44 & 29753/13714 = 2.170 & 20571 & \underline{39}\cdot\underline{4}\cdot1\\45 & 21777/10298 = 2.115 & 20596 & \underline{39}\cdot\underline{5}\cdot1\\46 & 25033/11105 = 2.254 & 22210 & \underline{40}\cdot\underline{4}\cdot2\cdot6\\47 & 11381/5306 = 2.145 & 23877 & \underline{41}\cdot\underline{5}\cdot1\\48 & 40449/18310 = 2.209 & 27465 & \underline{43}\cdot\underline{4}\cdot1\\49 & 126745/58802 = 2.155 & 29401 & \underline{44}\cdot\underline{4}\cdot1\\50 & 34937/15713 = 2.223 & 31426 & \underline{45}\cdot\underline{4}\cdot1\\    
    
    \bottomrule
  \end{tabular}
  \medskip
  \label{tab:best-values}
\end{table}

We are interested in the behavior of $s_{\ve a}$ (or $k_{\ve a}$) as a
function of $\ve a$, and, in particular, to find $\ve a$ in $\Z^N$ for
big $N$ but with $s_{\ve a}$ as big as possible, or, equivalently,
with $k_{\ve a}$ as small as possible. As is clear from the definition
of $s_{\ve a}$ its value does not change if we divide ${\ve a}$ by the
gcd of its entries. When looking for ${\ve a}$ with best $s_{\ve a}$
we can therefore assume that ${\ve a}$ is primitive. Except for the
first few $N$, we do not know any method to determine, for a
given~$N$, the smallest possible weight $k_{\ve a}$, when ${\ve a}$
runs though all integral vectors (with positive entries) of
length~$N$. For $N=1$ the minimum $s_a$ of $B(ax)$, for any integral
$a$, is~$0$, which is assumed by $\vartheta(\tau,z)$.

Already for $N=2$ it is not completely evident to
determine~$s_{\ve a}$ for a given (primitive)~${\ve a}=(a,b)$. A
simple calculation shows $S_{(a,b)}(r,s) = (as-br)^2$.  Writing
$r=-1-2k$ and $s=1+2l$, we have
$S_{(a,b)}(r,s)=\left(a+b+2(al+bk)\right)^2$. The minimum over all
integers $k$ and $l$ equals obviously the rest $s$ of $a+b$
modulo~$2$, whence~$s_{(a,b)}=3/(a^2+b^2)$ if $a+b$ is odd, and
$s_{(a,b)}=0$ otherwise.  The maximal $s_{(a,b)}$ is therefore assumed
for $a,b=1,2$, for which we have $s_{(a,b)}=3/5$.

For larger $N$ we did searches by trial and error to find ${\ve a}$
with small~$k_{\ve a}$.  Our best results are listed in
Table~\ref{tab:best-values}. We do not know how far off
our~$k_{\ve a}$ are from the true minima. Note that, for small $N$,
Theorem~\ref{thm:theoretical-answer-to-the-challenge} does not give
any useful hint in this respect.

\FloatBarrier

\section{Theta quarks}
\label{sec:quarks}

It turns out that there are infinite families of theta blocks which
are holomorphic Jacobi forms. An explanation for this will be given by
the theory which we shall develop in Section~\ref{sec:families}. In
this section we discuss the first non-trivial example of such a
family, the family of {\em theta quarks}, which was already introduced
in the introduction. Recall that this family is given by
\begin{equation*}
  Q_{a,b}
  =
  \vartheta_a\vartheta_b\vartheta_{a+b}/\eta
  \qquad(a,b\in\Z_{> 0})
  .
\end{equation*}
We use the word ``quark'' for these functions because the product
of any three of them is a Jacobi form without character on the full
modular group.  We shall give six different proofs for the fact that
$Q_{a,b}$, for any pair of positive integers $a$, $b$ is indeed
holomorphic at infinity.
\begin{Theorem}
  \label{thm:quarks-are-Jf}
  For any pair of positive integers $a$ and $b$, the function
  $Q_{a,b}$ defines a holomorphic Jacobi form of weight~$1$,
  index~$a^2+ab+b^2$ and character $\varepsilon^3$. It is a cusp form
  if and only if $3\mid a'b'(a'+b')$ where $a'=a/g$ and $b'=b/g$ with
  $g$ denoting the greatest common divisor of $a$ and $b$.
\end{Theorem}
\begin{Remark}
  Note the the condition $3\mid a'b'(a'+b')$ is equivalent to
  $a'\not\equiv b'\bmod 3$ as we shall occasionally use in the
  following proofs.
\end{Remark}
\begin{proof}[First proof]
  According to Theorem~\ref{thm:order} we have to show that
  \begin{equation*}
    \min_x\ord(Q_{a,b},x)\ge 0
  \end{equation*}
  with equality if and only if $3g$ divides~$a-b$.  For this recall
  \begin{equation*}
    \ord(Q_{a,b},x)=B(ax)+B(bx)+B(-(a+b)x) - 1/24
  \end{equation*}
  (where we used that $B(x)$ is an even function), so that
  \begin{equation*}
    \min_{x}\ord(Q_{a,b},x)
    \ge
    \min_{(x,y,z)\in H}\left(B(x)+B(y)+B(z)\right) - 1/24.
  \end{equation*}
  where $H$ denotes the hyperplane $x+y+z=0$.  If $x$, $y$ or $z$ is
  an integer the right-hand side is greater or equal to
  $B(0) = 1/8> 1/{24}$. Otherwise the right-hand side is
  differentiable in small neighborhood of $(x,y,z)$ and we can apply
  the method of Lagrangian multipliers: if $(x,y,z)$ is a local
  minimum then
  $(\overline x, \overline y, \overline z) =\lambda (1,1,1)$ for some
  $\lambda$, where $\overline x, \overline y,\dots$ denote the
  fractional parts of $x, y,\dots$.  The minimum of $B(x)+B(y)+B(z)$
  on $H$ is therefore taken on at
  $\overline x=\overline y=\overline z=1/3$ or
  $\overline x=\overline y=\overline z=2/3$, and it equals in either
  case~$1/24$.

  We leave it to the reader to work out when $Q_{a,b}$ is a cusp form.
\end{proof}

\begin{proof}[Second proof]
  For this proof we use the criterion of
  Proposition~\ref{prop:hyperbolic-criterion}. In the notations of the
  preceding section, we have $Q_{a,b}=\vartheta_{\ve a}/\eta$, where
  $\ve a=(a,b,a+b)$. The vector $\ve u=(1,1,-1)$ is perpendicular to
  $\ve a=(a,b,a+b)$, and hence by the remark after
  Proposition~\ref{prop:dual-criterion}
  \begin{equation*}
    S_{\ve a}(\ve n)\ge \frac 1{\ve u^2}\ve a^2=\frac 13\ve a^2
  \end{equation*}
  fo all $\ve n$ in $\ve 1+2\Z^3$.  According to
  Proposition~\ref{prop:hyperbolic-criterion}, the Jacobi form
  $Q_{a,b}$ is therefore holomorphic at infinity, and it is a cusp
  form if and only if the last inequality is strict for all~$\ve n$.
\end{proof}

\noindent{\em Third proof.}
The holomorphy of $Q_{a,b}$ also follows from the following explicit
formula for its Fourier expansion.
\begin{Theorem}
  \label{thm:theta-quarks-Fourier-expansion}
  One has
  \begin{equation}
    \label{eq:quark-expansion}
    Q_{a,b}
    =
    -\sum_{r,s\in\Z}
    \leg s3
    q^{r^2+rs+s^2/3}
    \zeta^{(a-b)r+as}
    .
  \end{equation} 
\end{Theorem}
\begin{proof}
  We have an isomorphism of $\Z$-lattices
  \begin{gather*}
    \bigl\{(l,m,n)\,\in\,\Z^3\mid\,l\equiv m\equiv n\bmod2\bigr\}
    \cong
    \bigl\{(r,s,t)\,\in\,\Z^3\mid\,s\equiv t\bmod3\bigr\} \\
    (l,\,m,\,n)\mapsto
    \bigl(\frac{n-m}2\,,\;\frac{l+m}2-n\,,\;-l-m-n\bigr)
  \end{gather*}
  with respect to which $\leg{-4}{lmn}=\leg{-4}t.\,$ Hence
  \begin{equation*}
    \begin{aligned}
      -\,\vartheta_a\,\vartheta_b\,\vartheta_{a+b} &=
      \sum_{l,\,m,\,n\,\in\,\Z}\leg{-4}{lmn}
      q^{\frac{l^2+m^2+n^2}8}\,\zeta^{\frac{al+bm-(a+b)n}2} \\
      &= \sum_{\begin{subarray}{c} r,\,s,\,t\,\in\,\Z\\s\equiv
          t \bmod 3
        \end{subarray}
      }\leg{-4}t q^{r^2+rs+s^2\!/3\,+\,t^2\!/24}\;\zeta^{(a-b)r+as}\,,
    \end{aligned}
  \end{equation*}
  and~\eqref{eq:quark-expansion} follows because
  $\sum\limits_{t\equiv s\bmod3}\leg{-4}t q^{t^2\!/24} =\leg s3
  \eta(\tau)$ for all $s$.
\end{proof}

\begin{Remark}
  The above isomorphism of lattices is $\mathfrak S_3$-equivariant if
  we introduce new coordinates $(u,v,w)$ with $u+v+w=0$,
  $u\equiv v\equiv w\bmod3$ which are related to $r$ and~$s$ by
  $(u,v,w)=(-3r-2s,\,3r+s,\,s)$.  Then~\eqref{eq:quark-expansion} can
  be symmetrically written in terms of the three integers $a$, $b$ and
  $c=-a-b$ with sum 0 by
  \begin{equation}
    \label{eq:nicer-quark-expansion}
    Q_{a,b} = \sum_{\begin{subarray}{c}u+v+w=0 \\
        u\equiv v\equiv w\bmod 3
      \end{subarray}
    } \leg u3\,q^{(u^2+v^2+w^2)/18}\;\zeta^{-(au+bv+cw)/3}\,,
  \end{equation}
  and the proof for this follows by using the equivariant isomorphism
  from the lattice
  $\{(t,u,v,w)\in\Z^4\mid t\equiv u\equiv v\equiv w\bmod3,\;u+v+w=0\}$
  to the lattice $\{(l,m,n)\in\Z^3\mid l\equiv m\equiv n\bmod2\}$
  given by $(l,m,n)=-\frac13(t+2u,t+2v,t+2w)$.
\end{Remark}

\begin{proof}[Fourth proof]
  Using the formula~\eqref{eq:nicer-quark-expansion} we have to show
  \begin{equation*}
    \ve a^2\cdot\ve x^2-(\ve a \cdot \ve x)^2 \ge 0
    ,
  \end{equation*}
  where $\ve a = (a,b,c)$ and $\ve x = (u,v,w)/3$; but this is the
  Cauchy-Schwarz inequality.  Recall that the Cauchy-Schwarz
  inequality is strict unless $\ve a$ is a multiple of $\ve x$.  In
  other words, $Q_{a,b}$ is a cusp form if and only if
  $\ve a = (a,b,-(a+b))$ is never proportional to a vector
  $\ve x = (u,v,w)$ in $\Z^3$ with $u\equiv v\equiv w\bmod 3$.
\end{proof}

\begin{proof}[Fifth proof]
  As we shall see in Section~\ref{sec:families} we can obtain the
  theta quarks as {\em pullbacks} of the function $\vartheta_{A_2}$
  defined by the Macdonald identity (also known as Kac-Weyl
  denominator formula) for the affine Lie algebra with positive root
  system~$A_2$.  The theory of affine Kac-Moody algebras gives in
  particular a formula for the Fourier expansion of this function,
  which shows that the pullbacks are indeed holomorphic at infinity
  (see \cite{Macdonald}, \cite{KP}).
  More details will be given in Part III (see
  Example~\ref{A-n-example}), where we shall also give a new proof of
  the Macdonald identities which does not make any use of affine Lie
  algebras.
\end{proof}

\begin{proof}[Sixth proof]
  In Section~\ref{sec:Borcherds-products} we shall see that the
  function $\vartheta_{A_2}$, which the fifth proof is based on, is
  the first Fourier-Jacobi coefficient of a holomorphic Borcherds
  product (see~\eqref{eq:A-2-lift}), and hence its pullbacks to theta
  quarks are in particular holomorphic at infinity. For details we
  refer the reader to the proof of Theorem~\ref{thm:app-theorem} and
  the subsequent remark.
\end{proof}

\section{Other families of low weight}
\label{sec:Other-families-of-low-weight}

The series of theta quarks of the preceding section is not the only
infinite family of theta blocks of low weight. In fact, as we shall
see in Part~III there are infinitely many such families.  In this
section we discuss various of these families which have low
weight. More precisely we shall discuss families of weight~$1$, $3/2$
and $2$. Recall that a theta block of weight $k$ consists of~$N$
functions $\vartheta_a$ divided by $\eta^{N-k}$. If the character
is~$\varepsilon^h$ then $2N+2k\equiv h\bmod 24$, hence the length of
the theta block occurs in the arithmetic progression
\begin{equation*}
  N=-k+h/2+12d\quad(d=0,1,2\dots)
  .
\end{equation*}
In Table~\ref{tab:low-weight-tb} we list various families of theta
blocks of low weight. For systematic reasons, which shall become clear
in Part~III we included also the family $Q_{a,b}$ of the last section
and renamed the function $R_{a,b,c,d}$ of \eqref{eq:R-definition} to
$\mathfrak {A}_{4,a,b,c,d}$.

Most remarkable is the series $\mathfrak{A}_{4;a,b,c,d}$, which, for
given $a,b,c,d$, yields a Jacobi form in $J_{2,m}$ with index
\begin{multline*}
  m=a^2+(a+b)^2+(a+b+c)^2+(a+b+c+d)^2\\+b^2+(b+c)^2+(b+c+d)^2+c^2+(c+d)^2+d^2.
\end{multline*}
In particular we have $\mathfrak{A}_{4;1,1,1,2}$ in $J_{2,37}$. The
latter space is one-dimens\-ional and contains only one cusp form,
which is in fact the cusp form of smallest index in weight $2$ with
trivial character. The first few coefficients of this cusp form were
computed laboriously in~\cite[p.145]{Eichler-Zagier}. Here
$\mathfrak{A}_{4;1,1,1,2}$ provides a closed formula.

\begin{table}[th]
  \begin{adjustwidth}{-.5cm}{-.5cm}
    \centering
    \caption{Families of theta blocks of low weight.}
    \newcolumntype{L}{>{$}l<{$}} \newcolumntype{R}{>{$}r<{$}}
    \begin{tabular}{lRL}
      \toprule
      Wt&\text{Char.}&\text{Family}\\
      \midrule
      \multirow{3}{*}{$1$}
        &8&Q_{a,b}=\mathfrak{A}_{2;a,b}=\mathfrak{D}_{2;a,b}=\eta^{-1} \, \vt { a } \, \vt { a + b } \, \vt { b }\\
        &10&\mathfrak{B}_{2;a,b}=\mathfrak{C}_{2;a,b}=\eta^{-2} \, \vt { a } \, \vt { a + b } \, \vt { a + 2 b } \, \vt { b }\\
        &14&\mathfrak{G}_{2,a,b}=\eta^{-4} \, \vt { a } \, \vt { 3 a + b } \, \vt { 3 a + 2 b } \, \vt { 2 a + b } \, \vt { a + b } \, \vt { b }\\
      \midrule
      \multirow{3}{*}{$\frac 32$}
        &15&\mathfrak{A}_{3;a,b,c}=\mathfrak{D}_{3;a,b,c}=\eta^{-3} \, \vt { a } \, \vt { a + b } \, \vt { a + b + c } \, \vt { b } \, \vt { b + c } \, \vt { c }\\
        &21&\mathfrak{B}_{3;a,b,c}=\eta^{-6} \, \vt { a } \, \vt { a + b } \, \vt { a + 2 b + 2 c } \, \vt { a + b + c } \, \vt { a + b + 2 c } \, \vt { b } \, \vt { b + c } \, \vt { b + 2 c } \, \vt { c }\\
        &21&\mathfrak{C}_{3;a,b,c}=\eta^{-6} \, \vt { a } \, \vt { 2 a + 2 b + c } \, \vt { a + b } \, \vt { a + 2 b + c } \, \vt { a + b + c } \, \vt { b } \, \vt { 2 b + c } \, \vt { b + c } \, \vt { c }\\
      \midrule
      \multirow{11}{*}{$2$}
        &0&\mathfrak{A}_{4;a,b,c,d}=\eta^{-6} \, \vt { a } \, \vt { a + b } \, \vt { a + b + c } \, \vt { a + b + c + d } \, \vt { b } \, \vt { b + c } \, \vt { b + c + d } \, \vt { c } \, \vt { c + d } \, \vt { d }\\
        &12&\mathfrak{B}_{4;a,b,c,d}=\eta^{-12} \, \vt { a } \, \vt { a + b } \, \vt { a + 2 b + 2 c + 2 d } \, \vt { a + b + c } \, \vt { a + b + 2 c + 2 d } \, \vt { a + b + c + d }\\&&\qquad\cdot \, \vt { a + b + c + 2 d } \, \vt { b } \, \vt { b + c } \, \vt { b + 2 c + 2 d } \, \vt { b + c + d } \, \vt { b + c + 2 d } \, \vt { c } \, \vt { c + d } \, \vt { c + 2 d } \, \vt { d }\\
        &12&\mathfrak{C}_{4;a,b,c,d}=\eta^{-12} \, \vt { a } \, \vt { 2 a + 2 b + 2 c + d } \, \vt { a + b } \, \vt { a + 2 b + 2 c + d } \, \vt { a + b + c } \, \vt { a + b + 2 c + d }\\&&\qquad\cdot \, \vt { a + b + c + d } \, \vt { b } \, \vt { 2 b + 2 c + d } \, \vt { b + c } \, \vt { b + 2 c + d } \, \vt { b + c + d } \, \vt { c } \, \vt { 2 c + d } \, \vt { c + d } \, \vt { d }\\
        &4&\mathfrak{D}_{4;a,b,c,d}=\eta^{-8} \, \vt { a } \, \vt { a + b } \, \vt { a + 2 b + c + d } \, \vt { a + b + c } \, \vt { a + b + c + d } \, \vt { a + b + d } \, \vt { b } \, \vt { b + c }\\&&\qquad\cdot \, \vt { b + c + d } \, \vt { b + d } \, \vt { c } \, \vt { d }\\
        &4&\mathfrak{F}_{4;a,b,c,d}=\eta^{-20} \, \vt { a } \, \vt { 2 a + 3 b + 4 c + 2 d } \, \vt { a + b } \, \vt { a + 3 b + 4 c + 2 d } \, \vt { a + 2 b + 2 c }\\&&\qquad\cdot \, \vt { a + 2 b + 4 c + 2 d } \, \vt { a + 2 b + 3 c + d } \, \vt { a + 2 b + 3 c + 2 d } \, \vt { a + 2 b + 2 c + d } \, \vt { a + 2 b + 2 c + 2 d }\\&&\qquad\cdot \, \vt { a + b + c } \, \vt { a + b + 2 c } \, \vt { a + b + 2 c + d } \, \vt { a + b + 2 c + 2 d } \, \vt { a + b + c + d } \, \vt { b } \, \vt { b + c } \, \vt { b + 2 c }\\&&\qquad\cdot  \, \vt { b + 2 c + d } \, \vt { b + 2 c + 2 d } \, \vt { b + c + d } \, \vt { c } \, \vt { c + d } \, \vt { d }\\
      \bottomrule
    \end{tabular}
    \label{tab:low-weight-tb}
  \end{adjustwidth}
\end{table}

We leave it as an exercise to the reader to verify that the given
families are indeed holomorphic at infinity. In principle this can be
done along the lines of the first two proofs for the family of theta
quarks as in the preceding section. However, for weights $3/2$ and $2$,
this becomes rather tedious. A more conceptual proof will be given in
Part~III (cf. Theorem~\ref{thm:Kac-Weil-den-formula}), and the family
$\mathfrak{A}_{4;a,b,d}$ will be discussed in the next section as one
instance in a natural infinite collection of families of theta
blocks. Here we confine ourselves to the families
$\mathfrak{B}_{2;a,b}$ and~$\mathfrak{G}_{2,a,b}$.

\begin{Proposition}
  The function
  \begin{equation*}
    \mathfrak{B}_{2;a,b}
    =
    \frac {\vt { a } \, \vt { a + b } \, \vt { a + 2 b } \, \vt { b }}{\eta^2}
  \end{equation*}
  is a holomorphic Jacobi form of weight $1$ and (integral or half
  integral) index $3 a^2/2 + 3 a b + 3 b^2$ with
  character~$\eta^{10}$.  For coprime $a$ and $b$ it is a cusp form if
  and only if $a$ is odd or $3\nmid b(a+b)$.
\end{Proposition}

\begin{proof}
  We analyze the theta block $\vartheta_{\ve a}/\eta^2$ for
  $\ve a=(a,a+b,a+2b,b)$ (notation as in
  \S\ref{sec:Experimental-search-for-long-theta-blocks}).  According
  to Proposition~\ref{prop:hyperbolic-criterion}, we have to prove
  that
  \begin{equation*}
    S_{\ve a}(\ve n)
    \ge \frac 23 {\ve a}^2
  \end{equation*}
  for all $\ve n$ in $\ve 1+2\Z^4$.  For this we use the remark after
  Proposition~\ref{prop:dual-criterion}: The vectors
  $\ve u_1 = (0,1,-1,1)$ and $\ve u_2=(1,-1,0,1)$ and $\ve a$ are
  pairwise orthogonal, and $\ve u_1^2=\ve u_2^2=3$, and hence the
  claimed inequality follows.  We leave the proof of the cusp
  condition to the reader.
\end{proof}

\begin{Proposition}
  The function
  \begin{equation*}
    \mathfrak{G}_{2;a,b} =
    \frac{ \vt { a } \, \vt { 3 a + b } \, \vt { 3 a + 2 b } \, \vt { 2 a + b } \, \vt { a + b } \, \vt { b }}{\eta^4}
  \end{equation*}
  is a holomorphic Jacobi form of weight $1$ and index
  $4(3a^2+3ab+b^2)$ with character $\eta^{14}$.
\end{Proposition}

\begin{proof}
  We proceed as in the preceding proof.  Setting
  \begin{equation*}
    \ve a=(a,3a+b,3a+2b,2a+b,a+b,b)
    ,
  \end{equation*}
  we have to prove
  \begin{equation*}
    S_{\ve a}(\ve n) \ge \frac 43 \ve a^2
  \end{equation*}
  for all $\ve n$ in $\ve 1+2\Z^6$. For this we apply
  Proposition~\ref{prop:dual-criterion} to the vectors $\ve u_j$ given
  by
  \begin{equation*}
    \begin{pmatrix}
      \ve u_1\\ \ve u_2\\ \ve u_3\\ \ve u_4
    \end{pmatrix}
    =
    \left(
      \begin{array}{rrrrrr}
        1&-1&0&1&0&0\\
        1&0&0&0&-1&1\\
        0&1&-1&0&0&1\\
        1&0&0&-1&1&0
      \end{array}
    \right)
    .
  \end{equation*}
  It is quickly checked that they are orthogonal to $\ve a$, and that
  the Gram matrix $G=\left(\ve u_i\cdot\ve u_j\right)$ satisfies
  \begin{equation*}
    4G^{-1}
    =
    \left(
      \begin{array}{rrrr}
        2 & -1 & 1 & 0 \\
        -1 & 2 & -1 & 0 \\
        1 & -1 & 2 & 0 \\
        0 & 0 & 0 & \frac{4}{3}
      \end{array}
    \right)
    .
  \end{equation*}
  Using $\ve n\cdot \ve u_j\equiv 1\bmod 2$ for
  $\ve n\in \ve 1+2\Z^6$, we deduce from
  Proposition~\ref{prop:dual-criterion}
  \begin{equation*}
    S_{\ve a}(\ve n)/\ve a^2
    \ge
    \frac 14 \min_{\ve x \in \ve 1+2\Z^3} \ve x^tK\ve x + \frac 13
    ,
  \end{equation*}
  where
  \begin{equation*}
    K
    =
    \left(
      \begin{array}{rrr}
        2 & -1 & 1 \\
        -1 & 2 & -1 \\
        1 & -1 & 2
      \end{array}
    \right)
    .
  \end{equation*}
  The minimum in question must be an even integer (since $K$ is
  even). It must be $\ge 4$ since for odd $x,y,z$, we have
  $\frac 12(x,y,z)K(x,y,z)^t=x^2-xy+xz+y^2-yz+z^2\equiv 0\bmod 2$; in
  fact, it is $4$ as one sees for $x-y=z=1$. The desired estimate is
  now obvious.
\end{proof}

\FloatBarrier

\section{An infinite collection of families}
\label{sec:A4-family}

In the previous section we saw various infinite families of theta
blocks. In Part III we shall propose a general theory which explains
the existence of these families and generates even more. More
specifically, we shall associate an infinite family to every root
system. The infinite families which we shall propose in this section
turn out to be those attached to the root systems $A_n$. However, we
include this section in the hope that the reader might find it
profitable to study the latter families here using elementary arguments
without having to go through the details of the theory developed in
Part III.

For the rest of this section we fix an integer $n\ge 2$. For any
integral vector $\ve a = (a_0,\dots,a_n)$ with pairwise different
entries, we set
\begin{equation}
  \label{eq:A-n-series}
  \Theta_{\ve a}
  :=
  \eta^{-n(n-1)/2} \prod_{0\le i < j \le n} \vt {a_i-a_j}
  .
\end{equation}
Clearly, $\Theta_{\ve a}$ depends only on the coset of $\ve a$ in
$\Z^{n+1}/\Z\cdot \ve 1$, where ${\ve 1}=(1,1,\cdots,1)$. Moreover,
changing the signs of any entries or the order of the entries
of~$\ve a$ leaves $\Theta_{\ve a}$ invariant up to sign. The
assumption that the~$a_j$ are pairwise different ensures that
$\Theta_{\ve a}$ does not vanish identically. Note that for $n=2$, we
have
\begin{equation*}
  \Theta_{a+b,b,0}=\eta^{-1}\vartheta_{a}\vartheta_{a+c}\vartheta_{b}
  ,
\end{equation*}
which is the family of theta quarks, and similarly
\begin{equation*}
  \Theta_{\left(0,a,a+b,a+b+c,a+b+c+d\right)}
  =
  R_{a,b,c,d} 
  .
\end{equation*}

We also define a quadratic form $Q$ by
\begin{equation*}
  Q({\ve a}) := \frac 12 \sum_{0\le i < j \le n} (a_i-a_j)^2
  =
  \frac {n+1}2\left(\sum_{i=0}^n a_i^2\right)
  - \frac 12 \left(\sum_{i=0}^n a_i\right)^2
  .
\end{equation*}
Again we recognize that $Q(\ve a)$ depends only on
$\ve a \bmod \Z\cdot\ve 1$.

In this section we shall prove that the functions $\Theta_{\ve a}$,
with ${\ve a}$ as above a vector in $\Z^{n+1}$ having pairwise distinct
entries, are theta blocks.  More precisely, we shall prove:

\begin{Theorem}
  \label{thm:A-n-series}
  The function $\Theta_{\ve a}$ defined in~\eqref{eq:A-n-series} is a
  theta block of length $n(n+1)/2$ and weight~$n/2$. More precisely,
  $\Theta_{\ve a}$ belongs to the space
  $J_{\frac n2, Q({\ve a})}\big(\varepsilon^{n(n+2)}\big)$. In
  particular, if $n+1$ is relatively prime to~$6$, it belongs to
  $J_{\frac n2, Q({\ve a})}$.
\end{Theorem}

The first case where the character $\varepsilon^{n(n+2)}$ is trivial
occurs for $n=4$, when the $\Theta_{\ve a}$ define Jacobi forms in
$J_{2,Q(\ve a)}$. In fact, the family $\Theta_{\ve a}$ equals the
family $R_{a,b,c,d}=\mathfrak{A}_{4;a,b,c,d}$ mentioned in the
introduction and in Table~\ref{tab:low-weight-tb}.
There are at least three more infinite families (all of the
form six theta functions over $\eta^6$) which yield Jacobi forms of
weight~$2$ without character (see
Table~\ref{tab:theta-R-representations} in
Section~\ref{sec:Examples-constructed-from-root-systems}).

Finally, one may ask when $\Theta_{\ve a}$ is a cusp form. The answer,
whose proof can be found at the end of the proof of
Theorem~\ref{thm:A-n-series} below, is as follows.
\begin{Supplement}
  Let $g$ denote the gcd of the differences $a_i-a_j$. Then
  $\Theta_{\ve a}$ is a cusp form if and only if there exists
  $0\le i<j\le n$ such that $(a_i-a_j)/g$ is divisible by~$n+1$.
\end{Supplement}

Just as for the family $Q_{a,b}$ of theta quarks in
Theorem~\ref{thm:theta-quarks-Fourier-expansion}, one can describe the
Fourier expansion of $\Theta_{\ve a}$ in closed form.

\begin{Theorem}
  \label{thm:A-n-Fourier-expansion}
  For the theta block $\Theta_{\ve a}$ defined
  in~\eqref{eq:A-n-series}, one has
  \begin{equation*}
    \Theta_{\ve a}
    =
    \sum_{
      \begin{subarray}c
        \ve x \in (\frac n2 + \Z)^{n+1}\\
        \ve x\cdot \ve 1 = 0
      \end{subarray}
    }
    \sigma(\ve x)\, q^{\ve x^2/2(n+1)}\,\zeta^{\ve a\cdot \ve x}
    ,  
  \end{equation*}
  where $\sigma(x)=\sym{sig}(\pi)$ if there is a permutation $\pi$ of
  $\{0,\dots,n\}$ such that
  $\ve x \equiv -\frac n2\ve 1 + (\pi(0),\pi(1),\dots,\pi(n)) \bmod
  (n+1)\Z$, and $\sigma(x)=0$ otherwise.
\end{Theorem}

A proof of this identity will be given in a more general context in
Part~III. It can easily be deduced by from
Theorem~\ref{thm:McDonalds-identity} in Part III applied to the root
system~$A_n$. Alternatively it can also be obtained directly, without
referring to root systems, by restriction to one variable of a more
general identity for many variables discussed in
Theorem~\ref{thm:picture} below. More precisely, our identity is
obtained by applying Theorem~\ref{thm:picture} to (in the notations of
that theorem)
\begin{align*}
  &\lat L=\left(L, (\ve x,\ve y)\mapsto \tfrac {\ve x\cdot\ve y}{n+1}\right),
    \quad
    s:e_i-e_j \ (0\le i<j\le n),
  \\
  &G = \text{permutations of the entries of vectors in $L$},
  \\
  &w=\left(-\tfrac n2,-\tfrac n2+1,\dots,\tfrac n2\right)
    ,
\end{align*}
where $L$ denotes the lattice of all vectors $\ve x$ in $\Z^{n+1}$
which satisfy $\sum_{i=1}^{n+1}x_i=0$ and $x_i\equiv x_j\bmod (n+1)$
for all $0\le i,j\le n$, and where $e_i$ is the vector of length~$n+1$
with $+1$ at the $i$th place and~$0$ at all other places. To obtain
literally our theorem one has then, first of all, to replace the
variable $z \in \C\otimes L$ used in Theorem~\ref{thm:picture} by
$\ve a z$ where now $z$ runs through the complex numbers. Secondly,
one has to use that $\lat L$ is isometric to the lattice $\Z^{n+1}/\Z$
equipped with the quadratic form~$Q$ from the beginning of this
section via the map (from right to left)
$\ve a \mapsto (n+1)\ve a - (\ve a\cdot \ve 1)\ve1$. We leave the
details to the interested reader.

It might be amusing to look for a combinatorial proof of the identity
of Theorem~\ref{thm:A-n-Fourier-expansion} along the lines of the
proof of the special case
Theorem~\ref{thm:theta-quarks-Fourier-expansion}. We finally mention a
nice restatement of Theorem~\ref{thm:A-n-Fourier-expansion}, which is
as follows.

\begin{Theorem}
  \label{thm:A-n-Fourier-expansion-determinant}
  One has
  \begin{equation}
    \label{eq:determinant-identity}
    \Theta_{\ve a}(\tau,z)
    =
    \int_0^1 \det \left[
      \begin{matrix}
        \vartheta_{0}^\ast(\tau,za_0+w)&\cdots
        &\vartheta^\ast_{0}(\tau,za_n+w)
        \\
        \vdots&&\vdots
        \\
        \vartheta_{n}^\ast(\tau,za_0+w)&\cdots
        &\vartheta_{n}^\ast(\tau,za_n+w)
      \end{matrix}
    \right]\,dw
    ,
  \end{equation}
  where
  \begin{equation*}
    \vartheta_{j}^\ast
    =
    \sum_{s\in j-\frac {n}2 + (n+1)\Z}q^{\frac {s^2}{2(n+1)}}\zeta^s
    .
  \end{equation*}
\end{Theorem}

This is indeed merely a restatement of the preceding Theorem. To
recognize this, write the determinant after the integral in the form
\begin{multline*}
  \sum_{\pi} \sym{sig}(\pi) \prod_{j=0}^n
  \vartheta_{\pi(j)}^\ast(\tau,a_jz+w)
  \\
  = \sum_{\pi\in S_N} \sym{sig}(\pi) \prod_{j=0}^n \sum_{x_j \in
    \pi(j)-\frac n2 + (n+1)\Z} q^{{x_j^2}/{2(n+1)}}\,\e{(a_jz+w)x_j} ,
\end{multline*}
where $\pi$ runs through the group of permutations of $\{0,\dots,n\}$.
Writing the product as an $(n+1)$-ary theta series, and integrating in
$w$ from $0$ to $1$ yields the Fourier expansion of $\Theta_{\ve a}$
as given in Theorem~\ref{thm:A-n-Fourier-expansion}.

Note that~\eqref{eq:determinant-identity} suggests an
elementary proof. Namely, it is obvious that, for any fixed $\tau$,
the right-hand side $I_{\ve a}$ vanishes at the $(a_i-a_j)$-division
points of $\C/\Z\tau+\Z$ (as it should in view of the claimed identity
and the zeros of $\Theta_{\ve a}$).  Indeed, if we replace $z$ be
$(\tau\lambda+\mu)\mu/(a_i-a_j)$ with any integers~$\mu, \lambda$ then
the determinant on the right-hand side
of~\eqref{eq:determinant-identity} becomes zero since the $i$th and
$j$th row become equal up to multiplication by a constant (since
$\frac {a_i}{a_i-a_j}=\frac {a_j}{a_i-a_j}+1$).  Unfortunately, this
still does not prove that the divisors of $I_{\ve a}(\tau,\_)$ and
$\Theta_{\ve a}(\tau,\_)$, viewed as theta functions of the elliptic
curve $\C/\Z\tau+\Z$, coincide; for this we would have to consider
also multiplicities. However, if we could prove that the divisors
coincide (or at least one is contained in the other) and that
$I_{\ve a}$ is also in
$J_{\frac n2, Q({\ve a})}\big(\varepsilon^{n(n+2)}\big)$ (note that
the transformation law with respect to $z\mapsto z+\lambda \tau+\nu$
with integral $\lambda,\mu$ is obvious) then we could conclude that
$I_{\ve a}$ and $\Theta_{\ve a}$ are equal up to multiplication by a
holomorphic modular function $f$ of weight $0$ on $\SL$. Comparing the
non-zero terms with lowest $q$-power, i.e.~verifying
\begin{multline*}
  q^{-(n(n-1)/48}\prod_{0\le i<j\le n}
  q^{1/8}\zeta^{(a_i-a_j)/2}\left(1-\zeta^{-(a_i-a_j)/2}\right)
  \\
  = q^{w^2/2(n+1)}\sum_\pi \sym{sig}(\pi)\zeta^{ w \cdot \left(
      a_{\pi(0)},\dots,a_{\pi(n)} \right)} ,
\end{multline*} shows then that $f$ is also holomorphic at infinity,
whence constant (and equal to $1$).

\begin{proof}[Proof of Theorem~\ref{thm:A-n-series}]
  We have to show that
  \begin{equation*}
    f(x):=\sum_{0\le i<j\le n} B\left((a_i-a_j)x\right) \ge \frac {n(n-1)}{24}
  \end{equation*}
  for all $x$ in $\R$. For this we replace $a_ix$ by $x_i$
  ($i=0,\dots,n$) and show that, more generally,
  \begin{equation*}
    \sum_{0\le i<j\le n} B\left(x_i-x_j\right) \ge \frac {n(n-1)}{24}
  \end{equation*}
  for any $\ve x = (x_0,\dots,x_n)$ in $\R^{n+1}$.

  Since the function in question is symmetric and periodic in
  each variable, we can
  assume that $0\le x_0\le \cdots \le x_n\le 1$, in which case
  $B(x_i-x_j)=\frac 12\left(x_i-x_j+\frac 12\right)^2$ for
  $0\le i<j\le n$, so we need only find the minimum of
  \begin{equation*}
    S:=\sum_{0\le i<j\le n} \left(x_i-x_j+\frac 12\right)^2  
  \end{equation*}
  over $\R^{n+1}/\R\cdot\ve 1$. Restricting to $\ve x$ with
  $\sum_i x_i=0$ (i.e.~the orthogonal complement of $\R\cdot\ve 1$
  in~$\R^{n+1}$) and minimizing $S$ using Lagrange multipliers shows
  that $S$ assumes its local minima where the partial derivatives
  $\frac {\partial S}{\partial x_k}$ ($0\le k\le n$) are independent
  of~$k$. Since we have
  $\frac 12\frac {\partial S}{\partial x_k}=(n+1)x_i+\frac 12 (n-2k)$
  (for $\sum_i x_i=0$) the latter condition is
  $x_k=\frac 1{2(n+1)}(2k-n)$, and then
  \begin{equation*}
    S=\sum_{0\le i<j\le n} \left(\frac {i-j}{n+1} +\frac 12\right)^2
    =\frac {n(n-1)}{24}
    ,
  \end{equation*}
  which proves the theorem.

  Note that we the preceding proof also shows that
  $f(x) = \frac {n(n-1)}{24}$ if and only if the differences
  $(a_i-a_j)x$ are in $\frac 1{n+1}\Z$ but not integral
  ($0\le i<j\le n$).  From this the supplement to the theorem is
  obvious.
\end{proof}

\part*{Part III: General Theory}

\section{Infinite families and Jacobi forms of lattice index}
\label{sec:families}

In this section we describe a general principle for constructing
infinite families of theta blocks which are proper Jacobi forms. This
principle is summarized in Theorem~\ref{thm:picture}. As we shall see
in the next section, all of the infinite series of theta blocks that
we studied in the previous sections can in fact be obtained using this
principle. To explain our construction we need to consider a more
general type of Jacobi form, namely Jacobi forms whose index is a
lattice. We explain these in the following paragraphs before we state
the aforementioned construction. A more thorough theory of lattice
index Jacobi forms is developed in~\cite{Boylan-Skoruppa-Joli-1},
\cite{G}, \cite{CG} and various other articles. We recall here the
basics of the theory of Jacobi forms of lattice index as developed
in~\cite{Boylan-Skoruppa-Joli-1}.

Let $\lat L=(L,\beta)$ be an integral lattice. Hence $L$ is a free
$\Z$-module of, say, rank~$n$ and $\beta:L\times L\rightarrow \Z$ is a
symmetric non-degenerate bilinear form. If $U$ is a $\Z$-submodule of
full rank in $\Q\otimes L$ we denote by $U^\sharp$ its dual subgroup,
i.e.~the subgroup of all elements $y$ in $\Q\otimes L$ such that
$\beta(y,x)$ takes integral values for all $x\in U$. We shall use in
the following $\beta(x)=\frac 12 \beta(x,x)$. Note that~$\beta(x)$ is
not necessarily integral. If it is we call $\lat L$ {\em even},
otherwise {\em odd}.  In any case, the map $x\mapsto \beta(x)$ defines
an element of order $2$ in the dual group $\sym{Hom}(L,\Q/\Z)$ of $L$
that is trivial on $2L$. The kernel $\ev{L}$ of this homomorphism
defines an even sublattice of index~$2$ in~$L$. Since $\beta$ is
non-degenerate there exists an element $r$ in $\Q\otimes L$ such that
$\beta(x)\equiv \beta(r,x)\bmod\Z$ for $x$ in~$L$. We set
\begin{equation*}
  \spv L
  :=
  \big\{
  r \in \Q\otimes L:
  \beta(x)\equiv \beta(r,x)\bmod\Z
  \text{ for all $x$ in $L$}
  \big\}
  ,
\end{equation*}
and following the literature we call $\spv L$ on lattices the {\em
  shadow of $\lat L$}, and we call the elements of $\spv L$ {\em \Lvec
  s of~$\lat L$}.  Clearly, for an even $\lat L$, we have
$\spv L=\dual L$, and, for an odd $\lat L$, we have
$\dual{\ev L} = \dual L \cup \spv L$ (i.e.~$\spv L$ is the non-trivial
coset in $\dual{\ev L}/\dual L$).

Recall from Section~\ref{sec:review} that $\varepsilon^h$ denotes the
$\SL$-cocycle defined by $\varepsilon^h(A) = f(A\tau)/f(\tau)$, where
$f(\tau)$ denotes any (fixed) branch of the function
$\eta(\tau)^h$. By slight abuse of language we occasionally call the
multiplier system $\varepsilon^h$ a {\em character}.

Let $k$ and $h$ be rational numbers such that $k \equiv h/2 \bmod \Z$.

\begin{Definition}
  A {\em Jacobi form of weight~$k$, index~$\lat L$ and
    character~$\varepsilon^h$} is a holomorphic function
  $\phi(\tau,z)$ of a variable $\tau\in\HP$ and a variable
  $z\in \C\otimes L$ which satisfies the following properties:
  \begin{enumerate}
  \item[(i)] For all $A = \mat abcd$ in $\SL$ one has
    \begin{equation}
      \label{eq:A-transformation-rule}
      \phi\left(A\tau,\frac z{c\tau+d}\right) =
      \e{\frac{c\,\beta(z)}{c\tau+d}}\,(c\tau +
      d)^{k-h/2}\,\varepsilon^h(A)\,\phi(\tau,z) .
    \end{equation}
  \item[(ii)] For all $x,y \in L$ one has
    \begin{equation*}
      \phi(\tau,z+x\tau + y)
      =
      e\big( \beta(x+y) \big)\,e\big(-\tau \beta(x)-\beta(x,z)\big)
      \,\phi(\tau,z) .
    \end{equation*}
  \item[(iii)] The Fourier development of $\phi$ is of the form
    \begin{equation}
      \label{eq:FC}
      \phi(\tau,z)
      =
      \sum_{n\in \frac h{24}+\Z}
      \sum_{\begin{subarray}{c} r\in \spv{L}\\
          n \ge \beta(r)
        \end{subarray}} c(n,r) \, q^n \, e\big(\beta(r,z)\big) .
    \end{equation}
  \end{enumerate}
  The space of Jacobi forms of weight $k$, index $\lat L$ and
  character $\varepsilon^r$ is denoted by
  $J_{k,\lat L}\left(\varepsilon^h\right)$.
\end{Definition}
Note that the crucial point in~(iii) is the condition $n\ge \beta(r)$.
That~$\phi$ has Fourier expansion with $n$ and $r$ in the range
described by the first conditions below the sum signs holds true for
any holomorphic~$\phi(\tau,z)$ satisfying the transformation laws~(i)
and~(ii)~(as one easily sees by applying these transformations laws to
$\tau\mapsto \tau+1$ and, for all $\mu$ in $L$, to $z\mapsto
z+\mu$). Note also that the factor $e\big(\beta(x+y)\big)$ in~(ii)
defines a linear character of the group $L\times L$. It is trivial if
$\lat L$ is even. A priori, for the transformation formula~(ii), one
could consider also other characters of $L\times L$.  However, it can
be shown~\cite{Boylan-Skoruppa-Joli-1} that, for a character different
from the given one, there are no non-trivial functions satisfying~(i)
and~(ii).

Note also that $J_{k,\lat L}\left(\varepsilon^h\right)$ depends only
on the coset $h+24\Z$, as follows from
$\varepsilon^{h+24k}(A)=(c\tau+d)^{12k}\varepsilon^h(A)$
($A=\mat abcd$).

If we fix a $\Z$-basis $\{a_p\}$ for $L$ we can identify $L$ and
$\C\otimes L$ with $\Z^n$ and $\C^n$, respectively, and Jacobi forms
of index $\lat L$ can be considered as holomorphic functions on
$\HP\times \C^n$.  In fact, if $\lat L$ is an even lattice, so that
the Gram matrix $F=\frac 12\left(\beta(a_p,a_q)\right)$ is
half-integral, and if $h=0$, the space $J_{k,\lat L}(\varepsilon^h)$
then becomes what in the literature~\cite{Skoruppa:Critical-Weight} is
usually called the space of Jacobi forms of weight $k$ and matrix
index $F$ and which is denoted by $J_{k,F}$.  Moreover, if $\lat L$ is
of rank 1 with determinant $m=|L^\sharp/L|$, then
$J_{k,\lat L}\left(\varepsilon^h\right)$ is nothing other than the
space $J_{k,\frac m2}(\varepsilon^h)$ that was introduces in
Section~\ref{sec:review}.

There is a family of natural maps between all these spaces of Jacobi
forms.  Namely, if $\alpha:\lat L\rightarrow \lat M$ is an isometric
embedding then the application
$(\alpha^*\phi)(\tau,z)=\phi(\tau, \alpha z)$ defines a map
\begin{equation}
  \label{eq:the-so-useful-pullback}
  \alpha^*:
  J_{k,\lat M}(\varepsilon^h) \rightarrow J_{k,\lat L}(\varepsilon^h).
\end{equation}
This follows immediately from the definition of our Jacobi forms.

There are two particular cases where such embeddings are of special
interest for our considerations. The fist case occurs when a lattice
$\lat L=(L,\beta)$ can be isometrically embedded into the lattice
$\lat \Z^N := (\Z^N,\cdot)$ (where the dot denotes the standard scalar
product of column vectors).  Such an embedding permits to construct
Jacobi forms of index $\lat L$ in a simple way.  Namely, let
$\alpha_j$ be the coordinate functions of this embedding, so that
$\beta(x,x)=\sum_j\alpha_j(x)^2$.  Then
\begin{equation*}
  \prod_{j=1}^N\vartheta\left(\tau,\alpha_j(z)\right) \in J_{\frac
    N2,\lat L}(\varepsilon^{3N}) .
\end{equation*}
Vice versa, if such a product defines a Jacobi form of index $\lat L$
then necessarily $\beta(x,x)=\sum_j\alpha_j(x)^2$, and the $\alpha_j$
define an isometric embedding of $\lat L$ into $\lat \Z^N$.

The other interesting embedding is of the form
\begin{equation*}
  s_x:(\Z,(u,v)\mapsto muv)\rightarrow \lat L=(L,\beta), \quad s_x(u)\mapsto ux ,
\end{equation*}
where $x$ is a non-zero element in $L$ and $m=\beta(x,x)$. Here we
obtain maps
\begin{equation*}
  s_x^*: J_{k,\lat L}\left(\varepsilon^h\right)
  \rightarrow
  J_{k,\frac m2}\left(\varepsilon^h\right), \quad
  \phi(\tau,z) \mapsto \phi(\tau,xw)\quad (w\in\C).
\end{equation*}

In fact, all the families of theta blocks that we found so far are of
the form $\{s_x^*\phi\}_{x\in L}$ for suitable lattices $\lat L$ and
special Jacobi forms $\phi$
in~$J_{k,\lat L}\left(\varepsilon^h\right)$. Moreover, these special
Jacobi forms $\phi$ are always obtained via the first construction,
i.e.~via an embedding of $\lat L$ into $\lat \Z^N$ for a suitable
$N$. In all these examples the weight $k$ of the special Jacobi form
equals $n/2$, where $n$ is the rank of $\lat L$. This is due to the
fact that in those cases we can divide by a power of $\eta$.  In
general a division by a power of $\eta$ will not yield a proper Jacobi
form since the condition~(iii) in the definition of Jacobi forms is
not invariant under such a division. However, a special situation
which makes such a division possible, and which applies to all our
examples, is described by Theorem~\ref{thm:picture} below.

For the statement of the theorem we need some preparations. By a {\em
  eutactic star (of rank $N$) on a lattice $\lat L = (L,\beta)$} we
understand a family $s$ of non-zero vectors $s_j$ in $\dual L$
($1\le j\le N$) such that
\begin{equation*}
  x
  =
  \sum_{j=1}^N \beta(s_j,x)s_j
\end{equation*}
for all $x$ in $\Q\otimes L$.  For a eutactic star $s$, one has
\begin{equation*}
  \beta(x,x)=\sum_j\beta(s_j,x)^2
\end{equation*}
for all $x$, i.e.~the map
$x\mapsto \big(\beta(s_1,x),\dots,\beta(s_N,x)\big)$ defines an
isometric embedding $\alpha_s:\lat L\rightarrow \lat \Z^N$. Vice
versa, if $\alpha$ is such an embedding, then, since $\beta$ is
non-degenerate, there exist vectors $s_j$ such that the $j$th
coordinate function of $\alpha$ is given by $\beta(s_j,x)$. It is easy
to show that the family $s_j$ (omitting the possible zero vectors) is
a eutactic star.

For a eutactic star $s$ on $\lat L$, we set
\begin{equation*}
  \vartheta_s(\tau,z)
  =
  \prod_{j=1}^N \vartheta\big(\tau,\beta(s_j,z)\big)
  \quad
  (z\in\C\otimes L)
  .
\end{equation*}
From our previous discussion we know that the function $\vartheta_s$
defines a non-zero (holomorphic) Jacobi form of weight $N/2$ and index
$\lat L$.  We are interested to find eutactic stars $s$ such that the
$\vartheta_s$ can be divided by a high power of $\eta$ and still
remains holomorphic at infinity (i.e.~satisfies the condition
$n \ge \beta(r)$ in the Fourier expansion~(iii) in the definition of
Jacobi forms).  It is not hard to see that the weight of a non-zero
Jacobi form of index $\lat L$ which has rank $n$ is $\ge n/2$. Thus the
highest power of $\eta$ by which we are allowed to divide
$\vartheta_s$ is $\eta^{N-n}$. We shall not discuss here the
question of determining the exact power but refer the reader
to~\cite{Boylan-Skoruppa-Joli-1}.  Instead we describe here one
situation where $\eta^{(n-N)/2}\vartheta_s(\tau,z)$ is in fact a
holomorphic Jacobi form.

For this let G be a subgroup of the orthogonal group $O(\lat L)$ that
leaves $s$ invariant up to signs, i.e. such that for each $g$ in $G$
there exists a permutation $\sigma$ of the indices $1 \le j \le N$ and
signs $\epsilon_j\in\{\pm 1\}$ such that $gs_j = \epsilon_j s_{\sigma(j)}$
for all $j$.  We set
\begin{equation*}
  \sym{sn}(g)=\prod_j \epsilon_j
  .
\end{equation*}
Note that $\sym{sn}(g)$ does not depend on the choice of $\sigma$.  It
follows that $g\mapsto\sym{sn}(g)$ defines a linear character
$\sym{sn}:G\rightarrow\{\pm 1\}$.

The group $G$ acts naturally on ${\spv L}/\ev L$.  We call the
eutactic star $s$ {\em $G$-extremal on $\lat L$} if there is exactly
one $G$-orbit in ${\spv L}/\ev L$ whose elements have their
stabilizers in the kernel of~$\sym{sn}$.

\begin{Theorem}
  \label{thm:picture}
  Let $\lat L=(L,\beta)$ be an integral lattice of rank~$n$, let $s$
  be a $G$-extremal eutactic star of rank~$N$ on $\lat L$. Then there is a
  constant $\gamma$ and a vector $w$ in~$\spv L$ such that
  \begin{equation}
    \label{eq:product-identity}
    \eta^{n-N} \prod_{j=1}^N\vartheta\big(\tau,\beta(s_j,z)\big)
    =
    \gamma
    \sum_{x\in w+\ev L} q^{\beta(x)}\sum_{g\in G}\sym{sn}(g)\,e\big(\beta(gx,z)\big)
    .
  \end{equation}
  In particular, the product on the left defines an element of the
  space of Jacobi forms~$J_{n/2,\lat L}(\varepsilon^{n+2N})$.
\end{Theorem}

\begin{Remark}
  Let $x$ be an element of $\R\otimes L$ such that
  $\beta(s_j,x)\not=0$ for all~$j$. (Such $x$ exist since the $s_j$
  span $\R\otimes L$ and therefore cannot be contained in any
  hyperplane.) The identity~\eqref{eq:product-identity} then holds
  true with $w$ replaced by
  \begin{equation*}
    w_0
    =\frac12\left(\epsilon_1 s_1
      + \epsilon_2 s_2
      + \cdots
      + \epsilon_N  s_N\right)
    ,
  \end{equation*}
  where $\epsilon_j$ denotes the sign of~$\beta(s_j,x)$.  Indeed,
  comparing the coefficients of the smallest $q$-power on both sides
  of~\eqref{eq:product-identity} one finds that
  \begin{equation*}
    \prod_{j=1}^N\left(e\left(\tfrac12\beta(s_j,z)\right) - e\left(\tfrac 12\beta(-s_j,z)\right)\right)
    =
    \gamma
    \sum_{x,g}
    \sym{sn}(g) e\left(\beta(g(w+x),z)\right)
    ,
  \end{equation*}
  where the sum on the right is over all $g$ in~$G$ and all $x$
  in~$\ev L$ such that $\beta(w+x)=(n+2N)/24$.  The left-hand side
  equals the sum $\sum_v \pm e\left(\beta(v,z)\right)$, where $v$ runs
  through all vectors $v$ of the form
  $v=v_\sigma=\frac12\sum_{j=1}^N \sigma_j s_j$ with
  $\sigma_j=\pm1$. From this we see that we can replace~$w$ by any
  $v_{\sigma_0}$ among these $v$ which is different from~$0$ and
  different from all $v_{\sigma}$ with $\sigma\not=\sigma_0$. But
  $w_0=v_{\epsilon}$ is such a $v_{\sigma_0}$ since $\beta(w_0,x)>0$
  and $\beta(w_0-v_\sigma,x)>0$ for all $\sigma\not=\epsilon$.

  Note also that it follows that $q^{\beta(w_0)}$ is the smallest
  $q$-power occurring on both sides of~~\eqref{eq:product-identity}. In
  other words
  \begin{equation*}
    \beta(w_0)=\frac {n+2N}{24}
    .
  \end{equation*}
\end{Remark}
 
\begin{proof}[Proof of Theorem~\ref{thm:picture}]
  As before denote the product on the left-hand side of the claimed
  identity (without the $\eta$-power) by $\vartheta_s$. It is clear
  that $\vartheta_s$ is an element of
  $J_{N/2,\lat L}(\varepsilon^{3N})$. However, $\vartheta_s$ satisfies
  in addition $g^*\vartheta_s = \sym{sn}(g)\,\vartheta_s$ for all $g$
  in $G$, as follows from the very definition of $\sym{sn}$ and the
  identity $\vartheta(\tau,-z) = -\vartheta(\tau,z)$.

  For an integer $h$ and for $k$ in~$\frac h2+\Z$, let
  $V_k(\varepsilon^h)$ be the subspace of all Jacobi forms $\phi$ in
  $J_{k,\lat L}(\varepsilon^h)$ that satisfy
  \begin{equation}
    \label{eq:G-invariance}
    g^*\phi = \sym{sn}(g)\,\phi
    \quad
    \text{for all }g\in G
    .
  \end{equation}

  Denote the function on the right-hand side of the claimed
  identity~\eqref{eq:product-identity} by $\phi_s$.
  We shall show in a moment that, for all integers~$h$ and all $k$
  in~$\frac h2 +\Z$, we have
  \begin{equation}
    \label{eq:V_k}
    V_k(\varepsilon^h)
    =
    M_{k-(r+n)/2}\,\eta^r \phi_s
    ,
  \end{equation}
  where $0\le r<24$, $h\equiv r+n+2N\bmod 24$. Here, for any $l$, we
  use $M_l$ for the space of elliptic modular forms of weight $l$ on
  $\SL$ (which is trivial unless $l$ is an even integer).

  But then the claimed identity~\eqref{eq:product-identity} is
  immediate.  Namely, from~\eqref{eq:V_k}, we deduce
  $\vartheta_s = f\eta^r\phi_s$ for some modular form $f$ of level
  one. If $f$ had a zero at a point $\tau_0$ in the upper half-plane,
  then $\vartheta_s(\tau_0,z)$ would vanish identically as function
  of~$z$.  However, this is impossible as the product expansion for
  $\vartheta(\tau,z)$ shows.  We conclude that~$f$ must itself be a
  power of $\eta$. Comparing weights then proves the claimed formula.

  Note that we used here only that the left-hand side of
  in~\eqref{eq:V_k} is contained in the right-hand side. It follows
  from $\vartheta_s = \eta^{n-N}\phi_s$ that $\phi_s$ is an element
  of~$V_{n/2}(\varepsilon^{n+2N})$, whence that the right-hand side
  of~\eqref{eq:V_k} is contained in the left-hand side.
  
  It remains to prove that the left-hand side of~\eqref{eq:V_k} is
  contained in the right-hand side.  Applying the transformation
  law~(ii) for Jacobi forms to $z\mapsto z+x\tau$ ($x\in L$) we
  obtain, for the Fourier coefficients $c(n,r)$ of a Jacobi form
  $\phi$ in $J_{k,\lat L}(\varepsilon^h)$, the identities
  \begin{equation*}
    c\big(n+\beta(r+x)-\beta(r),r+x\big)
    =
    c(n,r)\,e\big(\beta(x)\big)
    .
  \end{equation*}
  Hence, if we set
  \begin{equation*}
    C(D,r):=c\big(D+\beta(r),r\big)
    ,
  \end{equation*}
  then $C(D,r+x) = C(D,r)\,e\big(\beta(x)\big)$ for all~$x$ in~$L$. In
  particular, we recognize that $r\mapsto C(D,r)$, for fixed $D$,
  factors through a map on~$\spv L/\ev L$.
  
  Now assume that $\phi$ is contained in the left-hand side
  of~\eqref{eq:V_k}. Then $g^*\phi=\sym{sn}(g)\phi$ for all $g$, from
  which we deduce
  \begin{equation*}
    C(D,g^{-1}r) = \sym{sn}(g)\,C(D,r)
    .
  \end{equation*}
  Since $s$ is extremal this implies $C(D,r)=0$ unless the stabilizer
  of $r~+~\ev L$ in~$G$ is contained in the kernel of~$\sym{sn}$. By
  assumption there is exactly one $G$-orbit in~$\spv L/\ev L$ whose
  elements have stabilizer in the kernel of~$\sym{sn}$. Let $w+\ev L$
  an element of this orbit. The Fourier expansion~(iii) of~$\phi$ can
  then be written in the form
  \begin{multline*}
    \phi(\tau,z) = \sum_{r\in \spv{L}} \sum_{\begin{subarray}{c} D\in
        -\beta(r) + \frac h{24}+\Z\\ D\ge 0
      \end{subarray}} C(D,r) \, q^{D+\beta(r)} \,
    e\big(\beta(r,z)\big)\\
    =\nu\sum_{g\in G, x\in L} \sum_{\begin{subarray}{c} D\in -\beta(w)
        + \frac h{24}+\Z\\ D\ge 0
      \end{subarray}} \sym{sn}(g)\,C(D,w) \, q^{D+\beta(w+x)} \,
    e\big(\beta(g(w+x),z)\big) ,
  \end{multline*}
  where $1/\nu$ is the order of the stabilizer of $w+\ev L$ in the
  group $G$.  We therefore find $\phi=f\,\phi_s$, where
  $f=\nu\sum_{D} C(D,w)\,q^D$. From the usual theory of transformation
  laws for theta functions one can easily deduce that $\phi_s$ defines
  an element of $J_{n/2,\lat L}(\varepsilon^{n+2N})$ (for details we
  refer the reader to~\cite{Boylan-Skoruppa-Joli-1}). It follows that
  $f$ is a modular form on $\SL$ of weight $k-n/2$ with multiplier
  system $\varepsilon^r$. But this space of modular forms equals
  $M_{k-(r+n)/2}\eta^r$, which proves that $\phi$ lies in the the
  right-hand side of~\eqref{eq:V_k}. This proves the theorem.
\end{proof}

\begin{Example}[Jacobi triple product identity]
  \label{ex:JTPI}
  The simplest non-trivial example for the situation described in
  Theorem~\ref{thm:picture} is given by the eutactic star $s$ on
  $\lat \Z$ consisting of the single vector $s_1=1$ in $\Z$.  Here $s$
  is $G$-extremal, where $G$ is generated by $[-1]$ (multiplication by
  $-1$), and where $\sym{sn}([-1])=-1$. The discriminant module of
  $\dual{\ev \Z}/\ev \Z = (\frac 12 \Z)/2\Z$ decomposes into the three
  $G$-orbits $\{\red {{1/2}}, \red {{{3}/2}} \}$, $\{\red 1\}$ and
  $\{\red 0\}$ (where $\red x$ denotes the coset of $x$ modulo
  $2\Z$). Only the stabilizer of the first one is trivial. In this
  case the resulting identity\eqref{eq:product-identity} takes the
  form~\eqref{eq:theta}, which is the Jacobi triple product identity.
\end{Example}

\section{Examples constructed from root systems}
\label{sec:Examples-constructed-from-root-systems}

The theorem of the preceding section described a general principle for
constructing from special lattices Jacobi forms in several variables
that generate infinite families of holomorphic theta blocks
(i.e.~theta blocks that are holomorphic at infinity).  In this
section, we show that there are indeed infinitely many lattices to
which the theorem can be applied, namely, lattices constructed from
root systems. Example~\ref{ex:JTPI} is the most basic example for this
theory. The corresponding infinite families of theta blocks which will
arise from our construction in fact include all the examples of
families that were introduced in the previous sections.

The main result of this section can be summarized as follows.

\begin{Theorem}
  \label{thm:Kac-Weil-den-formula}
  Let $R$ be a root system\footnote{All {\em root systems} considered
    here are to be understood in the strict sense
    (see~\cite[\S~9.2]{Humphreys}), i.e.~any root system can be
    partitioned into the union of pairwise orthogonal sets each of
    which is a root system in the Euclidean space generated by its
    elements and as such isomorphic to one of the irreducible root
    systems $A_n$, $B_n$, $C_n$, $D_n$, $E_6$, $E_7$, $E_8$, $F_4$,
    $G_2$.} of dimension $n$, let~$R^+$ be a system of positive roots
  of~$R$ and let $F$ denote the subset of simple roots in $R^+$.  For
  $r$ in $R^+$ and $f$ in $F$, let $\gamma_{r,f}$ be the (non-negative)
  integers such that $r=\sum_{f\in F}\gamma_{r,f}f$. The function
  \begin{equation*}
    \vartheta_R(\tau,z)
    :=
    \eta(\tau)^{n-|R^+|}
    \prod_{r\in R^+}\vartheta\big(\tau,\sum_{f\in F}\gamma_{r,f}z_f\big)
    .
  \end{equation*}
  ($\tau\in\HP$, $z=\{z_f\}_{f\in F} \in \C^{F}$) defines a Jacobi
  form in $J_{n/2,\lat R}\big(\varepsilon^{n+2N}\big)$.  Here the
  lattice $\lat R$ equals $\Z^F$ equipped with the quadratic form
  $Q(z):=\frac12\sum_{r\in R^+}\big(\sum_f\gamma_{r,f} z_f\big)^2$.
\end{Theorem}

\begin{table}[tbh]
  \centering
  \caption{\Small%
    The Jacobi form $\vt R$ associated to the irreducible root system
    $R$ consists of $|R^+|$ many $\vartheta$'s multiplied by
    $\eta^{-\nu}$ and has weight $k$ and character~$\varepsilon^{l}$.}
  \begin{tabular}{>$l<$>$l<$>$l<$>$l<$>$l<$}
    \toprule
    R&|R^+|&\nu&k&l\\
    \midrule
    A_n&n(n+1)/2&n(n-1)/2&n/2&n(n+2)
    \\
    B_n&n^2&n(n-1)&n/2&n(2n+1)
    \\
    C_n&n^2&n(n-1)&n/2&n(2n+1)
    \\
    D_n&n(n-1)&n(n-2)&n/2&n(2n-1)
    \\
    E_6&36&30&3&6
    \\
    E_7&63&56&7/2&13
    \\
    E_8&120&114&4&8
    \\
    F_4&24&20&2&4
    \\
    G_2&6&4&1&14
    \\
    \bottomrule
  \end{tabular}
  \label{tab:parameters-of-the-root-families}
\end{table}

\begin{Remark}
  1.  We remark that the matrix $C := \big(\gamma_{r,f}\big)$ that
  defines~$\vartheta_R$ does not depend on the choice of the set of
  positive roots (up to permutations of its rows or columns). Indeed,
  the Weyl group of $R$ acts transitively on the collection of
  possible sets of positive roots, at the same time permuting the
  respective subsets of simple roots. It is not difficult to calculate
  the matrix $C$ directly from the Dynkin diagram or Cartan matrix of
  $R$ (see e.g.~\cite[\S21.3]{Fulton-Harris}).

  2. Obviously it suffices to prove the theorem for irreducible root
  systems since for any two root systems $R$ and $R'$ with ambient
  Euclidean spaces $E$ and $E'$ one has
  $\vartheta_{R\oplus R'}=\vartheta_R\vartheta_{R'}$, where
  $R\oplus R'$ denotes the root system
  $\left(R\times\{0\}\right)\cup \left(\{0\}\times R'\right)$ in
  $E\times E'$.

  3. As already explained in the previous section every choice of
  integer vectors $a=\{a_f\}_{f\in F}$ such that $a_f\not=0$ for all
  $f$ yields a theta block
  \begin{equation*}
    \eta(\tau)^{n-N}
    \prod_{r\in R^+}\vartheta\big(\tau,w\sum_{f\in F}\gamma_{r,f}a_f\big)
    \in J_{n/2,Q(a)}\big(\varepsilon^{n+2N}\big)
  \end{equation*}
  in the variables $\tau,w$ in $\HP\times\C$. Note that we can even
  assume that $a_f>0$ for all~$f$. For this let $(\cdot,\cdot)$ denote the
  scalar product of the ambient Euclidean space $E$ of the root
  system~$R$. For a given~$a$ let $x$ be the element of~$E$ such that
  $(x,f)=a_f$ for all~$f$. It follows
  $\sum_f \gamma_{r,f} a_f = (x,r)$ for all $r$ in~$R^+$. The general
  theory of root systems shows that there is a $g$ in the Weyl group
  of $R$ which maps $x$ into the fundamental Weyl chamber, i.e.~such
  that $a_f' :=(gx,f)>0$ for all~$f$. But there is a permutation
  $r\mapsto r'$ of $R^+$ such that $gr=\pm r'$, and we have
  $\sum_f \gamma_{r',f} a_f' = \pm (gx,gr) = \pm (x,r)= \pm\sum_f
  \gamma_{r,f} a_f$. Therefore, $\{a_f\}_f$ and $\{a_f'\}$ yield the
  same theta block up to a sign.
\end{Remark}

\begin{Example}
  \label{A-n-example}
  The only root system of rank 1 is $A_1$, and we have
  $\vartheta_{A_1}=\vartheta$ (see Example~\ref{ex:JTPI}).  If we
  choose for $R$ the root system $A_n$, then any chosen simple roots
  $f_j$ ($1\le j\le n$) can be ordered such that the positive roots
  are the sums of consecutive roots $f_i+f_{i+1}+\cdots+f_j$
  ($1\le i\le j\le n$). Accordingly
  \begin{equation*}
    \vartheta_{A_n}(\tau,z)
    =
    \eta(\tau)^{-n(n-1)/2}\prod_{1\le i\le j\le n}\vartheta\big(\tau, z_i+\cdots+z_j)
    \in
    J_{n/2,\lat {A_n}}\big(\varepsilon^{n(n+2)}\big)
    .
  \end{equation*}
  (we write $z_i$ for $z_{f_i}$). For $n=2$, we obtain the function
  \begin{equation*}
    \vartheta_{A_2}(\tau,z)
    =\vartheta(\tau,z_1)\vartheta(\tau,z_2)\vartheta(\tau,z_1+z_2)/\eta(\tau)
    ,
  \end{equation*}
  which under specialization yields the infinite family of theta quarks.
\end{Example}

\begin{Example}
  The spaces of Jacobi forms $J_{2,m}$ are, for integral $m$, deeply
  connected to the arithmetic of the modular forms of weight~$2$ in
  $\Gamma_0(m)$ (see e.g.~\cite{Skoruppa-Zagier}). There are four
  infinite families of theta blocks of weight~$2$ with trivial
  character that we can deduce from the $\vartheta_R$, as is easily
  inferred from Table~\ref{tab:parameters-of-the-root-families}.
  These are the families of theta blocks associated to
  \begin{equation*}
    \vt {A_4},\ 
    \vt {G_2}\vt {B_2} = \vt {G_2}\vt {C_2},\ 
    \vartheta \vt {B_3},\ 
    \vartheta \vt {C_3} 
    .
  \end{equation*}
  (Recall that ${B_2}$ is isomorphic to ${C_2}$.) The members of these
  families consist in each case of~$10$ $\vartheta$'s over~$\eta^6$
  (see Table~\ref{tab:theta-R-representations}).
  
  \begin{table}[tbh]
    \centering
    \caption{\Small%
      The four infinite families $\vt R\left(\tau,(a,b,c,d)z\right)$
      of theta blocks of weight~$2$ and trivial character associated
      to root systems (We write $\vt n$ for the function
      $\vt {}(\tau,n z)$.)}
    \begin{tabular}{>$l<$>$l<$}
      \toprule
      R &\vartheta_R\left(\tau,(a,b,c,d)z\right)\\
      \midrule
      A_4&
           \eta^{-6} \, \vt { a } \, \vt { a + b } \, \vt { a + b + c } \, \vt { a + b + c + d } \, 
           \vt { b } \, \vt { b + c } \, \vt { b + c + d } \, \vt { c } \, \vt { c + d } \, \vt { d }
      \\
      G_2\oplus B_2&
                     \eta^{-6} \, \vt { a } \, \vt { 3 a + b } \, \vt { 3 a + 2 b } \, \vt {
                     2 a + b } \, \vt { a + b } \, \vt { b }

                     \, \vt { c } \, \vt { c + d } \, \vt { c + 2 d } \, \vt { d }
      \\
      A_1\oplus B_3&

                     \eta^{-6} \, \vt { a } \, \vt { b } \, \vt { b + c } \, \vt { b + 2 c + 2 d } \, \vt
                     { b + c + d } \, \vt { b + c + 2 d } \, \vt { c } \, \vt { c + d } \,
                     \vt { c + 2 d } \, \vt { d }
      \\
      A_1\oplus C_3&
                     \eta^{-6} \, \vt { a } \, \vt { b } \, \vt { 2 b + 2 c + d } \, \vt { b + c } \, \vt
                     { b + 2 c + d } \, \vt { b + c + d } \, \vt { c } \, \vt { 2 c + d } \,
                     \vt { c + d } \, \vt { d }
      \\
      \bottomrule
    \end{tabular}
    \label{tab:theta-R-representations}
  \end{table}
  
\end{Example}

We now come to the proof of Theorem~\ref{thm:Kac-Weil-den-formula}. In
the course of the proof we shall redefine $\vartheta_R$ and
$\underline{R}$, but shall eventually see that the new and old
definitions in fact define the same objects.  As already pointed out
in the remark after the theorem we can assume without loss of
generality that the root system $R$ is irreducible.

Let $R$ be an irreducible root system of dimension $n$, let $R^+$ be a
system of positive roots of~$R$, let $N$ be the number of positive
roots, and let\footnote{If all roots have square length~$2$ then $h$
  coincides with the {\em Coxeter number} of the given root system,
  otherwise it is different.}
\begin{equation}
  \label{eq:Coxeter-number}
  h
  =
  \frac 1n \sum_{r\in R^+}(r,r)
  ,
\end{equation}
where $(\cdot,\cdot)$ denotes the Euclidean inner product of the
ambient Euclidean vector space $E$ of the root system $R$.  We let $W$
be the lattice
\begin{equation*}
  W
  =
  \big\{
  x\in E:
  (x,r)/h\in\Z\text{ for all $r\in R$}
  \big\}
  ,
\end{equation*}
and we set
\begin{equation}
  \label{eq:lat-R-definition}
  \lat R = \big(W, (\cdot,\cdot)/h\big)
  .
\end{equation}
The dual lattice $\dual W$ (with respect to the scalar product
$(\cdot,\cdot)/h$) equals the lattice $\Lambda$ spanned by the roots
$r$ in $R$.

\begin{Lemma}
  One has
  \begin{equation}
    \label{eq:root-embedding}
    h (z,z) = \sum_{r\in R^+} (r,z)^2
    ,
  \end{equation}
  for all $z$ in $E$.
\end{Lemma}

\begin{proof}
  The bilinear form $\beta(x,y):=\sum_{r\in R^+} (r,x)(r,y)$ is
  symmetric and positive definite (since the roots $r$ span
  $E$). There  hence exists an automorphism $\lambda$ of the real vector
  space $E$ such that $\beta(x,y)=(\lambda(x),y)$. The Weyl group of
  $R$ permutes the roots, and therefore $\beta(x,y)$ is invariant
  under the Weyl group. This in turn implies that $\lambda$ commutes
  with the elements of the Weyl group. However, the latter is known to
  act irreducibly on $E$ (see e.g.~\cite[\S10.4, Lemma
  B]{Humphreys}). By Schur's lemma we then conclude that $\lambda$ is
  multiplication by a scalar~$c$, whence $\beta(x,y)=c(x,y)$. It
  remains to show $c=h$. For this choose an orthonormal basis $e_j$ of
  $E$. Then, using Parseval's identity, we find
  $cn=\sum_j c(e_j,e_j)=\sum_{r,j} (r,e_j)^2=\sum_{r} (r,r)$, which
  proves the lemma.
\end{proof}

The lemma implies that $\lat R$ is an integral lattice and, in
particular, that $W$ is contained in its dual, which is $\Lambda$.

From~\eqref{eq:root-embedding} we immediately have available the
embedding $\lat R\rightarrow \lat\Z^{N}$ defined by
$z\mapsto \left((z,r_1),\dots,(z,r_N)\right)$, where $r_j$ runs
through $R^+$. In other words, $R^+$ is a eutactic star on $\lat
R$. The Weyl group $G$ of $R$ leaves $R$ invariant, and the character
$\sym{sn}$ considered in the preceding section associates to an
element $g$ in the Weyl group the number $(-1)^{\ell(g)}$, where
$\ell(g)$ is the length of $g$, i.e.~the number of roots in $R^+$ such
that $gr$ is negative.

\begin{Theorem}
  \label{eq:roots-are-extremal}
  The eutactic star $R^+$ on $\underline{R}$ is extremal with respect
  to the Weyl group~$G$ of~$R$.
\end{Theorem}

We shall prove this theorem in the next section.  We can now apply
Theorem~\ref{thm:picture} and the remark after the theorem to the
eutactic star $R^+$ on $W$ and conclude (leaving the computation of
the constant~$\gamma$ in Theorem~\ref{thm:picture} to the reader)
\begin{Theorem}
  \label{thm:McDonalds-identity}
  Let $R$ be an irreducible root system with a choice of positive
  roots $R^+$, and let $w$ be half the sum of the positive roots
  of~$R$. Then, in the notations of the preceding paragraphs, we have
  \begin{multline}
    \label{eq:root-theta}
    \vartheta_{R}(\tau,z):=\eta(\tau)^{n-N}\prod_{r\in R^+}\vartheta\big(\tau,(r,z)/h\big)\\
    = \sum_{x\in w+\ev W}q^{(x,x)/2h}\sum_{g\in G} \sym{sn}(g) \,
    e\big((gx,z)/h\big)
  \end{multline}
  for all $\tau$ is the upper half-plane and all $z$ in $\C\otimes W$.
  In particular the function $\vartheta_{R}$ defines a holomorphic
  Jacobi form in $J_{n/2,\lat R}(\varepsilon^{n+2N})$.
\end{Theorem}

\begin{Remark}
  1. Let $F$ be the set of simple roots in~$R^+$. For any $f$ in~$F$
  and $z$ in $\C\otimes W$, we set $z_f:=(f,z)/h$. The application
  $z\mapsto \left\{z_f\right\}_{f\in F}$ defines an isomorphism of
  $\C$-vector spaces $\C\otimes W\rightarrow\C^F$, which maps $W$ onto
  $\Z^F$. We have $(r,z)/h = \sum_{f\in F} \gamma_{r,f} z_f$ and
  accordingly (using~\eqref{eq:root-embedding})
  $(x,x)/2h=Q\left(\{x_f\}\right)$ with $\gamma_{r,f}$ and $Q$ as in
  Theorem~\ref{thm:Kac-Weil-den-formula}. Thus, under the application
  $z\mapsto \left\{z_f\right\}_{f\in F}$, the lattice $\underline {R}$
  and the function $\vartheta_R(\tau,z)$ take on the form described in
  Theorem~\ref{thm:Kac-Weil-den-formula}, which is therefore merely a
  weaker form of the preceding theorem.

  2.  The identities of the preceding theorem, read as identities
  between formal power series by replacing $e\big((gx,z)/h\big)$ by a
  formal variable $e^{gx}$, are known as {\em Macdonald
    identities}~\cite[(0.4))]{Macdonald}\footnote{To identify the
    identity of Theorem~\ref{thm:McDonalds-identity} for a given root
    system $R$ with Macdonald's identity (0.4) in~\cite{Macdonald} for
    the coroot system~$R^\vee$ (i.e.~the system $r^\vee = 2r/(r,r)$
    ($r\in R$) one needs the formula
    $h=(\alpha+\rho,\alpha+\rho)-(\rho,\rho)$, where $\rho$ and
    $\alpha$ are the Weyl vector and highest root of $R$ (see
    Lemma~\ref{eq:h-w-alpha-identity} for a proof).  Moreover, one needs
    to note that Macdonald's $\chi(\mu)$ is zero unless $\mu$ is in
    $\ev M$.}
\end{Remark}

\FloatBarrier

\section{Proof of Theorem~\ref{eq:roots-are-extremal}}

We continue the notations of the paragraphs before
Theorem~\ref{eq:roots-are-extremal}. In other words
\begin{itemize}
\item $R$ is an irreducible root system,
\item $R^+$ a fixed choice of positive roots,
\item $\Lambda$ the lattice spanned by its roots,
\item $G$ the Weyl group of $R$,
\item $h=\frac 1n\sum_{r\in R^+} (r,r)$,
  $w=\frac 12 \sum_{r\in R^+}r$, and
\item $W=h\dual \Lambda$,
  $\underline R =\left(W,(\cdot,\cdot)/h\right)$ (so that $\Lambda$
  becomes the dual of $\underline R$, i.e.~the dual of $W$ with
  respect to~$(\cdot,\cdot)/h$).
\end{itemize}
Note that $w$ is an element of $\spv W$. Indeed,
$(x,x)/h \equiv (2w,x)/h\bmod 2$ for all $x$ in~$W$, as follows
from~\eqref{eq:root-embedding}.  Moreover, let
\begin{itemize}
\item $\alpha$ be the highest root in~$R^+$,
\item $C$ the fundamental Weyl chamber associated to $R^+$, and
\item $r^\vee$, for any root $r$, the coroot of~$r$
  (i.e.~$r^\vee=2r/(r,r)$).
\end{itemize}

Theorem~\ref{eq:roots-are-extremal} is an immediate consequence of the
following lemma, which we shall prove below.

\begin{Lemma}
  \label{lem:fundamental-lemma}
  Let $v$ be any element in $\spv W$ which has minimal length among
  all elements in $v + \ev W$\footnote{In general there might be
    several elements of minimal length in a given coset
    in~$\spv W/\ev W$. For instance, $w$ has minimal length in
    $w+\ev W$ for any irreducible root system, but $(w-hf,w-hf)=(w,w)$
    and $hf\in \ev W$ for two of the six simple roots of $E_6$.}.
  Then:
  \begin{enumerate}
  \item $(\alpha,v) \le h$.
  \item If $(\alpha,v)=h$ then $v\equiv g_\alpha(v) \bmod \ev W$,
    where $g_\alpha$ is the reflection through the hyperplane
    perpendicular to~$\alpha$.
  \item If $(\alpha,v)<h$ and $v\in C$, then $v=w$.
  \end{enumerate}
\end{Lemma}

\begin{proof}[Proof of Theorem~\ref{eq:roots-are-extremal}]
Indeed, to prove Theorem~\ref{eq:roots-are-extremal}, let $v+\ev W$
be a class in $\spv W/\ev W$. We can assume that $v$ is a vector of
minimal length in its class and that $v$ is contained in the closure
$\overline C$ (since $E=\bigcup_{g\in G} g\overline C$). By the lemma
$(\alpha,v)\le h$, and either $v=w$, or else $v + \ev W$ is stabilized
by $g_\alpha$ or $v$ is contained in a wall of the Weyl chamber. In
the latter case $v$ is stabilized by the reflection through the
hyperplane containing the wall, which is perpendicular to some
fundamental root. Any reflection through a hyperplane has determinant
$-1$. But determinant and the character $g\mapsto (-1)^{\ell(g)}$
coincide for reflections through hyperplanes perpendicular to roots,
as follows from the fact that a reflection through the hyperplane
perpendicular to a fundamental root has length~$1$~\cite[\S10.2, Lemma
B]{Humphreys}, and that the Weyl group is generated by such
reflections.

It remains to prove that the Weyl group stabilizer of $w+\ev W$ is
contained in the kernel of $\sym{sn}$. For this note that
$\vartheta_R$ satisfies $g^*\vartheta_R=\sym{sn}(g)\vartheta_R$ ($g$
in $G$). Since $\vartheta_R$ is obviously different from zero its
Fourier development contains a Fourier coefficient
$C(D,x)\not=0$. Since $C\left(D,g(x)\right)=\sym{sn}(g)C(D,x)$ for all
$g$ in $G$ and $C(D,x)$ depends only on $x+\ev W$ we see that the
orbit of $x+\ev W$ is not stabilized by any $g$ of odd length. By what
we have seen $x+\ev W$ must then be in the orbit of $w+\ev W$. This
proves Theorem~\ref{eq:roots-are-extremal}.
\end{proof}

\begin{proof}[Proof of Lemma~\ref{lem:fundamental-lemma}]
  To prove the lemma let $f_i$ ($1\le i\le n$) be the simple roots of
  $R^+$, and let $\lambda_i$ be the dual basis of the basis $f_i^\vee$
  of~$E$.  One has $w=\lambda_1+\cdots+\lambda_n$ (see~\cite[\S13.3,
  Lemma~A]{Humphreys} for a short proof).

  To prove (i) note that by the very definition of root system
  $(\alpha^\vee,r)$ is integral for every root~$r$,
  i.e.~$h\alpha^\vee$ defines an element of $W$.  Moreover,
  $(\alpha^\vee,w)$ is integral too (since $w=\sum_i \lambda_i$, and
  since the $f_i^\vee$ are simple roots for the coroot system $r^\vee$
  ($r\in R$), so that in particular $\alpha^\vee$ is an integral
  linear combination of the $f_i^\vee$), so $h\alpha^\vee$ defines
  an element of $\ev W$.  Since $v$ has minimal length in its class
  $v+\ev W$ we have in particular
  $(v-h\alpha^\vee,v-h\alpha^\vee) \ge (v,v)$, i.e.
  $h\ge (\alpha,v)$.

  For (ii) note that $g(v)=v-(\alpha,v)\alpha^\vee$, whence
  $v-g(v)=h\alpha^\vee$.

  For (iii) we  now suppose that $v$ is in $C$ and
  \begin{equation*}
    h>(\alpha,v)
    .
  \end{equation*}
  By Lemma~\ref{eq:h-w-alpha-identity} below we have
  $h=\frac 12 (\alpha,\alpha) + (\alpha,w)$.  It follows that
  \begin{equation*}
    (\alpha^\vee,w) \ge (\alpha^\vee,v)
    ,
  \end{equation*}
  where we used that both sides of this inequality are integers (that
  $(\alpha^\vee,w)$ is an integer was proved above; but then the
  right-hand side is also integral since $v\in w+\Lambda$ and
  $(\alpha^\vee,\Lambda)\subseteq\Z$).

  Since $\alpha^\vee = \sum_i (\alpha^\vee,\lambda_i) f_i^\vee$,
  $w=\sum_i \lambda_i$, and $v=\sum_i (v,f_i^\vee)\lambda_i$ the last
  inequality can be written as
  \begin{equation*}
    \sum_i (\alpha^\vee,\lambda_i) \ge \sum_i (\alpha^\vee,\lambda_i)\,(v,f_i^\vee)  
  \end{equation*}
  But the $(v,f_i^\vee)$ are strictly positive (since $v\in C$) and
  integers (since $v\in w + \Lambda$, $(w,f_i^\vee)=1$ and
  $(\Lambda,f^\vee)\subseteq \Z$). Moreover, the
  $(\alpha^\vee,\lambda_i)$ are strictly positive (since
  $\alpha = \sum_{i} \alpha_if_i$ with non-negative integers
  $\alpha_i$, which are all strictly positive since $\alpha$ is the
  highest root, so that in particular $\alpha-f_i$ is still a linear
  combination in the $f_i$ with non-negative integers). The last
  inequality therefore implies $(v,f_i^\vee)=1$ for all $i$,
  i.e.~$v=w$.  This proves the lemma.
\end{proof}

  \begin{Lemma}
    \label{eq:h-w-alpha-identity}
    Let $\alpha$ be the highest root in $R^+$. Then
    \begin{equation*}
      h = \frac 12\left((\alpha+w,\alpha+w) - (w,w)\right)
      .
    \end{equation*}
  \end{Lemma}

\begin{proof}
  From~\eqref{eq:root-embedding} we obtain
  \begin{equation*}
    2h = \sum_{r\in R^+} (r,\alpha^\vee)(r,\alpha)
    .
  \end{equation*}
  For any positive root $r\not=\alpha$ not perpendicular to $\alpha$,
  one has $(r,\alpha^\vee)=1$.  (Since the highest root is a long root
  one has $(\alpha,\alpha)\ge (r,r)$, which implies
  $(r,\alpha^\vee)\le (r^\vee,\alpha)$. On the other hand, by the
  Cauchy-Schwartz inequality $(r,\alpha^\vee)(r^\vee,\alpha) < 4$, and
  since both scalar products are integers we find
  $(r,\alpha^\vee)=\pm1$. But $(r,\alpha^\vee)=-1$ would imply
  $s_\alpha(r)=r+\alpha$, contradicting the fact that $s_\alpha(r)$ is
  a root and $\alpha$ is the highest root.)  It follows that
  $2h=(\alpha,\alpha)+(y,\alpha)$, where $y$ is the sum over all
  positive roots which are not perpendicular to $\alpha$ (here we use
  also $(\alpha,\alpha^\vee)=2$). Obviously $(y,\alpha)=(2w,\alpha)$
  so that $2h=(\alpha,\alpha)+2(w,\alpha)$.  The claimed formula
  now becomes obvious.
\end{proof}

\section{Theta blocks of weight \texorpdfstring{$1/2$}{1/2} and
  weight \texorpdfstring{$1$}{1}}
\label{sec:small-weights}

It is possible to give a complete description of the Jacobi forms of
weight $1/2$ and weight~$1$ (and scalar index). A first description
of this kind was can be found in \cite{Skoruppa:Dissertation}, where
is was proved that there are essentially only two Jacobi forms of
weight~$1/2$, and that there is no non-zero Jacobi form of weight~$1$
and trivial character (see Theorem~\ref{thm:singular_weight-rank-1}
below). In~\cite{Skoruppa:Critical-Weight} and
in~\cite{Boylan-Skoruppa-Joli-2} these results were extended to
include Jacobi forms of weight 1 with arbitrary character.

The key for obtaining explicit formulas for Jacobi forms of weight
$1/2$ and weight $1$ is the {\em theta expansion} of a Jacobi form,
and the theory of Weil representations of $\SL$. To explain this,
let $\lat L =(L,\beta)$ be an integral positive definite lattice of
rank $n$.  Recall from Section~\ref{sec:families} that $\spv L$
denotes the shadow of $L$.  For any linear character $\lambda$ of
$\spv L \cup \dual L$ that continues the character
$x\mapsto \e{\beta(x)}$ of $L$ define a holomorphic function
$\vt {\lat L,\lambda}(\tau,z)$ of variables $\tau \in \HP$ and
$z \in \C\otimes L$ by
\begin{equation*}
  \vt {\lat L,\lambda}(\tau,z)
  =
  \sum_{
    r\in \spv L
  }
  \lambda(r)\,
  \e{\tau\beta(r)+\beta(r,z)}
  ,
\end{equation*}
and let $\Theta(\lat L)$ denote the complex space spanned by all the
$\vt {\lat L,\lambda}$, where~$\lambda$ runs through all these
characters.  Note that $\Theta(\lat L)$ has dimension $|\dual L/L|$.
It can be shown that the space $\Theta(\lat L)$ becomes a right
$\Mp$-module via the map
$(\alpha,\phi)\mapsto \phi|_{\lat L,n/2}\alpha$. Here $\Mp$ is the
usual twofold central extension of $\SL$ used in the theory of
elliptic modular forms of weight $1/2$ consisting of pairs
$\alpha=(A,w)$, where $A=\mat abcd$ is in $\SL$ and $w$ is one of
the two holomorphic roots of the function $c\tau+d$
($\tau\in\HP$). Moreover, $\phi|_{\lat L, n/2}\alpha$ is defined as
the right-hand side of~\eqref{eq:A-transformation-rule} with the
factor $(c\tau+d)^{k-h/2}\varepsilon^h(A)$ replaced by
$w(\tau)^{-n}$. That $\Theta(\lat L)$ with respect to the given
action is an $\Mp$-module is a well-known fact for
even~$\lat L$~\cite{Kloosterman-I}; for odd $\lat L$
see~\cite{Boylan-Skoruppa-Joli-1}. The representations associated to
the $\Mp$-modules $\Theta(\lat L)$ can be characterized purely
algebraically as a natural class of representations, which for even
$\lat L$ are known as {\em Weil representations of $\SL$}.

Every Jacobi form $\phi$ in~$J_{k,\lat L}(\varepsilon^h)$ has a {\em
  theta expansion}, i.e.~it can be written in the form
\begin{equation*}
  \phi(\tau,z)
  =
  \sum_{\lambda}h_\lambda(\tau)\,\vt {\lat L,\lambda}(\tau,z)
\end{equation*}
with holomorphic functions $h_\lambda$ and with $\lambda$ running
through the characters of $\spv L\cup \dual L$ whose restriction to
$L$ is $x\mapsto \e{\beta(x)}$. This follows immediately from the
considerations at the end of the proof of
Theorem~\ref{eq:product-identity}.  For integral $h$ and integral or
half-integral $k$, the $h_\lambda$ are modular forms of weight
$k-n/2$ on some congruence subgroup of $\SL$. More precisely, there
exists a natural number $N$ such that $h_\lambda$ in in
$M_{k-n/2}(4N)$, where the latter denotes the space of all
holomorphic functions $h$ on $\HP$ such that
$h(A\tau)=w(\tau)^{2k-n}h(\tau)$ for all $(A,w)$ in $\Gamma(4N)^*$
and, for each $(A,w)$ in $\Mp$ the function $h(A\tau)w(\tau)^{n-2k}$
is bounded in $\Im(\tau)\ge 1$. Here $\Gamma(4N)^*$ is the section
of $\Gamma(4N)$ in $\Mp$ consisting of all $(A,w)$, where
$w(\tau)=\theta(A\tau)/\theta(\tau)$ with
$\theta(\tau)=\sum_{r\in\Z}\e{\tau r^2}$.

Using the invariance of the $\Theta(\lat L)$ under $\Mp$, we can
reformulate the theta expansions of Jacobi forms of index~$\lat L$
as a natural isomorphism
\begin{equation}
  \label{eq:natural-isomorphism}
  J_{k,\lat L}(\varepsilon^h)
  \isom
  \big(
  M_{k-n/2}\otimes \Theta(\lat L)
  \big)(\varepsilon^h)
  .
\end{equation}
Here $M_{k-n/2}$ denotes the (infinite-dimensional) $\Mp$-module
generated by all spaces $M_{k-n/2}(4N)$ with the $\Mp$-action
$((A,w),h)\mapsto h(A\tau)w(\tau)^{n-2k}$. (It can be verified that
the groups $\Gamma^*(4N)$ are normal in $\Mp$, so that $M_{k-n/2}$
is indeed invariant under the given action).  Moreover, for an
$\Mp$-module $V$, we let $V(\varepsilon^h)$ denote the subspace of
all $v$ such that $\alpha.v=\varepsilon(\alpha)^hv$ (where
$\varepsilon$ denotes the linear character of $\Mp$ which, for
$\alpha=(A,w)$ is defined by
$\varepsilon(\alpha) =\eta(A\tau)/w(\tau)\eta(\tau)$).

For the {\em singular weight} of the index $\lat L$, i.e.~for the
weight $k=n/2$, we obtain in particular
\begin{equation*}
  J_{n/2,\lat L}(\varepsilon^h)
  \isom
  \Theta(\lat L)(\varepsilon^h)
  .
\end{equation*}
For lattices of rank one, i.e.~for the lattices
\begin{equation*}
  \lat \Z(m)=(\Z,x,y\mapsto mxy)
\end{equation*}
the $\Mp$-modules $\Theta(\lat \Z(m))$ were decomposed into
irreducible parts in \cite[Satz 1.8,
p.~22]{Skoruppa:Dissertation}\footnote{Actually, in loc.cit.~only
  the $\Mp$-modules $\Theta(\lat \Z(2m))$ were decomposed. However,
  it is quickly checked that $\Theta(\lat \Z(m))$ is a
  $\Mp$-submodule of $\Theta(\lat \Z(4m))$, which allows to infer
  the decomposition of the former from the latter.}. As corollary of
the results there the following was proved (see \cite[Beispiele,
p.~26--27]{Skoruppa:Dissertation}):

\begin{Theorem}(\cite{Skoruppa:Dissertation})
  \label{thm:singular_weight-rank-1}
  For any integer $m\ge 1$ and $0\le a <24$ one has
  $J_{{\frac 12},{\frac m2}}(\eta^h)=0$ unless, for some integer
  $d$, one has $(m,h)=(d^2,3)$ or $(m,h)=(3d^2,1)$. In the latter
  case one has\footnote{We use here $\sectheta_d$ for the function
    $\sectheta(\tau,dz)$, where $\sectheta$ is the quintuple product defined in~\eqref{eq:omega}.
  }
  \begin{equation*}
    J_{{\frac12},{\frac {d^2}2}}(\varepsilon^3) = \C\cdot\vt d,
    \qquad
    J_{{\frac12},{\frac {3d^2}2}}(\varepsilon) =
    \C\cdot\eta\,\sectheta_d
    .
  \end{equation*}
\end{Theorem}

For the {\em critical weight} of the index $\lat L$, i.e., for
$k=(n+1)/2$, the isomorphism~\eqref{eq:natural-isomorphism} involves
$M_{1/2}$. Based on a theorem of Serre and Stark a complete
decomposition of the $\Mp$-module $M_{1/2}$ into irreducible parts
was given in~\cite[Satz 5.2, p.101]{Skoruppa:Dissertation}. In
particular, one has, for any natural number~$N$,
\begin{equation*}
  M_{1/2}(4N)
  =
  \bigoplus_{\begin{subarray}{c}
      d|N\\
      N/d \text{ squarefree}
    \end{subarray}
  }
  \Theta^{\text{null}}\big(\lat \Z(2d)\big)
  ,
\end{equation*}
where the spaces on the right are the images of
$\Theta\big(\lat \Z(m)\big)$ under the map
$\vartheta\mapsto \vartheta(\tau,0)$.  As a result of these
consideration we obtain the
isomorphism~\cite{Boylan-Skoruppa-Joli-1}
\begin{equation}
  \label{eq:critical-weight-isomorphism}
  J_{\frac{n+1}2,\,\lat L}(\varepsilon^a)
  \isom
  \bigoplus_{\begin{subarray}{c}
      d|N\\
      N/d \text{ squarefree}
    \end{subarray}
  }
  p^*\Theta\big(Z(2d)\oplus\lat L\big)(\varepsilon^a)
  ,
\end{equation}
where, for any $m$, the map $p$ is the isometric embedding of
$\lat L$ into
$\lat Z(2m)\oplus\lat L=\big(\Z\oplus L, (x\oplus y,x'\oplus
y')\mapsto mxx'+\beta(y,y')\big)$ given by $y\mapsto 0\oplus y$, and
where $p^*$ is the pullback defined
in~\eqref{eq:the-so-useful-pullback}. Moreover, for $N$ one can take
any multiple of the level of $\lat L$ and $24$. Note that the spaces
on the right-hand side of~\eqref{eq:critical-weight-isomorphism} are
spaces of Jacobi forms of singular weight. Thus, Jacobi forms of
singular weight and rank one index are always pullbacks of Jacobi
forms of critical weight of index of rank $n+1$.

To make~\eqref{eq:critical-weight-isomorphism} explicit we would
need a description of the one-dimensional $\Mp$-submodules of
$\Theta(\lat L)$ for arbitrary~$L$. For lattices of rank~$1$ such a
description led to Theorem~\ref{thm:singular_weight-rank-1}.  In
general we do not know how to describe the one-dimensional
$\Mp$-submodules of~$\Theta(\lat L)$. However, for lattices of
rank~2 such a description has been found
in~\cite{Boylan-Skoruppa-Joli-2}. As a result it was possible to
prove the following theorem.

\begin{Theorem}(\cite{Boylan-Skoruppa-Joli-2}) Let $m$ be a positive
  integer, and for $h=2,4,6,8,10,14$ let $R$ and $\phi_R$ be the root
  system and Jacobi form as described in the row of~$h$ in
  Table~\ref{tab:weight_one}. With $\underline R$ denoting the lattice
  defined in Theorem~\ref{thm:Kac-Weil-den-formula}, one has the
  following:
  \begin{enumerate}
  \item For $h=4,6,8,10,14$ the space $J_{1,m}(\varepsilon^h)$ is
    spanned by the theta blocks $\phi_{R}(\tau, \ell z)$, where $\ell$
    runs through all elements of $\lat R$ with square length~$2m$.
  \item For $h=2$ the space $J_{1,m}(\varepsilon^2)$ contains the
    theta blocks $\phi_{R}(\tau, \ell z)$ ($\ell \in \lat R$ with
    square length $2m$), but is in general not spanned by them.
  \item For all other values of $h$ modulo $24$ one has
    $J_{1,m}(\varepsilon^h)=0$.
  \end{enumerate}
\end{Theorem}

\begin{table}[ht]
  \centering
  \caption{\Small%
    The six $\phi_R (\tau,z_1,z_2)$ which yield infinite families of
    theta blocks of weight~$1$ and character $\varepsilon^h$. ( We use
    $a$ and $b$ for the projections onto the first respectively second
    coordinate, and we write $\vt \lambda$ and $\sectheta_{\lambda}$
    for the functions $\vartheta\left(\tau,\lambda(z_1,z_2)\right)$
    and $\sectheta\left(\tau,\lambda(z_1,z_2)\right)$.)}
  \begin{tabular}{>$l<$>$l<$>$l<$}
    \toprule
    h & R &
            \phi_{R} \\
    \midrule
    2 & A_1\oplus A_1 &
                        \eta^2 \, \sectheta_a \, \sectheta_b
    \\
    4 & A_1\oplus A_1 &
                        \eta \, \vt a \, \sectheta_b
    \\
    6 & A_1\oplus A_1 &
                        \vt {a} \, \vt {b} \\
    8 & A_2 &
              \eta^{-1} \, \vt { a } \, \vt { a + b } \, \vt { b }
    \\
    10 & B_2 &
               \eta^{-2} \, \vt { a } \, \vt { a + b } \, \vt { a + 2 b } \, \vt { b }
    \\
    14 & G_2 &
               \eta^{-4} \, \vt { a } \, \vt { 3 a + b } \, \vt { 3 a + 2 b } \, \vt {
               2 a + b } \, \vt { a + b } \, \vt { b }
    \\
    \bottomrule
  \end{tabular}
  \label{tab:weight_one}
\end{table}

\begin{Remark}
  That $J_{1,m}=0$ for all $m$ was already proved in~\cite[Satz 6.1,
  p.~113]{Skoruppa:Dissertation}, and that
  $J_{1,m}(\varepsilon^{16})=0$ and the description of the spaces
  $J_{1,m}(\varepsilon^8)$ was shown
  in~\cite[Thms.~11,~12]{Skoruppa:Critical-Weight}.
\end{Remark}

\FloatBarrier

\part*{Part IV: Applications and open questions}

\section{Borcherds products and theta blocks}
\label{sec:Borcherds-products}

In \cite{GN}, the authors proposed a construction of certain
Borcherds products using Jacobi forms. The general theory of Borcherds
products was developed in~\cite{Borcherds95} and~\cite{Borcherds98}.
We recall the construction of~\cite{GN}.

For any positive integer $m$ there is a level-raising Hecke type
operator $V_{m}: J_{k,t}^{!} \to J_{k,mt}^{!}$ (see~\cite[page
41]{Eichler-Zagier}).  For any $\phi$ in $J_{k,t}^{!}$, the Fourier
coefficients $c_\phi(n,r)$ of $\phi$ and $c_{\phi|V_m}(n,r)$ of
$\phi|V_m$ are related by the formula
\begin{equation*}
  c_{\phi | V_m}( n,r)= \sum_{ d \mid n,r,m } d^{k-1}
  c_\phi\left( \frac{nm}{d^2},\, \frac{r}{d}\right),
\end{equation*}
the sum being over all common positive divisors of $n$, $r$, and $m$.
We consider the following series
\begin{equation}
  \label{eq:lift}
  \sym{Lift}(\phi)(\tau, z,\omega)
  :=
  c_\phi(0,0)\, G_k(\tau)
  +
  \sum_{m \ge 1} \phi | V_m(\tau,z)\,\ex( mt \omega)
  ,
\end{equation}
where $G_k$ for even $k\ge 2$ denotes the Eisenstein series
\begin{equation*}
  G_k(\tau)
  =
  \frac 12 \zeta(1-k)
  + \sum_{n\ge 1}\sigma_{k-1}(n)\,e(\tau),
\end{equation*}
and where $G_k=0$ for all other $k$ (note that $c_\phi(0,0)=0$ for
$k=2$). If $\phi$ is holomorphic at infinity this series is convergent
for all $\mat \tau z z \omega$ with positive definite imaginary part
and defines an element of the space $M_k(\Gamma_t)$ of Siegel modular
forms of weight $k$ and genus $2$ on the paramodular group $\Gamma_t$
(see~\cite{G94}).  The map $\sym{Lift}$ for $t=1$ is the lifting that
was used by Maass to prove the original Saito-Kurokawa conjecture and
that was discussed in detail in~\cite[\S6]{Eichler-Zagier}.

If $\phi$ has weight~$0$, i.e.~if $\phi$ is in $J_{0,t}^{!}$, and if
$c_\phi(n,r)$ is an integer for all $n,r$ with $4tn-r^2 \le 0$, then
we define (see~\cite[eq.~(2.7)]{GN})
\begin{equation}
  \label{eq:B-lift}
  B(\phi)(\tau,z,\omega)
  =
  \Th(\phi)\,
  \ex(C\omega) \, \exp\bigl( -{\rm Lift}({\phi})\bigr)
  ,
\end{equation}
where $C=\frac 12 \sum_{l>0} c_\phi(0,l)\,l$ and where
\begin{equation*}
  \Th(\phi)
  =
  \eta(\tau)^{c_\phi(0,0)}
  \prod_{l \ge 1}
  \left( \frac{\vartheta_{l}(\tau,z)}{\eta(\tau)} \right)^{c_\phi(0,l)}
  .
\end{equation*}
A straightforward computation shows
\begin{equation*}
  B(\phi)
  =
  \Th(\phi)\,
  p^C \,
  \prod_{
    \begin{subarray}c
      n,l,m\in\Z\\
      m\ge 1 
    \end{subarray}
  }
  \left(1-q^n\zeta^lp^{tm}\right)^{c_\phi(nm,l)}
  ,
\end{equation*}
where $p=\ex(\omega)$.  This product and the series~\eqref{eq:B-lift}
converges in a connected subdomain of the Siegel upper half-plane, it
can be meromorphically continued to the whole upper half-plane, and
then it becomes meromorphic modular form of weight $c_\phi(0,0)/2$ for
the paramodular group $\Gamma_t$ with known character and divisor
(see~\cite[Thm.~2.1]{GN}). In fact, the function~$B(\phi)$ is a
Borcherds product in the neighborhood of a one-dimensional cusp of the
paramodular group.
  
As an immediate corollary we obtain the following proposition.
  
\begin{Theorem}
  \label{prop:val-1}
  Let $\phi$ be an element of $J_{0,t}^!$ with Fourier coefficients
  $c_\phi(n,l)$, and assume that $c_\phi(n,r)$ is an integer for any
  $n,r$ such that $4tn-r^2\le 0$, and that the sums
  $\sum_{d\ge 1} c_\phi(d^2n, dl)$ are non-negative for all $n,l$ with
  $4nt-l^2<0$. Then the theta quotient $\Th(\phi)$ is a Jacobi form
  (i.e.~holomorphic in $\HP\times\C$ and at infinity).
\end{Theorem}
\begin{Remark}
  In particular, if $c_\phi(n,r) \ge 0$ for all $n,r$ with $4tn-r^2<0$
  then the theta quotient of the proposition (which is then in fact a
  theta block) is holomorphic at infinity.
\end{Remark}
\begin{proof}
  As suggested by the product expansion the multiplicities of all
  irreducible rational quadratic divisors (Humbert surfaces) of
  $B(\phi)$ is given by the sums in the proposition (for a proof
  see~\cite[Thm.~2.1]{GN}), whence $B(\phi)$ is holomorphic. The theta
  quotient~$\Th(\phi)$ is the first non-zero Fourier-Jacobi
  coefficient of $B(\phi)$, and is hence holomorphic (including
  infinity).
\end{proof}

\begin{Example}[The first Jacobi and paramodular cusp form of
  weight~$3$]
  \label{ex:val-2}
  Consider the weak Jacobi form of weight~$0$ and index~$13$
  \begin{equation*}
    \varphi_{0,13}
    =\frac{\vartheta_2 \vartheta_3 \vartheta_4}{\vartheta^3}
    =\zeta^{\pm 3}+3\zeta^{\pm 2}+5\zeta^{\pm 1}+6 + O(q)
  \end{equation*}
  where $\zeta^{\pm m}=\zeta^{m}+\zeta^{- m}$.  Note that the
  $q^0$-part contains in fact all non-zero coefficients $c_\phi(n,r)$
  with $52 n-r^2<0$.  Indeed, the product
  \begin{equation*}
    \eta^5\varphi_{0,13}
    =
    \bigl(\eta{\vartheta_2}/{\vartheta}\bigr)^2
    \bigl(\eta^2{\vartheta_3}/{\vartheta}\bigr)
    \bigl(\eta{\vartheta_4}/{\vartheta_2}\bigr)
  \end{equation*}
  defines a generalized theta block of weight $5/2$ (see
  Corollary~\ref{cor:gen-theta-block-criterion}), and it is even
  holomorphic at infinity (in fact, each of the three factors in the
  last formula is already holomorphic at infinity as can be easily
  checked). It follows that $52n-r^2\ge -\frac {52\cdot 5}{24}$ for
  any non-zero Fourier coefficient $c_\phi(n,r)$ of
  $\varphi_{0,13}$. But $c_\phi(n,r)$ depends only on $52n-r^2$ and
  $\pm r\mod 26$. Analyzing the residues $r^2$ modulo $52$ we see that
  all non-zero Fourier coefficients with $52n-r^2<0$ are given by
  $c_\phi(0,3)=1$, $c_\phi(0,2)=3$ and $c_\phi(0,1)=5$.  In
  particular, the Borcherds product $B(\varphi_{0,13})$ is
  holomorphic. Its first Fourier-Jacobi coefficient
  \begin{equation*}
    \varphi_{3,13}
    =
    {\vartheta_3\vartheta_2^3\vartheta^5}/{\eta^3}=Q_{1,1}^2Q_{1,2}.
  \end{equation*}
  turns out to be a product of three theta quarks. It is among all
  Jacobi cusp form of weight $3$ the one with smallest index.  The
  divisor of $B(\varphi_{0,13})$ is a sum of Humbert modular surfaces:
  \begin{equation*}
    \sym{div} B(\varphi_{0,13})
    = \Gamma_{13}\langle z=1/3 \rangle
    +
    3\cdot\Gamma_{13}\langle z=1/2\rangle
    +
    9\cdot\Gamma_{13}\langle z=0 \rangle
    .
  \end{equation*}
  This is a part of the divisor of
  ${\rm Lift} (\varphi_{3,13})\in S_3(\Gamma_{13})$ because the
  lifting procedure preserves the divisor of the lifted Jacobi
  form~$\phi$ (more precisely, of the function
  $\phi(\tau,z)\,\ex(tw)$).  The quotient
  ${\rm Lift} (\varphi_{3,13})/B(\frac{\vartheta_2 \vartheta_3
    \vartheta_4}{\vartheta^3})$ is holomorphic on the the Siegel upper
  half-plane and $\Gamma_{13}$-invariant, whence constant by the
  K\"ocher principle.  Comparing the first Fourier-Jacobi coefficients
  we have two formulas for the first paramodular cusp form
  $F_3^{(13)}$ of (canonical) weight $3$:
  \begin{equation*}
    F_3^{(13)}={\rm Lift} ({\vartheta_3\vartheta_2^3\vartheta^5}/{\eta^3})= B({\vartheta_2
      \vartheta_3 \vartheta_4}/{\vartheta^3})\in S_3(\Gamma_{13}).
  \end{equation*}
  Moreover, we note that
  $F_{3}^{(13)}dZ\in H^{3,0}(\bar{\mathcal A}_{1,13})$ defines a
  canonical differential form on any smooth compact model of the
  moduli space of $(1,13)$-polarized abelian surfaces. The formula
  above determines the main part of its canonical divisor.
\end{Example}

An effective construction of weak Jacobi forms satisfying the
assumptions of Proposition~\ref{prop:val-1} was proposed in~\cite{GPY}
\footnote{The article~\cite{GPY} is partly based on results of the
  current paper.}.

\begin{Theorem}[\protect{\cite{GPY}}]
  \label{thm:GPY-theorem}
  Let $\Theta$ be a theta block of weight~$k>0$ and integral index~$t$
  and trivial character which has integral vanishing order~$v>0$ in
  $q$.  If $v$ is odd assume that $\Theta$ is holomorphic at
  infinity. Then $\psi = (-1)^{v} \frac{ \Theta | V_2 }{ \Theta }$ is
  a weakly holomorphic Jacobi form of weight~$0$ and index~$t$ which
  satisfies the assumptions of
  Theorem~\ref{prop:val-1}.
\end{Theorem}

\begin{Remark}
  The proof of the theorem can be found loc.cit., but the educated
  reader can also read it off from the formula
  \begin{equation*}
    \Theta | V_2(\tau,z)
    =
    4\Theta(2\tau,2z) +\tfrac{1}2
    \Theta(\tfrac{\tau}2,z) +\tfrac{1}2
    \Theta(\tfrac{\tau+1}2,z)
    .  
  \end{equation*}
  This formula shows in particular that the $q$-order of~$\Theta|V_2$
  equals $\lceil v/2\rceil$ and the $q$-order of $\psi$ equals
  ~$-\lfloor v/2 \rfloor$. For $v=1$, the function $\psi$ defines in
  particular a weak Jacobi form.
\end{Remark}

The above example of the paramodular form $F_3^{(13)}$ of weight $3$
is the blue print for the following conjecture.

\begin{Conjecture}[\cite{GPY}]
  Let $\Theta\in J_{k,t}$ be a theta block with trivial character and
  with order of vanishing~$1$ in $q$.  Then
  ${\rm Lift}(\Theta)= B(-\frac{\Theta|V_2}\Theta)$.
\end{Conjecture}

The next theorem shows that a similar conjecture might be true for
theta blocks with order in $q$ smaller than~$1$.

\begin{Theorem}
  \label{thm:app-theorem}
  \leavevmode
  \begin{enumerate}
  \item Let $\Theta=\prod_{i=1}^3Q_{a_i,b_i}\in J_{3,d}$ be a product
    of three theta quarks and set
    \begin{equation*}
      \phi:=-\frac{\Theta|V_2}{\Theta}
      .
    \end{equation*}
    (Note that Theorem~\ref{thm:GPY-theorem} implies that~$\phi$ is in
    $J_{0,d}^!$ satisfying the assumptions of
    Theorem~\ref{prop:val-1}.)  Then
    \begin{equation*}
      \sym{Lift}(\Theta)=B(\phi)\in M_3(\Gamma_d).
    \end{equation*}
    This is a cusp form if at least one of the three theta quarks is a
    cusp form.

  \item Let $Q_{a,b}\in J_{3,t}(\varepsilon^8)$ be an arbitrary theta
    quark, and set
    \begin{equation*}
      \phi
      :=
      -\frac{Q_{a,b}|_{\varepsilon^8}V_4}{Q_{a,b}}
      .
    \end{equation*}
    Then $\phi$ is in $J_{0,3t}^!$, and one has
    \begin{equation*}
      {\rm Lift}_{\varepsilon^8}(Q_{a,b})=B(\phi)
      .
    \end{equation*}
    This function defines a modular form of weight one with respect to
    the paramodular group $\Gamma_{3t}$ with a character $\chi_3$ of
    order!$3$.
  \end{enumerate}
\end{Theorem}

The lift $\sym{Lift}_{\varepsilon^8}(\phi)$ for a Jacobi form $\phi$
with character $\varepsilon^8$, which was introduced in ~\cite[Theorem
1.12]{GN}, is defined as in~\eqref{eq:lift} but with $\phi|V_m$
replaced by $\phi|_{\varepsilon^8}V_m$ and the summation restricted
to all $m\equiv 1\bmod 3$. The operator $|_{\varepsilon^8}V_m$ is a
Hecke type operator defined similar to $|V_m$, whose precise
definition is given in~\cite[1.12]{GN}.  The identity
$\sym{Lift}(\Theta)=B(\phi)$ of~(1) was already stated
in~\cite[Thm~8.3]{GPY}, and is in fact a corollary
of~\cite[Thm.~5.6]{G18}).

\begin{proof}[Proof of Theorem~\ref{thm:app-theorem}]
  We consider the function $\vartheta_{A_2}$ of
  Theorem~\ref{thm:McDonalds-identity} associated to the root
  system~$A_2$, which defines a Jacobi form of weight~$1$ with
  character~$\varepsilon^8$ and with lattice index~$\lat{A_2}$
  (defined in~\eqref{eq:lat-R-definition}).  Recall that this is the
  function occurring in the Macdonald identity (also known as
  denominator function) of the affine Kac-Moody algebra
  $\widehat A_2$. The lattice $\lat{A_2}$ is a root lattice of type
  $A_2$ (i.e.~its vectors of square length~$2$ span it and form a
  root system $\Phi$ of type~$A_2$). If $f_1$, $f_2$ are primitive
  roots of~$A_2$, then $\lambda_1=f_1$ and $\lambda_2=-f_2$ are
  fundamental weights of $\Phi$ (i.e.~$\lambda_1$, $\lambda_2$ form
  a dual basis of a set of primitive roots of $\Phi$). For $z$ in
  $\C\otimes \lat{A_2}$, we set $z_i=(z,\lambda_i)$, where $(\_,\_)$
  denotes the bilinear form of $\lat{A_2}$.  Then $\vartheta_{A_2}$
  becomes (see Example~\ref{A-n-example})
  \begin{equation*}
    \vartheta_{A_2}(\tau,z)
    = 
    \vartheta(\tau,z_1)\vartheta(\tau,z_2-z_1)\vartheta(\tau,z_2)
    /{\eta(\tau)}
    .
  \end{equation*}
  We note that $3\lambda_i$ is a reflective\footnote{A vector $x$ of
    a lattice~$\lat L=(L,\beta)$ is called {\em reflective} if the
    reflection $\sigma_x(y)=y-2x\beta(x,y)/\beta(x,x)$ defines an
    isometry of~$\lat L$.} vector of square length~$6$ in
  $\lat {A_2}$, and the divisor $z_i=0$ is the hyperplane of the
  reflection $\sigma_{\lambda_i}$.
  
  We need also the Jacobi form
  \begin{equation*}
    \vartheta_{3A_2}(\tau,Z)
    =
    \vartheta_{A_2}(\tau,Z_1)
    \vartheta_{A_2}(\tau,Z_2)
    \vartheta_{A_2}(\tau,Z_3) \in J_{3,3\lat {A_2}}
    .
  \end{equation*}
  Here $3\lat {A_2}$ stands for the threefold orthogonal sum of
  $\lat {A_2}$; moreover, we identify
  $\C\otimes \left(3\lat {A_2}\right)$ with the threefold direct sum
  of $\C\otimes \lat {A_2}$ and write accordingly any $Z$ in the
  former space as $Z=(Z_1,Z_2,Z_3)$ with $Z_i$ in
  $\C\otimes \lat {A_2}$.  We remark that $\vartheta_{3A_2}$
  coincides also with the Jacobi form associated by
  Theorem~\ref{thm:McDonalds-identity} to the threefold orthogonal
  sum of the root system~$A_2$.

  To the Jacobi forms $\vartheta_{A_2}$ and $\vartheta_{3A_2}$ we
  can apply a lifting construction similar to (\ref{eq:lift}) above
  (see~\cite{G94} and~\cite{CG} for the case of Jacobi forms with
  characters).  We obtain orthogonal modular forms
  \begin{gather*}
    \sym{Lift}_{\varepsilon^8}(\vartheta_{A_2}) \in
    M_1(\widetilde{\Orth}^+(2\lat U\oplus \lat {A_2}(-3)),\chi_3)
    ,\\
    \sym{Lift}(\vartheta_{3A_2})\in M_3(\widetilde{\Orth}^+(2\lat
    U\oplus 3 \lat {A_2}(-1))) ,
  \end{gather*}
  where $\lat U$ is the even unimodular lattice of signature
  $(1,1)$, $\lat {A_2}(n)$ is the lattice obtained from $\lat {A_2}$
  by renormalizing its bilinear form by the factor $n$ and
  $\widetilde{\Orth}^+(\dots)$ denote the stable orthogonal groups
  of the given lattices (which are both of signature $(2,8)$).  In
  both cases the (reflective) divisor of the lifted Jacobi form
  induces a subdivisor of the lifting.
  
  We construct a Jacobi form of weight $0$ with index $3\lat{A_2}$
  using again the operator $V_2$; we set
  \begin{equation}
    \label{eq:phi-3-A-2}
    \begin{aligned}
      \varphi_{0,3A_2}(\tau,Z) &:= -\frac{\vartheta_{3A_2}|
        V_2}{\vartheta_{3A_2}}= \sum_{
        \begin{subarray}c n\ge 0
          \\
          \ell\in \dual{3\lat{A_2}}
        \end{subarray}
      } c(n,\ell) \, q^n\,\ex((\ell,Z))
      \\
      &= 6+\sum_{i=1,3,5} (\zeta_i^{\pm 1}+\zeta_{i+1}^{\pm 1} +
      (\zeta_i\zeta_{i+1}^{-1})^{\pm 1}) + O(q)
    \end{aligned}
  \end{equation}
  where $\zeta_i=\exp(2\pi i z_i)$,
  $\zeta_i^{\pm 1}=\zeta_i+\zeta_i^{-1}$. The action of $V_2$
  on~$\varphi_{0,3A_2}$ is given by
  \begin{equation*}
    \vartheta_{3A_2}|\, V_2= 4\vartheta_{3A_2}(2\tau,2Z) +\tfrac{1}2
    \vartheta_{3A_2}(\tfrac{\tau}2,Z) +\tfrac{1}2
    \vartheta_{3A_2}(\tfrac{\tau+1}2,Z)
    .
  \end{equation*}
  Using this formula and the explicitly known divisor of
  $\vartheta_{3A_2}$ one verifies that
  $\varphi_{0,\,3A_2}\in J^{\mathrm{weak}}_{0,3\lat{A_2}}$, where
  the superscript {\em weak} means that $c(n,\ell)=0$ unless
  $n\ge 0$.  For any $\ell\in \dual{3\lat{A_2}}$, we therefore have
  $c(n,\ell) = 0$ unless
  \begin{equation*}
    2n-(\ell,\ell)\ge -\min_{v\in \ell+3A_2}
    (v,v)\ge -2.
  \end{equation*}
  This justifies the first terms of the Fourier expansion
  in~\eqref{eq:phi-3-A-2}. Consequently the Borcherds product
  $B(\varphi_{0,\,3A_2})$ is a holomorphic form of weight
  $c(0,0)/2=3$ with divisors of order~$1$ along all
  $\widetilde{\Orth}^+(2\lat U\oplus 3\lat {A_2}(-1))$-orbits of the
  vectors (of square length~$-6$) $\pm \lambda_i$,
  $\pm \lambda_{i+1}$ and $\pm(\lambda_i-\lambda_{i+1})$
  ($i\in \{1,2,3\}$).  Using the K\"ocher principle as in
  Example~\ref{ex:val-2} we finally obtain
  \begin{equation*}
    \sym{Lift}(\vartheta_{3A_2})=B(\varphi_{0,\,3A_2})
    .
  \end{equation*}
  This identity remains true if we replace $\vartheta_{3A_2}$ and
  $\varphi_{0,\,3A_2}$ by its pullbacks via
  $\C\ni w\mapsto (a_1,b_1,a_2,b_2,a_3,b_3)w$, which yields the
  identity claimed in~(1).
  
  Via the isometric embedding
  $\alpha:\lat {A_2}(3)\rightarrow 3\lat {A_2}$, $x\mapsto (x,x,x)$,
  we obtain pullbacks
  \begin{equation*}
    \alpha^*B(\varphi_{0,\,3A_2})
    \in
    M_3(\widetilde{\Orth}^+(2\lat U\oplus \lat{A_2}(-3)))
  \end{equation*}
  and
  \begin{equation*}
    \varphi_{0,\,A_2}
    :=
    \tfrac{1}3 \alpha^*\varphi_{0,\,3A_2}
    = 2+\zeta_1^{\pm 1}+\zeta_{2}^{\pm 1} + (\zeta_1\zeta_2^{-1})^{\pm 1}
    + O(q)
    ,
  \end{equation*}
  the latter defining a Jacobi form of index~$\lat{A_2}(3)$. The
  Borcherds lift
  $B(\varphi_{0,\,A_2}) \in M_1(\widetilde{\Orth}^+(2\lat U\oplus
  A_2(-3)),\chi_3)$ is a third root of
  $\alpha^*B(\varphi_{0,\,3A_2})$.  Its divisor is determined by the
  reflections corresponding to the fundamental weights.  Again
  $\sym{Lift}_{\varepsilon^8}(\varphi_{0,\,A_2})$ defines a function
  with the same divisor and we obtain
  \begin{equation}
    \label{eq:A-2-lift}
    B(\varphi_{0,\,A_2})
    =
    \sym{Lift}_{\varepsilon^8}(\vartheta_{A_2})
    .
  \end{equation}
  The specialization to $(z_1,z_2)=(-a,b)z$ is the second identity of
  the theorem.
\end{proof}

\begin{Remark}
  We remark that the proof of~\eqref{eq:A-2-lift} and the fact that
  both sides of this identity are holomorphic did not make use of
  the fact that $\vartheta_{A_2}$ is holomorphic at
  infinity. However, this is implied by~\eqref{eq:A-2-lift}, which
  yields the sixth proof of the fact that the theta quarks are
  holomorphic at infinity.
\end{Remark}

\section{Miscellaneous observations and open questions}

\subsection{Jacobi-Eisenstein series and Jacobi cusp forms of small weight}

The simplest theta block with trivial character is the product of
eight theta series
$\prod_{i=1}^8\,{\vartheta_{a_i}}\in J_{4,(a_1^2+\ldots+a_8^2)/2}$,
where $a_1+\ldots+a_8$ is even (and as usual
$\vt a(\tau,z)=\vartheta(\tau,az)$).  This is a cusp form if and only
if $(a_1\cdot\ldots\cdot a_8)/d^8$ is even, where
$d=\gcd(a_1,\ldots, a_8)$.  A similar product of~$24$ quintuple
products
$\prod_{i=1}^{24}\eta\,{\vartheta^*_{a_i}}\in
J_{12,\frac{3}2(a_1^2+\ldots+a_{24}^2)}$ (where
$\vartheta^*_a = \vartheta_{2a}/\vartheta_a$) is a Jacobi cusp form if
$(a_1\cdot\ldots\cdot a_{24})/d^{24}$ is divisible by $2$ or $3$
(see~\cite[Lemma 1.2]{GH98}).

In particular, $\vartheta^8$ equals the Jacobi-Eisenstein series
$E_{4,4,1}$ of weight $4$ and index $4$
(see~\cite[p.~25]{Eichler-Zagier}). The first Jacobi cusp form of
weight~$4$ is $\vartheta^6\vartheta_2^2\in J_{4,7}$.

The Fourier coefficients of the $24$-fold product $\vartheta^8$,
i.e., the eighth power of the Jacobi triple product, can be calculated
explicitly in terms of Cohen's numbers (see~\cite{GW-17}).  {\em It
  would be interesting to calculate the Fourier coefficients of the
  $120$-fold product $(\vartheta^*)^{24}\in J_{12,36}$.}  \smallskip

The first two examples of Jacobi forms of weights $2$ and $3$ are the
Jacobi-Eisenstein series $E_{2, 25}^{(\chi)}$ and $E_{3, 9, 1}$ where
$\chi=\leg *5$ is the primitive even character modulo~$5$ (we use the
notations of~\cite[p.~25--26]{Eichler-Zagier}). Both series are theta
blocks:
\begin{equation*}
  E_{2, 25}^{(\chi)}=\eta^{-6}
  \vartheta^4\vartheta_2^3\vartheta_3^2\vartheta_4 \quad{\rm and}\quad
  E_{3, 9, 1}=Q_{1,1}^3=\eta^{-3}\vartheta^6\vartheta_2^3.
\end{equation*}
It would be interesting to find explicitly their Fourier coefficients
similar to~\cite{Eichler-Zagier} and~\cite{GW-17}, which would give
new identities for these $24$-fold products.

The next two Jacobi forms of weight $2$ and $3$ are the cusp forms
$\varphi_{2,37}$ and $\varphi_{3,13}$ of weight $2$ and~$3$ and index
$37$ and~$13$, respectively.
A table of Fourier coefficients of $\varphi_{2,37}$ was given
in~\cite{Eichler-Zagier} (see pages 118--120 and Table 4 on page 145).
Now we can give explicit formulas for these two Jacobi cusp forms:
\begin{equation*}
  \varphi_{2,37}=\eta^{-6}
  \vartheta^3\vartheta_2^3\vartheta_3^2\vartheta_4\vartheta_5 \quad{\rm
    and}\quad
  \varphi_{3,13}=\eta^{-3}\vartheta^5\vartheta_2^3\vartheta_3.
\end{equation*}

We note that $\varphi_{3,13}$ provides the existence of a canonical
differential form on the moduli space of $(1,13)$-polarized abelian
surfaces and non-triviality of the third cohomology group
$H^3(\Gamma_{13},\C)$ of the paramodular group $\Gamma_{13}$ (see
\cite{G94}).

\subsection{\bf Jacobi cusp forms of weight $2$ and $3$ with large
  $q$-order}

The product of three theta quarks is a holomorphic Jacobi form of
type $9$-$\vartheta/3$-$\eta$. It has $q$-order one.  We can construct
$21$-$\vartheta/15$-$\eta$ theta blocks, which have then weight~$3$
and $q$-order $2$.  The following three examples are related to the
antisymmetric Siegel paramodular forms of weight $3$
(see~\cite{GPY16})
\begin{align*}
  \varphi_{3,122}&=\vartheta[-15;\,1^5,2^5,3^4,4^3, 5^2,6,7],\\
  \varphi_{3,167}&=\vartheta[-15;\,1^4,2^5,3^3,4^3, 5^2,6^2,7,8],\\
  \varphi_{3,173}&=\vartheta[-15;\,1^4,2^4,3^3,4^4, 5^2,6^2,7,8].
\end{align*}
Here we use the notation
\begin{equation*}
  \vartheta[-N;\, a^n,...,b^m]=\eta^{-N}\vartheta_a^n\cdot \ldots\cdot
  \vartheta_b^m.
\end{equation*}

For weight $2$, there are holomorphic theta blocks of type
$22$-$\vartheta/18$-$\eta$, which have then $q$-order~$2$:
\begin{align*}
  \varphi_{2,587}&=
                   \vartheta[-18;\,(1,2,3,4,5,6,8)^2,2,7,9,10,11,12,13,14],\\
  \varphi_{2,713}&=
                   \vartheta[-18;\,(1,2,4,5,6,8)^2,2,3,7,8,9,10,11,12,13,15],\\
  \varphi_{2,893}&=
                   \vartheta[-18;\,1,(2,3,4,5,6,8)^2,7,9,10,11,12,13,14,16,19].
\end{align*}
The problem of constructing new Hecke paramodular forms of
genus~$2$ is related to the question of existence of theta blocks of
$q$-order $2$.  The form $\varphi_{2,587}$ is the leading
Fourier-Jacobi coefficient of the unique antisymmetric Siegel
form~$F^{(587)}$ of weight $2$ for the paramodular
group~$\Gamma_{587}$ (see~\cite{GPY16}). The existence of~$F^{(587)}$
supports the first part of the Brumer conjecture.  According to its
second part the Spin-$L$-function of $F^{(587)}$ is equal to the
Hasse-Weil $L$-function of an abelian surface with conductor $N=587$.

The form $\varphi_{3,122}$ is the leading Fourier-Jacobi coefficient
of the antisymmetric Siegel form of weight $3$ for the paramodular
group $\Gamma_{112}$. It is expected that the $L$-function of this
paramodular form is related to a motivic $L$-function of Calabi-Yau
treefolds.

We can give also an example of a weight $2$ Jacobi form of
$q$-order~$3$, namely a theta block of type
$34$-$\vartheta/30$-$\eta$:
\begin{equation*}
  \varphi_{2,2p}= \vartheta[-30;\,(1,2,3,4,5)^2,6,7,8,9,10,\dots,
  27,28,30]
  .
\end{equation*}
Its index equals $2$ times the prime $p=8669$.

Theorem~\ref{thm:GPY-theorem} together with~\ref{prop:val-1} provides
a method to construct a theta block which is holomorphic at infinity
from a given theta block satisfying certain mild conditions.  We apply
this method to the $34$-$\vartheta/30$-$\eta$-block $\varphi_{2,2p}$.
Its $q$-order is~$3$ and it is holomorphic at infinity. Hence we can
apply the two cited theorems: setting
\begin{equation*}
  \psi_{0,2p}
  =\frac{\varphi_{2,2p}|V_2}{\varphi_{2,2p}}
  =
  c(0,0)+\sum_{0<l<m}c(0,l)(\zeta^l+\zeta^{-l}) + O(q)
  ,
\end{equation*}
the theta block
\begin{equation*}
  \Th(\varphi_{2,2p})
  =
  \eta^{c(0,0)}\prod_{l>0}
  \left(\frac{\vartheta_l}{\eta}\right)^{c(0,l)}
\end{equation*}
defines a Jacobi form. Note that the block $\Th(\varphi_{2,2p})$ has
weight~$444$, index~$41888608$, and $q$-order~$2488$; it is of the
form~$29412$-$\vartheta/28524$-$\eta$.

From Theorem~\ref{thm:theoretical-answer-to-the-challenge} we know
that the number $N$ of $\vartheta$ in a theta block of weight~$2$
which is holomorphic at infinity is bounded; namely, one has
$H_{2N}/555.960 \le 2$, which implies
$N\le \frac 12 e^{2\cdot 555.960} $. The $q$-order of a theta block of
type $N$-$\vartheta$/$n$-$\eta$ equals $N/8 - n/24$. Hence the
$q$-order~$v$ of a theta block of weight~$2$ and holomorphic at
infinity is bounded; one has $v\le \frac 1{16}e^{2\cdot
  555.960}$. This leads to the natural questiom: {\em to find the
  maximal possible $q$-order of theta blocks of weight $2$ or to find
  a reasonable upper bound.}

A theta block of weight~$2$ and trivial character needs to be of the
form $(10+12d)$-$\vartheta/(6+12d)$-$\eta$ ($d=0,1,2,\dots$).  In
Part~III we found four infinite families of theta blocks holomorphic
at infinity of weight~$2$ with trivial character of type
$10$-$\vartheta/6$-$\eta$ (see
Table~\ref{tab:theta-R-representations}).  In this section we saw
examples of theta blocks holomorphic at infinity of weight $2$ with
trivial character of type $22$-$\vartheta/18$-$\eta$ and
$34$-$\vartheta/30$-$\eta$. This raises the question: {\em to find an
  arithmetic or representation theoretic explanation for the existence
  of theta blocks of types $(10+12d)$-$\vartheta/(6+12d)$-$\eta$
  for~$d\ge 1$.}

\subsection{Jacobi forms of weight $2$ without character}

As we saw in Section~\ref{sec:small-weights}, all spaces of Jacobi
forms of weight~$1/2$ and weight~$1$ are spanned by theta blocks (with
the exception of weight $1$ and character~$\varepsilon^2$). We also know
from Section~\ref{sec:generalities} (Remark after
Theorem~\ref{thm:number-of-general-theta-blocks}) that, for growing
weight~$k$ and fixed index~$m$ and character~$\varepsilon^h$, the
proportion of the subspace of $J_{k,m}\left(\varepsilon^h\right)$
spanned by theta blocks becomes smaller and smaller.  In view of the
lifting of Jacobi forms in $J_{2,m}$ to modular forms of weight $2$
and level~$m$ (see~\cite[Thm.~5]{Skoruppa-Zagier}) it is of interest
{\em to know if all of the spaces $J_{2,m}$ are still spanned by theta
  blocks, or how big the subspace spanned by theta blocks is.}

\begin{table}[ht]
  \begin{adjustwidth}{-1cm}{-1cm}
    \centering
    \caption{\Small The table lists, for each index $1\le m<200$ such
      that $J_{2,m}\not=0$, the dimensions e and c of the subspace of
      Eisenstein series and cusp forms and the numbers te and tc of
      theta blocks in $J_{2,m}$ which are non-cusp forms and cusp
      forms, respectively.}

\newcolumntype{R}{>{$}r<{$}}

\begin{tabular}{RRRRR}
\toprule
   m &   e &   c &   te &   tc \\
\midrule
  25 &   1 &   0 &    1 &    0 \\
  37 &   0 &   1 &    0 &    1 \\
  43 &   0 &   1 &    0 &    1 \\
  49 &   2 &   0 &    3 &    0 \\
  50 &   1 &   0 &    1 &    0 \\
  53 &   0 &   1 &    0 &    1 \\
  57 &   0 &   1 &    0 &    1 \\
  58 &   0 &   1 &    0 &    1 \\
  61 &   0 &   1 &    0 &    1 \\
  64 &   1 &   0 &    1 &    0 \\
  65 &   0 &   1 &    0 &    1 \\
  67 &   0 &   2 &    0 &    3 \\
  73 &   0 &   2 &    0 &    3 \\
  74 &   0 &   1 &    0 &    1 \\
  75 &   1 &   0 &    1 &    0 \\
  77 &   0 &   1 &    0 &    1 \\
  79 &   0 &   1 &    0 &    1 \\
  81 &   2 &   0 &    3 &    0 \\
  82 &   0 &   1 &    0 &    1 \\
  83 &   0 &   1 &    0 &    1 \\
  85 &   0 &   2 &    0 &    3 \\
  86 &   0 &   1 &    0 &    1 \\
  88 &   0 &   1 &    0 &    1 \\
  89 &   0 &   1 &    0 &    1 \\
  91 &   0 &   2 &    0 &    3 \\
  92 &   0 &   1 &    0 &    1 \\
  93 &   0 &   2 &    0 &    3 \\
  97 &   0 &   3 &    0 &    7 \\
  98 &   2 &   0 &    3 &    0 \\
  99 &   0 &   1 &    0 &    1 \\
\bottomrule
\end{tabular}
\begin{tabular}{RRRRR}
\toprule
   m &   e &   c &   te &   tc \\
\midrule  
 100 &   2 &   0 &    3 &    0 \\
 101 &   0 &   1 &    0 &    1 \\
 102 &   0 &   1 &    0 &    1 \\
 103 &   0 &   2 &    0 &    3 \\
 106 &   0 &   2 &    0 &    2 \\
 107 &   0 &   2 &    0 &    3 \\
 109 &   0 &   3 &    0 &    6 \\
 111 &   0 &   1 &    0 &    1 \\
 112 &   0 &   1 &    0 &    1 \\
 113 &   0 &   3 &    0 &    7 \\
 114 &   0 &   1 &    0 &    1 \\
 115 &   0 &   2 &    0 &    3 \\
 116 &   0 &   1 &    0 &    1 \\
 117 &   0 &   1 &    0 &    1 \\
 118 &   0 &   1 &    0 &    1 \\
 121 &   4 &   1 &   16 &    1 \\
 122 &   0 &   2 &    0 &    3 \\
 123 &   0 &   2 &    0 &    3 \\
 124 &   0 &   1 &    0 &    1 \\
 125 &   1 &   2 &    3 &    3 \\
 127 &   0 &   3 &    0 &    7 \\
 128 &   1 &   1 &    2 &    1 \\
 129 &   0 &   2 &    0 &    2 \\
 130 &   0 &   2 &    0 &    3 \\
 131 &   0 &   1 &    0 &    1 \\
 133 &   0 &   4 &    0 &   12 \\
 134 &   0 &   2 &    0 &    3 \\
 135 &   0 &   1 &    0 &    1 \\
 136 &   0 &   1 &    0 &    1 \\
 137 &   0 &   4 &    0 &   11 \\
\bottomrule
\end{tabular}
\begin{tabular}{RRRRR}
\toprule
   m &   e &   c &   te &   tc \\
\midrule
 138 &   0 &   1 &    0 &    1 \\
 139 &   0 &   3 &    0 &    6 \\
 141 &   0 &   2 &    0 &    2 \\
 142 &   0 &   2 &    0 &    3 \\
 143 &   0 &   1 &    0 &    1 \\
 144 &   1 &   0 &    1 &    0 \\
 145 &   0 &   3 &    0 &    7 \\
 146 &   0 &   2 &    0 &    2 \\
 147 &   2 &   2 &    8 &    3 \\
 148 &   0 &   3 &    0 &    7 \\
 149 &   0 &   3 &    0 &    5 \\
 150 &   1 &   0 &    1 &    0 \\
 151 &   0 &   3 &    0 &    5 \\
 152 &   0 &   1 &    0 &    1 \\
 153 &   0 &   2 &    0 &    3 \\
 154 &   0 &   2 &    0 &    3 \\
 155 &   0 &   2 &    0 &    3 \\
 156 &   0 &   1 &    0 &    1 \\
 157 &   0 &   5 &    0 &   18 \\
 158 &   0 &   3 &    0 &    6 \\
 159 &   0 &   1 &    0 &    1 \\
 160 &   0 &   1 &    0 &    1 \\
 161 &   0 &   2 &    0 &    2 \\
 162 &   2 &   1 &    5 &    1 \\
 163 &   0 &   6 &    0 &   26 \\
 164 &   0 &   1 &    0 &    0 \\
 165 &   0 &   2 &    0 &    3 \\
 166 &   0 &   2 &    0 &    2 \\
 167 &   0 &   2 &    0 &    3 \\
 169 &   5 &   3 &   45 &    5 \\
\bottomrule
\end{tabular}
\begin{tabular}{RRRRR}
\toprule
   m &   e &   c &   te &   tc \\
\midrule
 170 &   0 &   3 &    0 &    7 \\
 171 &   0 &   2 &    0 &    3 \\
 172 &   0 &   3 &    0 &    6 \\
 173 &   0 &   4 &    0 &   11 \\
 174 &   0 &   1 &    0 &    0 \\
 175 &   1 &   2 &    3 &    2 \\
 176 &   0 &   2 &    0 &    3 \\
 177 &   0 &   4 &    0 &    9 \\
 178 &   0 &   3 &    0 &    5 \\
 179 &   0 &   3 &    0 &    5 \\
 181 &   0 &   5 &    0 &   14 \\
 182 &   0 &   2 &    0 &    3 \\
 183 &   0 &   3 &    0 &    5 \\
 184 &   0 &   3 &    0 &    6 \\
 185 &   0 &   4 &    0 &   11 \\
 186 &   0 &   2 &    0 &    2 \\
 187 &   0 &   5 &    0 &   20 \\
 188 &   0 &   2 &    0 &    2 \\
 189 &   0 &   2 &    0 &    3 \\
 190 &   0 &   2 &    0 &    3 \\
 191 &   0 &   2 &    0 &    1 \\
 192 &   1 &   1 &    2 &    1 \\
 193 &   0 &   7 &    0 &   33 \\
 194 &   0 &   3 &    0 &    4 \\
 195 &   0 &   1 &    0 &    1 \\
 196 &   4 &   1 &   13 &    1 \\
 197 &   0 &   6 &    0 &   27 \\
 198 &   0 &   2 &    0 &    3 \\
  199 &   0 &   4 &    0 &    8 \\
  &&&&\\
\bottomrule
\end{tabular}
    
    \label{tab:weight-2-no-character}
  \end{adjustwidth}
\end{table}

The first question can be quickly answered. A computer search for
$m<200$ shows that $J_{2,m}$ is spanned by theta blocks for all~$m$
with the exception of $m=164$ (see
Table~\ref{tab:weight-2-no-character}). In fact, $J_{2,164}$, which is
one-dimensional and contains exactly one cusp form, does not contain a
theta block, not even a single generalized theta block.

However, for computational purposes this is often no serious
problem. For instance, the one-dimensional space $J_{2,164}$ can be
easily obtained applying the index raising operator $V_2$
(see~\cite[\S 4]{Eichler-Zagier}) to the single theta block in
$J_{2,82}$ (which is
$\vartheta_{1}\vartheta_{2}^{3}\vartheta_{3}\vartheta_{4}^{2}\vartheta_{5}\vartheta_{6}\vartheta_{7}/\eta^6$). Alternatively
one can try to find sufficiently many theta blocks which are not
necessarily holomorphic at infinity but span a space containing a
given $J_{k,m}$.

In the context of the mentioned computations it is worthwhile to
mention that, for $1\le m<200$, the spaces $J_{2,m}$ contain no theta
quotients.

Concerning the second question we do not know of any method to
determine the size of the subspace in $J_{2,m}$ spanned by theta
blocks. Heuristically, however, one might expect it to be large in
general. Indeed, the well-known dimension formula shows
$\dim J_{2,m}\sim \frac {m+1}{24}$. On the other hand, already the
four families of theta blocks from
Table~\ref{tab:theta-R-representations} each provide as many theta
blocks of index~$m$ as there are positive integers $a,b,c,d$ such that
the sum of the squares of the indices of the theta block defining this
family equals $2m$, a number whose order of magnitude is~$m$.

\subsection{Theta blocks and elliptic curves}

As mentioned in the introduction, the first Jacobi cusp form of
weight~$2$, which has index $37$, is a theta block. This is of
particular interest since this form corresponds to the first elliptic
curve of odd rank (which has in fact rank~$1$ and level~$37$) via the Hecke
equivariant lifting of $J_{2,37}$ onto the space of modular forms of
weight $2$ and level~$37$.

In general, we do not know any reason that an arbitrary theta block in
$J_{2,m}$ is a Hecke eigenform except for the banal reason that
$J_{2,m}$ or the subspace of cusp forms in $J_{2,m}$ is
one-dimensional, so that any Jacobi form in one of these spaces is
trivially an eigenform. In particular, we do not expect that the
Jacobi form associated to an elliptic curve is a theta block. However,
there are exactly $52$ indices where $J_{2,m}$ contains only one cusp
form. For $10$ of these indices the corresponding Jacobi form is an
old form. For each index $m$ in the set $S$ of the remaining $42$ (see
Table~\ref{tab:elliptic-curves-and-theta-blocks}) the associated cusp
Jacobi form $\phi_m$ corresponds via the mentioned lifting to an
elliptic curve over the rationals of conductor~$m$ whose $L$-series
$L(E,s)$ has a minus sign in its functional equation. This
correspondence is given by the identities
\begin{equation*}
  \sum_{n\ge 1}{\tleg Dn}{n^{-s}}
  \sum_{n\ge 1}{C_{\phi_m}(Dn^2,rn)}{n^{-s}}
  =
  C_{\phi_m}(D,r) L(E,s)
  ,
\end{equation*}
valid for any negative fundamental discriminant $D$ and integer $r$
such that $D\equiv r^2 \bmod 4m$.

As it turns out, each of these $\phi_m$ with the exception of
$\phi_{300}$ is a theta block. More precisely, we found that for each
index $m\not=300$ in~$S$ there is exactly one theta block of
length~$10$ in $J_{2,m}$ which is a cusp form. (For $m\le 200$ and
$m=216$ we verified in addition that there is no theta block of length
strictly greater than~$10$ in the subspace of cusp forms of
$J_{2,m}$.)

In Table~\ref{tab:elliptic-curves-and-theta-blocks} we give
for each $m$ in~$S$ a minimal
equation for an elliptic curve over $\Q$ with conductor $m$ and root
number~$-1$ (in general the isogeny classes of the given curves
decompose into more than one rational isomorphism classes) and, for $m\not=300$, the corresponding theta block. All these
elliptic curves have rank~$1$. Except for $m\in\{89,121\}$ the
theta blocks in this table belong to one or more of the four families
associated to the root systems $A_4$, $G_2\oplus B_2$, $A_1\oplus B_3$
and $A_1\oplus C_3$ (see Table~\ref{tab:theta-R-representations}).

\begin{table}[htbp]
  \begin{adjustwidth}{-1cm}{-1cm}
    \centering
    \caption{\Small For each $m$ such that the subspace of cusp forms
      in $J_{2,m}$ is generated by a new form $\phi_m$ , the
      associated elliptic curve and a theta block representation of
      $\eta^6\,\phi_m$. (CL is the {\em Cremona label} of the respective
      elliptic curve.)}
    \begin{tabular}{llll}
      \toprule
      $m$&CL&Curve&Theta block\\
      \midrule

$37$&37a1&$y^2 + y = x^{3} -  x $&$\vartheta_{1}^{3}\vartheta_{2}^{3}\vartheta_{3}^{2}\vartheta_{4}\vartheta_{5}$\\
$43$&43a1&$y^2 + y = x^{3} + x^{2} $&$\vartheta_{1}^{3}\vartheta_{2}^{2}\vartheta_{3}^{2}\vartheta_{4}^{2}\vartheta_{5}$\\
$53$&53a1&$y^2 + x y + y = x^{3} -  x^{2} $&$\vartheta_{1}^{3}\vartheta_{2}^{2}\vartheta_{3}^{2}\vartheta_{4}\vartheta_{5}\vartheta_{6}$\\
$57$&57a1&$y^2 + y = x^{3} -  x^{2} - 2 x + 2 $&$\vartheta_{1}^{2}\vartheta_{2}^{2}\vartheta_{3}^{3}\vartheta_{4}\vartheta_{5}\vartheta_{6}$\\
$58$&58a1&$y^2 + x y = x^{3} -  x^{2} -  x + 1 $&$\vartheta_{1}^{2}\vartheta_{2}^{3}\vartheta_{3}\vartheta_{4}^{2}\vartheta_{5}\vartheta_{6}$\\
$61$&61a1&$y^2 + x y = x^{3} - 2 x + 1 $&$\vartheta_{1}^{2}\vartheta_{2}^{3}\vartheta_{3}^{2}\vartheta_{4}\vartheta_{5}\vartheta_{7}$\\
$65$&65a1&$y^2 + x y = x^{3} -  x $&$\vartheta_{1}^{2}\vartheta_{2}^{2}\vartheta_{3}^{2}\vartheta_{4}\vartheta_{5}^{2}\vartheta_{6}$\\
$77$&77a1&$y^2 + y = x^{3} + 2 x $&$\vartheta_{1}^{2}\vartheta_{2}^{2}\vartheta_{3}^{2}\vartheta_{4}\vartheta_{5}\vartheta_{6}\vartheta_{7}$\\
$79$&79a1&$y^2 + x y + y = x^{3} + x^{2} - 2 x $&$\vartheta_{1}^{2}\vartheta_{2}^{2}\vartheta_{3}^{2}\vartheta_{4}\vartheta_{5}^{2}\vartheta_{8}$\\
$82$&82a1&$y^2 + x y + y = x^{3} - 2 x $&$\vartheta_{1}\vartheta_{2}^{3}\vartheta_{3}\vartheta_{4}^{2}\vartheta_{5}\vartheta_{6}\vartheta_{7}$\\
$83$&83a1&$y^2 + x y + y = x^{3} + x^{2} + x $&$\vartheta_{1}^{2}\vartheta_{2}\vartheta_{3}^{2}\vartheta_{4}^{2}\vartheta_{5}\vartheta_{6}\vartheta_{7}$\\
$88$&88a1&$y^2 = x^{3} - 4 x + 4 $&$\vartheta_{1}^{2}\vartheta_{2}^{2}\vartheta_{3}\vartheta_{4}^{2}\vartheta_{5}\vartheta_{6}\vartheta_{8}$\\
$89$&89a1&$y^2 + x y + y = x^{3} + x^{2} -  x $&$\vartheta_{1}^{3}\vartheta_{2}\vartheta_{3}\vartheta_{4}\vartheta_{5}\vartheta_{6}^{2}\vartheta_{7}$\\
$92$&92b1&$y^2 = x^{3} -  x + 1 $&$\vartheta_{1}\vartheta_{2}^{2}\vartheta_{3}^{2}\vartheta_{4}^{2}\vartheta_{5}\vartheta_{6}\vartheta_{8}$\\
$99$&99a1&$y^2 + x y + y = x^{3} -  x^{2} - 2 x $&$\vartheta_{1}^{2}\vartheta_{2}\vartheta_{3}^{2}\vartheta_{4}^{2}\vartheta_{5}\vartheta_{6}\vartheta_{9}$\\
$101$&101a1&$y^2 + y = x^{3} + x^{2} -  x - 1 $&$\vartheta_{1}^{2}\vartheta_{2}\vartheta_{3}\vartheta_{4}\vartheta_{5}^{2}\vartheta_{6}^{2}\vartheta_{7}$\\
$102$&102a1&$y^2 + x y = x^{3} + x^{2} - 2 x $&$\vartheta_{1}\vartheta_{2}^{2}\vartheta_{3}^{2}\vartheta_{4}\vartheta_{5}\vartheta_{6}^{2}\vartheta_{8}$\\
$112$&112a1&$y^2 = x^{3} + x^{2} + 4 $&$\vartheta_{1}\vartheta_{2}^{2}\vartheta_{3}\vartheta_{4}^{2}\vartheta_{5}\vartheta_{6}\vartheta_{7}\vartheta_{8}$\\
$117$&117a1&$y^2 + x y + y = x^{3} -  x^{2} + 4 x + 6 $&$\vartheta_{1}\vartheta_{2}^{2}\vartheta_{3}^{2}\vartheta_{4}\vartheta_{5}\vartheta_{6}\vartheta_{7}\vartheta_{9}$\\
$118$&118a1&$y^2 + x y = x^{3} + x^{2} + x + 1 $&$\vartheta_{1}^{2}\vartheta_{2}^{2}\vartheta_{4}\vartheta_{5}\vartheta_{6}^{2}\vartheta_{7}\vartheta_{8}$\\
$121$&121b1&$y^2 + y = x^{3} -  x^{2} - 7 x + 10 $&$\vartheta_{1}\vartheta_{2}^{2}\vartheta_{3}^{2}\vartheta_{4}\vartheta_{5}^{2}\vartheta_{7}\vartheta_{10}$\\
$124$&124a1&$y^2 = x^{3} + x^{2} - 2 x + 1 $&$\vartheta_{1}^{2}\vartheta_{2}\vartheta_{4}^{2}\vartheta_{5}\vartheta_{6}^{2}\vartheta_{7}\vartheta_{8}$\\
$128$&128a1&$y^2 = x^{3} + x^{2} + x + 1 $&$\vartheta_{1}\vartheta_{2}^{2}\vartheta_{3}\vartheta_{4}^{2}\vartheta_{5}\vartheta_{6}\vartheta_{8}\vartheta_{9}$\\
$131$&131a1&$y^2 + y = x^{3} -  x^{2} + x $&$\vartheta_{1}^{2}\vartheta_{2}^{2}\vartheta_{3}^{2}\vartheta_{4}\vartheta_{5}\vartheta_{7}\vartheta_{12}$\\
$135$&135a1&$y^2 + y = x^{3} - 3 x + 4 $&$\vartheta_{1}\vartheta_{2}\vartheta_{3}^{2}\vartheta_{4}\vartheta_{5}^{2}\vartheta_{6}\vartheta_{8}\vartheta_{9}$\\
$136$&136a1&$y^2 = x^{3} + x^{2} - 4 x $&$\vartheta_{1}\vartheta_{2}^{2}\vartheta_{3}\vartheta_{4}\vartheta_{5}\vartheta_{6}\vartheta_{7}\vartheta_{8}^{2}$\\
$138$&138a1&$y^2 + x y = x^{3} + x^{2} -  x + 1 $&$\vartheta_{1}\vartheta_{2}\vartheta_{3}^{2}\vartheta_{4}^{2}\vartheta_{6}^{2}\vartheta_{7}\vartheta_{10}$\\
$143$&143a1&$y^2 + y = x^{3} -  x^{2} -  x - 2 $&$\vartheta_{1}^{2}\vartheta_{2}\vartheta_{3}\vartheta_{4}\vartheta_{5}\vartheta_{6}\vartheta_{7}\vartheta_{8}\vartheta_{9}$\\
$152$&152a1&$y^2 = x^{3} + x^{2} -  x + 3 $&$\vartheta_{1}\vartheta_{2}\vartheta_{3}^{2}\vartheta_{4}^{2}\vartheta_{6}\vartheta_{7}\vartheta_{8}\vartheta_{10}$\\
$156$&156a1&$y^2 = x^{3} -  x^{2} - 5 x + 6 $&$\vartheta_{1}\vartheta_{2}^{2}\vartheta_{3}^{2}\vartheta_{4}\vartheta_{5}\vartheta_{6}\vartheta_{8}\vartheta_{12}$\\
$160$&160a1&$y^2 = x^{3} + x^{2} - 6 x + 4 $&$\vartheta_{1}\vartheta_{2}\vartheta_{3}\vartheta_{4}^{2}\vartheta_{5}\vartheta_{6}\vartheta_{7}\vartheta_{8}\vartheta_{10}$\\
$162$&162a1&$y^2 + x y = x^{3} -  x^{2} - 6 x + 8 $&$\vartheta_{1}\vartheta_{2}\vartheta_{3}\vartheta_{4}^{2}\vartheta_{5}\vartheta_{6}^{2}\vartheta_{9}\vartheta_{10}$\\
$192$&192a1&$y^2 = x^{3} -  x^{2} - 4 x - 2 $&$\vartheta_{1}\vartheta_{2}\vartheta_{3}\vartheta_{4}\vartheta_{5}\vartheta_{6}^{2}\vartheta_{7}\vartheta_{8}\vartheta_{12}$\\
$196$&196a1&$y^2 = x^{3} -  x^{2} - 2 x + 1 $&$\vartheta_{1}\vartheta_{2}\vartheta_{3}\vartheta_{4}^{2}\vartheta_{5}\vartheta_{7}\vartheta_{8}^{2}\vartheta_{12}$\\
$200$&200b1&$y^2 = x^{3} + x^{2} - 3 x - 2 $&$\vartheta_{1}\vartheta_{2}\vartheta_{3}\vartheta_{4}\vartheta_{5}\vartheta_{6}\vartheta_{8}^{2}\vartheta_{9}\vartheta_{10}$\\
$210$&210d1&$y^2 + x y = x^{3} + x^{2} - 3 x - 3 $&$\vartheta_{1}\vartheta_{2}\vartheta_{3}\vartheta_{4}\vartheta_{5}\vartheta_{6}^{2}\vartheta_{7}\vartheta_{10}\vartheta_{12}$\\
$216$&216a1&$y^2 = x^{3} - 12 x + 20 $&$\vartheta_{2}^{2}\vartheta_{3}\vartheta_{4}\vartheta_{5}\vartheta_{6}\vartheta_{7}\vartheta_{8}\vartheta_{9}\vartheta_{12}$\\
$220$&220a1&$y^2 = x^{3} + x^{2} - 45 x + 100 $&$\vartheta_{2}^{2}\vartheta_{3}\vartheta_{4}\vartheta_{5}^{2}\vartheta_{7}\vartheta_{8}\vartheta_{10}\vartheta_{12}$\\
$240$&240c1&$y^2 = x^{3} -  x^{2} + 4 x $&$\vartheta_{1}\vartheta_{2}\vartheta_{3}\vartheta_{4}\vartheta_{5}\vartheta_{6}\vartheta_{8}\vartheta_{9}\vartheta_{10}\vartheta_{12}$\\
$252$&252b1&$y^2 = x^{3} - 12 x + 65 $&$\vartheta_{1}\vartheta_{2}\vartheta_{3}\vartheta_{4}\vartheta_{6}\vartheta_{7}\vartheta_{8}\vartheta_{9}\vartheta_{10}\vartheta_{12}$\\
$300$&300d1&$y^2 = x^{3} -  x^{2} - 13 x + 22 $&$\mathrm{?}$\\
$360$&360e1&$y^2 = x^{3} - 18 x - 27 $&$\vartheta_{2}\vartheta_{3}\vartheta_{4}\vartheta_{5}\vartheta_{6}\vartheta_{7}\vartheta_{9}\vartheta_{10}\vartheta_{12}\vartheta_{16}$\\

      \bottomrule
    \end{tabular}
    \label{tab:elliptic-curves-and-theta-blocks}
  \end{adjustwidth}
\end{table}

The space $J_{2,300}$ has dimension~$3$ and contains $5$~theta blocks
of length~$10$, which span the whole space. Here $\phi_{300}$ does not
equal any of these $5$ theta blocks. {\em We do not know whether it equals
a theta block of length greater than~$10$ (i.e., of length
$N=22,34,46,\dots$) or a generalized theta block.}

Concerning the question of an explicit example of a rational elliptic
curve whose associated Jacobi form is not a theta block, we found
$m=91$ as the first $m$ such that the subspace of cusp forms in
$J_{2,m}$ has dimension greater than~$1$ and contains a Hecke
eigenform with rational eigenvalues. In fact, $J_{2,91}$ contains no
non-cusp forms and has dimension~$2$, so that both Hecke eigenforms in
this space have rational eigenvalues and hence correspond to elliptic
curves. The Cremona label of these elliptic curves are 91.a1 and
91.b1, 91.b2, 91.b3, and they all have rank~$1$. The space $J_{2,91}$
contains exactly three cuspidal generalized theta blocks (which are in
fact theta blocks):
\begin{equation*}
  A = \frac {\vt{}^2\vt2\vt3\vt4^2\vt5^2\vt6\vt7}{\eta^{6}},\quad
  B = \frac {\vt{}^2\vt2^2\vt3^2\vt4\vt5\vt7\vt8}{\eta^{6}},\quad
  C = \frac {\vt{}\vt2^2\vt3^2\vt4^2\vt5\vt7^2}{\eta^{6}}
  .
\end{equation*}
One has $A+B=C$ (as follows for instance from the theta
relations below). The Hecke eigenforms are
\begin{equation*}
  A+B=C,\quad A-B
  .
\end{equation*}
Hence one is a theta block, the other one is not.

\subsection{Linear relations among theta blocks}

When studying linear dependencies between sets of theta blocks one can
restrict to sets whose elements have the same weight, same index and
same character (since the ring of all Jacobi forms is graded by
weight, index, character). Table~\ref{tab:weight-2-no-character}
suggests many concrete examples of linear dependencies. For instance
$J_{2,169}$ has dimension~$8$ but contains $50$ theta blocks.

Using the following identity, which seems to be due to Weierstrass
(see~\cite[1.]{Weierstrass}),
\begin{align*}
  &\vartheta(\tau,z_0+z_1) \vartheta(\tau,z_0-z_1) \vartheta(\tau,z_2+z_3) \vartheta(\tau,z_2-z_3)\\
  + &\vartheta(\tau,z_0+z_2) \vartheta(\tau,z_0-z_2) \vartheta(\tau,z_3+z_1) \vartheta(\tau,z_3-z_1)\\
  + &\vartheta(\tau,z_0+z_3) \vartheta(\tau,z_0-z_3) \vartheta(\tau,z_1+z_2) \vartheta(\tau,z_1-z_2)
  = 0
  .
\end{align*}
one obtains immediately an infinite family of linear relations between
theta blocks. Namely, using again $\vt a(\tau,z)=\vartheta(\tau,az)$
and substituting $(z_0+z_1,z_0-z_1,z_2+z_3,z_2-z_3)=(a,b,c,d)z$ yields
the relations
\begin{multline*}
  \vt {a} \vt {b} \vt {c} \vt {d} + \vt {(a+b+c-d)/2} \vt
  {(a+b-c+d)/2} \vt {(a-b+c+d)/2} \vt {(a-b-c-d)/2}
  \\
  = \vt {(a+b+c+d)/2} \vt {(a+b-c-d)/2} \vt {(a-b+c-d)/2} \vt
  {(a-b-c+d)/2} .
\end{multline*}
Here $a,b,c,d$ denotes any quadruple of integers whose sum is
even. For instance, for $a,b,c,d=1,4,5,6$ we obtain
$\vt {}\vt 4\vt 5\vt 6 + \vt 2 \vt 3\vt 4 \vt {-7} = \vt 8\vt {-3}\vt
{-2}\vt {-1}$, which after multiplication by
$\vt {}\vt 2\vt 3\vt 4\vt 5\vt 7/\eta^6$ yields the identity $A+B=C$ of
the preceding section.

There is also a five term relation similar to Weierstrass' three term
relation, whose terms are also products of four $\vartheta$, and which
is due to Jacobi (see~\cite[p.~507, formula (A)]{Jacobi}). {\em It is
  an interesting question if one can develop a theory based on such
  relations for theta functions in several variables which explains
  all linear relations among theta blocks.}

\FloatBarrier

\bibliography{quarks}

\begin{thebibliography}{GPY15}

\bibitem[BK96]{Belov-Konyagin}
A.~S. Belov and S.~V. Konyagin.
\newblock On an estimate for the free term of a nonnegative trigonometric
  polynomial with integer coefficients.
\newblock {\em Izv. Ross. Akad. Nauk Ser. Mat.}, 60(6):31--90, 1996.

\bibitem[Bor95]{Borcherds95}
Richard~E. Borcherds.
\newblock Automorphic forms on {${\rm O}_{s+2,2}({\bf R})$} and infinite
  products.
\newblock {\em Invent. Math.}, 120(1):161--213, 1995.

\bibitem[Bor98]{Borcherds98}
Richard~E. Borcherds.
\newblock Automorphic forms with singularities on {G}rassmannians.
\newblock {\em Invent. Math.}, 132(3):491--562, 1998.

\bibitem[BS13a]{Boylan-Skoruppa-Joli-1}
Hatice Boylan and Nils-Peter Skoruppa.
\newblock Jacobi forms of lattice index. {Part~I}: Basic theory, shadow
  representations and maximal lattices.
\newblock preprint, 2013.

\bibitem[BS13b]{Boylan-Skoruppa-Joli-2}
Hatice Boylan and Nils-Peter Skoruppa.
\newblock Jacobi forms of lattice index. {Part II}: Singular and critical
  weight.
\newblock preprint, 2013.

\bibitem[CG13]{CG}
F.~Cl{\'e}ry and V.~Gritsenko.
\newblock Modular forms of orthogonal type and {J}acobi theta-series.
\newblock {\em Abh. Math. Semin. Univ. Hambg.}, 83(2):187--217, 2013.

\bibitem[EZ85]{Eichler-Zagier}
Martin Eichler and Don Zagier.
\newblock {\em The theory of {J}acobi forms}, volume~55 of {\em Progress in
  Mathematics}.
\newblock Birkh\"auser Boston Inc., Boston, MA, 1985.

\bibitem[FH91]{Fulton-Harris}
William Fulton and Joe Harris.
\newblock {\em Representation theory}, volume 129 of {\em Graduate Texts in
  Mathematics}.
\newblock Springer-Verlag, New York, 1991.
\newblock A first course, Readings in Mathematics.

\bibitem[GH98]{GH98}
V.~Gritsenko and K.~Hulek.
\newblock Commutator coverings of {S}iegel threefolds.
\newblock {\em Duke Math. J.}, 94(3):509--542, 1998.

\bibitem[GN98]{GN}
Valeri~A. Gritsenko and Viacheslav~V. Nikulin.
\newblock Automorphic forms and {L}orentzian {K}ac-{M}oody algebras. {II}.
\newblock {\em Internat. J. Math.}, 9(2):201--275, 1998.

\bibitem[GPY]{GPY16}
V.~Gritsenko, C.~Poor, and D.~Yuen.
\newblock Antisymmetric paramodular forms of weights 2 and 3.
\newblock ArXiv:1609.04146, 20 pages.

\bibitem[GPY15]{GPY}
Valery Gritsenko, Cris Poor, and David~S. Yuen.
\newblock Borcherds products everywhere.
\newblock {\em J. Number Theory}, 148:164--195, 2015.

\bibitem[Gri94a]{G}
V.~Gritsenko.
\newblock Modular forms and moduli spaces of abelian and {$K3$} surfaces.
\newblock {\em Algebra i Analiz}, 6(6):65--102, 1994.

\bibitem[Gri94b]{G94}
Valeri Gritsenko.
\newblock Irrationality of the moduli spaces of polarized abelian surfaces.
\newblock {\em Internat. Math. Res. Notices}, (6):235 ff., approx. 9 pp.\,
  1994.

\bibitem[Gri18]{G18}
Valery Gritsenko.
\newblock Reflective modular forms and applications.
\newblock {\em Russian Math. Surveys 73:5}, pages 797--864, 2018.

\bibitem[GW18]{GW-17}
Valery Gritsenko and Haowu Wang.
\newblock Powers of {J}acobi triple product, {C}ohen's numbers and the
  {R}amanujan {$\Delta$}-function.
\newblock {\em Eur. J. Math.}, 4(2):561--584, 2018.

\bibitem[Hum78]{Humphreys}
James~E. Humphreys.
\newblock {\em Introduction to {L}ie algebras and representation theory},
  volume~9 of {\em Graduate Texts in Mathematics}.
\newblock Springer-Verlag, New York-Berlin, 1978.
\newblock Second printing, revised.

\bibitem[Jac81]{Jacobi}
C.~G.~J. Jacobi.
\newblock Theorie der elliptischen {F}unctionen aus den {E}igenschaften der
  {T}hetareihen abgeleitet.
\newblock In C.~W. Borchardt, editor, {\em Gesammelte Werke}, volume Erster
  {B}and, pages 497--538. Reimer, G., Berlin, 1881.

\bibitem[Klo46]{Kloosterman-I}
H.~D. Kloosterman.
\newblock The behaviour of general theta functions under the modular group and
  the characters of binary modular congruence {groups.~I}.
\newblock {\em Ann. of Math. (2)}, 47:317--375, 1946.

\bibitem[Kol94]{Kolountzakis}
Mihail~N. Kolountzakis.
\newblock On nonnegative cosine polynomials with nonnegative integral
  coefficients.
\newblock {\em Proc. Amer. Math. Soc.}, 120(1):157--163, 1994.

\bibitem[KP84]{KP}
Victor~G. Kac and Dale~H. Peterson.
\newblock Infinite-dimensional {L}ie algebras, theta functions and modular
  forms.
\newblock {\em Adv. in Math.}, 53(2):125--264, 1984.

\bibitem[Mac72]{Macdonald}
I.~G. Macdonald.
\newblock Affine root systems and {D}edekind's {$\eta $}-function.
\newblock {\em Invent. Math.}, 15:91--143, 1972.

\bibitem[MPS81]{McGehee-Carruth-Pigno}
O.~Carruth McGehee, Louis Pigno, and Brent Smith.
\newblock Hardy's inequality and the {$L^{1}$} norm of exponential sums.
\newblock {\em Ann. of Math. (2)}, 113(3):613--618, 1981.

\bibitem[Odl82]{Odlyzko}
A.~M. Odlyzko.
\newblock Minima of cosine sums and maxima of polynomials on the unit circle.
\newblock {\em J. London Math. Soc. (2)}, 26(3):412--420, 1982.

\bibitem[Sko85]{Skoruppa:Dissertation}
Nils-Peter Skoruppa.
\newblock {\em \"{U}ber den {Z}usammenhang zwischen {J}acobiformen und
  {M}odulformen halbganzen {G}ewichts}.
\newblock Bonner Mathematische Schriften [Bonn Mathematical Publications], 159.
  Universit\"at Bonn Mathematisches Institut, Bonn, 1985.
\newblock Dissertation, Rheinische Friedrich-Wilhelms-Universit{\"a}t, Bonn,
  1984.

\bibitem[Sko08]{Skoruppa:Critical-Weight}
Nils-Peter Skoruppa.
\newblock Jacobi forms of critical weight and {W}eil representations.
\newblock In {\em Modular forms on {S}chiermonnikoog}, pages 239--266.
  Cambridge Univ. Press, Cambridge, 2008.

\bibitem[SZ88]{Skoruppa-Zagier}
Nils-Peter Skoruppa and Don Zagier.
\newblock Jacobi forms and a certain space of modular forms.
\newblock {\em Invent. Math.}, 94(1):113--146, 1988.

\bibitem[Wei82]{Weierstrass}
K.~Weierstrass.
\newblock Zur {T}heorie der {J}acobischen {F}unktionen von mehreren
  {V}er{\"a}nderlichen.
\newblock {\em Sitzungsber. {K\"o}nigl. Preuss. Akad. Wiss.}, pages 505--508,
  1882.
\newblock Werke, Band 3, pp. 155--159.

\end{thebibliography}
\bibliographystyle{alpha}

\end{document}